\documentclass[final,hidelinks,onefignum,onetabnum]{siamart1116}

\usepackage{enumitem}
\usepackage{subfig}
\usepackage{amsmath}
\usepackage{bm}
\usepackage{tikz}
\usepackage{appendix}
\usepackage{amsfonts}
\usepackage{multirow}
\usepackage{graphicx}
\usepackage{caption}
\usepackage{float} 

\allowdisplaybreaks[4]   
\usepackage{pgfplots}
\usepackage{booktabs}
\usetikzlibrary{patterns}
\usepackage{xcolor}
\definecolor{darkgreen}{rgb}{0.0, 0.5, 0.0}

\usepackage{algorithm,algorithmicx,algpseudocode}

\algdef{SE}[DOWHILE]{Do}{doWhile}{\algorithmicdo}[1]{\algorithmicwhile\ #1}

\usepackage{geometry}
\usepackage[framemethod=tikz]{mdframed}

\ifpdf
\hypersetup{
	pdftitle={},
	pdfauthor={}
}
\fi

\newtheorem{thm}{Theorem}[section]
\newtheorem{lem}[thm]{Lemma}
\newtheorem{rem}[thm]{Remark}

\newtheorem{Exp}{Example}[section]

\newcommand{\bd}[1]{\boldsymbol{\bold{#1}}}
\newcommand{\bdbar}[1]{\bd{\overline #1}}
\newcommand{\bdtilde}[1]{\bd{\widetilde #1}}
\newcommand{\bdhat}[1]{\bd{\widehat #1}}

\newcommand{\abs}[1]{\left\vert#1\right\vert}

\def\iph{{i+\frac{1}{2}}} 

\def\ipf{{i+\frac{1}{4}}} 

\def\imf{{i-\frac{1}{4}}} 

\def\imh{{i-\frac{1}{2}}} 

\def\ipmh{{i\pm\frac{1}{2}}} 
\def\jpmh{{j\pm\frac{1}{2}}} 
\def\jph{{j+\frac{1}{2}}}

\def\jmh{{j-\frac{1}{2}}}

\def\Hat{\widehat}
\def\Tilde{\widetilde}
\def\Bar{\overline}

\def\ij{{i,j}}
\def\ijph{{\iph,\jph}}
\def\ijmh{{\imh,\jmh}}

\def\dx{{\rm d}x}
\def\dy{{\rm d}y}

\def\DT{\Delta t}
\def\DX{\Delta x}
\def\DY{\Delta y}

\begin{document}
	
	\large

\headers{}{}

\title{Provably positivity-preserving, globally divergence-free central DG methods for ideal MHD system
\thanks{	\normalsize
			The authors were partially supported by Shenzhen Science and Technology Program (Grant No.~RCJC20221008092757098) and National Natural Science Foundation of China (Grant No.~12171227).
R.~Yan was also partially supported by Postdoctoral Fellowship Programs of CPSF (Grant Nos.~2025T180848, 2024M761274, and GZB20240293).}}

\author{	\normalsize Ruifang Yan\thanks{	\normalsize
		Department of Mathematics and Shenzhen International Center for Mathematics, Southern University of Science and Technology, Shenzhen 518055, China (\email{yanrf@sustech.edu.cn}).}
	\and Huihui Cao\thanks{	\normalsize
       Shenzhen International Center for Mathematics and Department of Mathematics, Southern University
of Science and Technology, Shenzhen 518055, China (\email{caohh@sustech.edu.cn}).}
	\and Kailiang Wu\thanks{\normalsize Corresponding author. Department of Mathematics, Shenzhen International Center for Mathematics, and Guangdong Provincial Key Laboratory of Computational Science and Material Design, Southern University of Science and Technology, Shenzhen 518055, China (\email{wukl@sustech.edu.cn}).}}

\maketitle
\begin{abstract}
	\normalsize
The compressible MHD equations possess two essential structural properties: (i) an algebraic structure ensuring the positivity of density and pressure; and (ii) a differential structure maintaining the divergence-free (DF) constraint on the magnetic field. A deep connection between these two properties has been revealed in both non-central [K.~Wu, {\em SIAM J.~Numer.~Anal.}, 56 (2018), pp.~2124--2147] and central discontinuous Galerkin (CDG) frameworks [K.~Wu, H.~Jiang \& C.-W.~Shu, {\em SIAM J.~Numer.~Anal.}, 61 (2023), pp.~250--285]. However, existing methods could provably preserve positivity \textbf{\textit{only}} in conjunction with a \textbf{\textit{locally}} DF property.
Constructing a uniformly high-order method that is both {\em provably positive} and \textbf{\textit{globally DF}} has remained a challenging problem.

This paper proposes a numerical method, termed {\em PosDiv-CDG}, that provably preserves both positivity and the global DF condition at arbitrarily high order in multiple dimensions. It resolves the fundamental structural incompatibility between standard positivity-preserving limiters and global DF enforcement in the CDG framework. The method integrates a novel positivity-limiting strategy, a modified dissipation mechanism guided by convex decomposition, and an auxiliary evolution equation for the magnetic field, which are designed based on rigorous theoretical analysis.
Notably, we provide a rigorous proof of positivity preservation for the updated cell averages under an explicit CFL-type condition. The proof leverages the geometric quasi-linearization (GQL) technique, which reformulates the nonlinear positivity constraint into an equivalent linear form. This enables the derivation of flux-based inequalities and technical estimates under the global DF constraint. To suppress nonphysical oscillations near shocks, we develop a compact, non-intrusive convex-oscillation-suppressing (COS) procedure based on the entropy function. The COS process acts only on non-magnetic variables, avoids costly characteristic decomposition, and maintains both the globally DF property and high-order accuracy. Several challenging experiments---including low plasma-beta MHD jets with Mach numbers up to {\bf 1,000,000}---demonstrate the proposed method robustness, high-order accuracy, non-oscillatory behavior, and its ability to preserve both positivity and globally DF structures under extreme conditions.

\end{abstract}



\section{Introduction}\label{sec:introduction}

More than 99\% of the visible matter in the universe exists in the plasma state. From stars like the Sun to the vast interstellar medium, plasma---an ionized, highly conductive fluid---dominates the cosmos. Its motion induces electric currents that generate magnetic fields, which in turn act back on the plasma through Lorentz forces, leading to rich and highly nonlinear interactions.
The dynamics of an inviscid, perfectly conducting plasma are governed by the equations of ideal compressible  magnetohydrodynamics (MHD). In a $d$-dimensional spatial domain ($d = 1, 2, 3$), the ideal MHD system can be written in conservative form as
\begin{align}\label{eqn:HCL}
	\frac{\partial \mathbf{U}}{\partial t} + \sum_{\ell=1}^d \frac{\partial \mathbf{F}_\ell(\mathbf{U})}{\partial x_\ell} = \mathbf{0}.
\end{align}
Here, $\mathbf{U} = (\rho, \bm{m}, \bm{B}, E)^\top \in \mathbb{R}^8$ is the vector of conservative variables: $\rho$ is the mass density, $\bm{m} = (m_1, m_2, m_3)$ the momentum, $\bm{B} = (B_1, B_2, B_3)$ the magnetic field, and
$
E = \rho e + \frac{1}{2} \rho |\bm{u}|^2 + \frac{1}{2} |\bm{B}|^2
$
the total energy, with $e$ denoting the specific internal energy. The flux in the $x_\ell$-th spatial direction is
\begin{align*}
	\mathbf{F}_\ell(\mathbf{U}) = \big(
	m_\ell,\,
	m_\ell \bm{u} - B_\ell \bm{B} + p_{\text{tot}} \bm{e}_\ell,\,
	u_\ell \bm{B} - B_\ell \bm{u},\,
	u_\ell(E + p_{\text{tot}}) - B_\ell (\bm{u} \cdot \bm{B})
	\big)^\top,
\end{align*}
where $\bm{u} = \bm{m} / \rho$ is the fluid velocity, and $\bm{e}_\ell$ the $\ell$-th Cartesian unit vector. The total pressure is $p_{\text{tot}} = p + p_m$, with $p$ the thermal pressure and $p_m = \frac{1}{2} |\bm{B}|^2$ the magnetic pressure. To close the system, an equation of state (EOS) $p = p(\rho, e)$ is required. We assume that for any $\rho > 0$, $e > 0$ if and only if $p(\rho, e) > 0$.

The ideal MHD system \eqref{eqn:HCL} possesses two fundamental physical structures that are essential for both its mathematical well-posedness and physical validity:
\begin{itemize}[leftmargin=1.5em]
	\item \textbf{Algebraic Structure.} The density and internal energy (or pressure) must remain positive:
	\begin{align}\label{PP}
		\rho > 0, \qquad \rho e = E - \frac{|{\bm m}|^2}{2\rho} - \frac{ |{\bm B}|^2 }{2} > 0,
	\end{align}
	which is essential for the system to be hyperbolic and well-posed.
	\item \textbf{Differential Structure.} The magnetic field should satisfy the divergence-free (DF) constraint:
	\begin{align}\label{DF}
		\nabla \cdot \bm{B} := \sum_{\ell=1}^d \frac{\partial B_\ell}{\partial x_\ell} = 0,
	\end{align}
	which reflects the physical principle of the absence of magnetic monopoles.
\end{itemize}
These properties are preserved by the exact solution of \eqref{eqn:HCL}, provided they hold initially. However, maintaining them in numerical simulations---especially in multiple dimensions ($d \geq 2$)---remains a major challenge, as violations can cause numerical instability, nonphysical results, or even simulation failure \cite{evans1988simulation,balsara1999staggered,toth2000b,dedner2002hyperbolic,li2005locally}.

Over the years, many methods have been proposed to handle the DF condition \eqref{DF}, including but not limited to the projection method \cite{brackbill1980effect}, constrained transport~\cite{evans1988simulation, balsara2004second, arminjon2005central}, eight-wave approach \cite{powell1994approximate, powell1995upwind}, hyperbolic divergence cleaning \cite{dedner2002hyperbolic}, and locally or globally DF methods~\cite{balsara2001divergence, li2005locally, balsara2009divergence, li2011central, li2012arbitrary, xu2016new, fu2018globally, balsara2021globally,liu2025globally}. In contrast, designing positivity-preserving (PP) schemes that respect the algebraic structure has long been a major difficulty. Early efforts focused on building robust one-dimensional Riemann solver~\cite{janhunen2000positive,bouchut2007multiwave,bouchut2010multiwave,klingenberg2010relaxation}, followed by the second-order PP MUSCL--Hancock scheme \cite{waagan2009positive,waagan2011robust}.
While notable efforts have been made to extend high-order PP limiters from compressible Euler equations \cite{zhang2010positivity,xu2014parametrized,xiong2016parametrized} to the ideal MHD setting \cite{balsara2012self,cheng2013positivity,christlieb2015positivity,christlieb2016high}, the positivity of the resulting high-order schemes was mainly justified through numerical experiments with limited theoretical evidence, and a rigorous and complete proof of positivity for high-order MHD schemes was unavailable, particularly in multidimensional cases. In fact, seeking {\em provably}  PP schemes for MHD remained an open problem for years, even for first-order schemes. As explicitly noted in \cite{christlieb2015positivity}, ``there is still no rigorous proof that the Lax--Friedrichs scheme or any other first-order scheme is PP in the multi-D case...''.

An interesting and somewhat surprising discovery is that two seemingly distinct structures---the positivity constraint (a pointwise algebraic condition) and the DF condition (a differential constraint)---are intrinsically connected \cite{WuSINUM2018,WuTangM3AS,wu2023provably}. A central challenge in designing structure-preserving MHD schemes lies in the limited understanding of this algebraic-differential coupling. Bridging this gap was highly nontrivial. However, theoretical advances in \cite{WuSINUM2018} have begun to reveal a deep and consequential interplay between the two, thanks to the geometric quasi-linearization (GQL) approach \cite{WuTangM3AS,WuSINUM2018,wu2023geometric}, which equivalently transforms nonlinear algebraic constraints into linear ones. In particular, it was rigorously proven in \cite{WuSINUM2018} that for conservative MHD schemes, the PP property depends critically on satisfying a discrete DF condition \cite{WuSINUM2018}. Even minor violations of this condition can lead to a loss of pressure positivity---not only in numerical solutions \cite{WuSINUM2018}, but even in smooth exact solutions of the continuous system \cite{wu2018provably}. This highlights a strong and nonlinear connection between these two structural properties and suggests that the PP property can only be achieved in a DF framework.

While several existing methods~\cite{balsara2001divergence,balsara2009divergence,li2011central,li2012arbitrary,xu2016new,fu2018globally,balsara2021globally,liu2025globally} aim to enforce a globally DF magnetic field,\footnote{A numerical magnetic field is said to be globally DF if it is locally DF within each cell and its normal component is continuous across cell interfaces.} they often conflict with traditional PP limiters. The local scaling nature of such limiters can disrupt the globally DF property. This incompatibility has, to date, prevented the construction of a scheme that simultaneously satisfies both properties. To our knowledge, \emph{no existing high-order numerical method is both provably PP and globally DF for ideal MHD}.
Recent research has instead focused on developing high-order schemes that preserve PP while enforcing only a \textbf{locally or discretely} DF condition~\cite{wu2018provably,wu2019provably,wu2023provably,ding2023new,liu2025structure}. A key observation in~\cite{wu2019provably} is that the symmetrizable MHD formulation, which incorporates the Godunov--Powell source terms, admits exact solutions that remain inherently PP, even when the DF condition is not strictly satisfied. Building on this insight, it has been shown that carefully designed discretizations of the Godunov--Powell formulation can yield arbitrarily high-order, provably PP schemes that rely only on a locally DF condition---naturally compatible with local PP limiters~\cite{wu2018provably,wu2019provably,ding2023new}. This opens a promising path toward the construction of provably PP and locally DF MHD schemes.
Yet, the key open question remains:
\begin{quote}
	\emph{Can one design a uniformly high-order scheme that is both provably PP and \textbf{globally} DF?}
\end{quote}

This work makes the first rigorous attempt to resolve this fundamental problem within the central discontinuous Galerkin (CDG) framework, by developing an arbitrarily high-order structure-preserving CDG method, referred to as the {\em PosDiv-CDG method}. The CDG framework~\cite{liu2007central}, a variant of the classical DG method~\cite{cockburn1989tvb2,cockburn1998runge}, retains essential features such as high-order accuracy, compactness, and parallel efficiency. By utilizing overlapping primal-dual meshes, CDG methods eliminate the need for Riemann solvers and allow for larger time steps~\cite{liu20082,reyna2015operator}. These features have facilitated the development of structure-preserving CDG methods for scalar conservation laws~\cite{li2016maximum}, relativistic hydrodynamics~\cite{wu2016physical}, shallow water equations~\cite{li2017positivity}, and Euler equations with gravitation~\cite{jiang2022positivity}. Recent advances~\cite{wu2023provably,ding2024gql} have uncovered a fundamental connection between the PP property of CDG methods and a discrete DF condition on the dual mesh, enabling the construction of high-order CDG schemes that are provably PP and locally DF. Another important development is due to Li and Xu~\cite{li2011central,li2012arbitrary}, who systematically constructed globally DF CDG methods of arbitrary order by discretizing the normal components of the magnetic field on cell interfaces and reconstructing globally DF fields within each cell. However, ensuring positivity preservation in such a globally DF framework has remained a major unresolved challenge. This work overcomes the challenge through a series of algorithmic and theoretical innovations, summarized as follows:

\begin{itemize}[leftmargin=*]

	\item \textbf{Provably PP and globally DF scheme:} We present the first numerical method that simultaneously ensures provably PP and globally DF properties for multidimensional compressible MHD at arbitrarily high order. This resolves the intrinsic incompatibility between traditional PP limiters and globally DF magnetic fields. The PosDiv-CDG method integrates a divergence-compatible PP limiter, an entropy-based convex-oscillation-suppressing (COS) procedure, modified dissipation guided by cell average decomposition, and auxiliary magnetic field evolution, which are developed through rigorous PP analysis.
	
	\item \textbf{Divergence-compatible PP limiter:} A new PP limiting strategy is introduced, which modifies conservative variables while preserving the globally DF structure of the magnetic field. This is achieved by transforming the limited point values into primitive variables and subsequently restoring the globally DF magnetic components---thus resolving the structural conflict between positivity and globally DF preservation.

	\item \textbf{Entropy-based oscillation elimination:} To suppress spurious oscillations while preserving the globally DF structure, we develop a compact, non-intrusive COS procedure based on the entropy function. The COS technique acts only on non-magnetic variables, avoids expensive characteristic decomposition, and preserves the globally DF property, high-order accuracy, and high resolution.
	
	\item \textbf{Modified dissipation for cell average positivity:} Since PP-limited point values generally do not correspond to a single polynomial, we introduce a careful modification of the CDG dissipation terms associated with the zeroth moment. By decomposing cell averages and utilizing limited point values, we preserve high-order accuracy while ensuring provable positivity of the updated cell averages.
	
	\item \textbf{Auxiliary magnetic field averages:} To address the difficulty that the globally DF reconstructed magnetic field may not ensure admissible cell average of the full conservative state, we evolve auxiliary cell averages of the magnetic field using modified CDG updates. These auxiliary quantities allow rigorous control over the full state vector and guarantee that the evolved averages remain within the admissible set.
	
	\item \textbf{Rigorous positivity proof:} We provide a highly nontrivial, rigorous proof that the method preserves positivity of updated cell averages under a CFL condition. The proof employs the geometric quasi-linearization (GQL) framework~\cite{WuSINUM2018,wu2023geometric}, which reformulates the nonlinear positivity constraint into an equivalent linear form, enabling the use of flux inequalities and technical estimates under the global DF condition.
	
	\item \textbf{Robust validation under extreme conditions:} Several challenging examples, including MHD jets with Mach numbers reaching up to \textbf{1,000,000},  demonstrate the method’s high-order accuracy, non-oscillatory behavior, robustness, and ability to preserve both positivity and DF structures under extreme, shock-dominated, highly magnetized conditions.
\end{itemize}

The remainder of this paper is organized as follows. In \Cref{sec:GDF-CDG}, we briefly review the globally DF CDG method for the 2D ideal MHD equations~\eqref{eqn:HCL}, originally proposed by Li and Xu~\cite{li2011central,li2012arbitrary} . \Cref{sec:PPGDF-CDG} details the challenges of achieving provable positivity within the globally DF CDG framework, and presents the proposed PosDiv-CDG scheme, which incorporates the entropy-induced COS procedure, the PP limiting procedure compatible with the global DF conditions, the cell average decomposition with modified dissipation terms, the evolution of auxiliary magnetic field averages, and the rigorous PP analysis. Numerical experiments demonstrating the accuracy and robustness of the proposed scheme are presented in \Cref{sec:exp}, followed by concluding remarks in \Cref{sec:con}.

\section{Brief review of globally DF CDG method}\label{sec:GDF-CDG}

This section reviews the globally DF CDG method in \cite{li2011central,li2012arbitrary} for the two-dimensional (2D) ideal MHD equations.

The conservative vector $\bd U$ is split into two parts:
$
	\bd R = (\rho,m_1,m_2,m_3,B_3,E)^\top$ and
	$\bd Q = (B_1, B_2)^\top$.
Accordingly, the flux vector $\bd F_\ell$ is decomposed into $\bd F_\ell^{\bd R}(\bd U)$ and $\bd F_\ell^{\bd Q}(\bd U)$. Consequently, the 2D MHD equations \eqref{eqn:HCL} can be rewritten as
\begin{align}
	& \frac{\partial\bd R}{\partial t} + \frac{\partial\bd F_1^{\bd R}(\bd U)}{\partial x}
	+ \frac{\partial\bd F_2^{\bd R}(\bd U)}{\partial y} = {\bd 0}, \label{eq:R} \\
	& \frac{\partial\bd Q}{\partial t} + \frac{\partial\bd F_1^{\bd Q}(\bd U)}{\partial x}
	+ \frac{\partial\bd F_2^{\bd Q}(\bd U)}{\partial y} = {\bd 0}. \label{eq:Q}
\end{align}

Assume that the computational domain $\Omega = [x_{\text{min}},x_{\text{max}}] \times [y_{\text{min}},y_{\text{max}}]$ is covered by two overlapping Cartesian meshes: the primal mesh $\mathcal{K}_h^C = \{ C_{i,j} \}$ and the dual mesh $\mathcal{K}_h^D = \{ D_{i+\frac{1}{2},j+\frac{1}{2}} \}$. The primal and dual cells are defined by
$
C_{i,j} = [x_{i-\frac{1}{2}},x_{i+\frac{1}{2}}] \times [y_{j-\frac{1}{2}},y_{j+\frac{1}{2}}]$ and
$D_{i+\frac{1}{2},j+\frac{1}{2}} = [x_i,x_{i+1}] \times [y_j,y_{j+1}].
$
The mesh sizes $\Delta x$ and $\Delta y$ are assumed to be uniform in the $x$- and $y$-directions, respectively.
The CDG method seeks two approximate solutions, denoted by $\bd U_{h}^\sharp(x,y,t)$, where $\sharp$ stands for either $C$ or $D$, defined on the primal mesh $\mathcal{K}_h^C$ and the dual mesh $\mathcal{K}_h^D$, respectively.
Accordingly, $\bd R_{h}^{\sharp}(x,y,t)$ and $\bd Q_{h}^{\sharp}(x,y,t)$ denote the CDG approximations to $\bd R$ and $\bd Q$.
For simplicity, we focus on the forward Euler method for time discretization, although high-order strong-stability-preserving (SSP) methods can also be used.

\subsection{Standard CDG discretization for $\bold R$}

Since equations \eqref{eq:R} do not involve the DF constraint \eqref{DF}, they can be discretized using
the standard CDG method \cite{liu2007central,li2011central}.
The finite element space for approximating $\bd R$ is defined by
$
	\mathbb{V}_h^{\sharp,k}
	:= \left\{ {\bm w} \,:\, {\bm w}|_K \in [\mathbb{P}^k(K)]^6 \quad \forall K \in \mathcal{K}_h^\sharp \right\},
$
where $\mathbb{P}^k(K)$ denotes the space of polynomials of total degree at most $k$ on cell $K$.

Given the CDG solutions $\bd U_{h}^{C,n}$ and $\bd U_{h}^{D,n}$ at time $t = t^n$, the updated solutions $\bd R_{h}^{C,n+1}$ and $\bd R_{h}^{D,n+1}$ at $t^{n+1} = t^n + \Delta t$ are computed as follows:
\begin{equation}\label{eqn:cdg-c}
		\resizebox{0.91\hsize}{!}{$
	 \begin{aligned}
		& \int_{C_\ij} \bd R_{h}^{C,n+1} {\phi} {\rm d}x{\rm d}y
		= \int_{C_\ij} \bd R_{h}^{C,n}  {\phi} {\rm d}x{\rm d}y
		+ \theta    {\bm d}_{ij} (\phi )
		+ \DT\int_{C_\ij} \bd F^{\bd R}(\bd U_{h}^{D,n})\cdot\nabla{\phi} {\rm d}x{\rm d}y  \\
		&  - \DT\int_{y_{\jmh}}^{y_{\jph}}
		\left( \bd F_1^{\bd R}\big(\bd U_{h}^{D,n}(x_\iph,y)\big)  {\phi}(x_\iph^-,y)
		- \bd F_1^{\bd R}\big(\bd U_{h}^{D,n}(x_\imh,y)\big) {\phi}(x_\imh^+,y)\right){\rm d}y  \\
		& - \DT\int_{x_{\imh}}^{x_{\iph}}
		\left( \bd F_2^{\bd R}\big(\bd U_{h}^{D,n}(x,y_\jph)\big)  {\phi}(x,y_\jph^-)
		- \bd F_2^{\bd R}\big(\bd U_{h}^{D,n}(x,y_\jmh)\big)  {\phi}(x,y_\jmh^+)\right){\rm d}x,
	\end{aligned}
$}
\end{equation}
	\begin{equation}\label{eqn:cdg-d}
		\begin{aligned}
     &\int_{D_\ijph} \bd R_{h}^{D,n+1}  {\phi} {\rm d}x{\rm d}y = \int_{D_\ijph} \bd R_{h}^{D,n}  {\phi} {\rm d}x{\rm d}y
     = \int_{D_\ijph} (\theta\bd R_{h}^{C,n}+(1-\theta)\bd R_{h}^{D,n})  {\phi} {\rm d}x{\rm d}y
    \\
     & \quad\quad + \theta    {\bm d}_{i+\frac12,j+\frac12} (\phi ) + \DT\int_{D_\ijph} \bd F^{\bd R}(\bd U_{h}^{C,n})\cdot\nabla{\phi} {\rm d}x{\rm d}y  \\
    & \quad\quad - \DT\int_{y_{j}}^{y_{j+1}}
    \left( \bd F_1^{\bd R}\big(\bd U_{h}^{C,n}(x_{i+1},y)\big)  {\phi}(x_{i+1}^-,y)
         - \bd F_1^{\bd R}\big(\bd U_{h}^{C,n}(x_i,y)\big)  {\phi}(x_i^+,y)\right){\rm d}y  \\
    & \quad\quad - \DT\int_{x_{i}}^{x_{i+1}}
    \left( \bd F_2^{\bd R}\big(\bd U_{h}^{C,n}(x,y_{j+1})\big)  {\phi}(x,y_{j+1}^-)
         - \bd F_2^{\bd R}\big(\bd U_{h}^{C,n}(x,y_j)\big)  {\phi}(x,y_j^+)\right){\rm d}x
\end{aligned}
\end{equation}
for all $\phi \in \mathbb P^k$.
Here, $\bd F^{\bd R} = ( \bd F^{\bd R}_1, \bd F^{\bd R}_2 )$,
$\theta = \frac{\DT}{\tau_{\text{max}}}$ with $\tau_{\text{max}}$ being the maximum time step allowed by the CFL condition,
and the CDG dissipation terms are given by
\begin{align}\label{eq:dissipationC}
	{\bm d}_{i,j} (\phi) &:=  \int_{C_\ij} (\bd R_{h}^{D,n}-\bd R_{h}^{C,n})  {\phi} {\rm d}x{\rm d}y,
	\\
	\label{eq:dissipationD}
	{\bm d}_{i+\frac12,j+\frac12} (\phi )  &:=  \int_{D_\ijph} (\bd R_{h}^{C,n}-\bd R_{h}^{D,n})  {\phi} {\rm d}x{\rm d}y.
\end{align}

\subsection{Globally DF CDG discretization for ${\bf Q}$}\label{subsec:GDF-Q}

The induction equations \eqref{eq:Q} for ${\bf Q}$ can be written as
\begin{align}\label{eqn:B1}
	\frac{\partial B_1}{\partial t} + \frac{\partial G(\bd U)}{\partial y} = 0, \qquad
	\frac{\partial B_2}{\partial t} - \frac{\partial G(\bd U)}{\partial x} = 0,
\end{align}
where $G(\bd U) = B_1 u_2 - B_2 u_1$ is the third component of the electric field $- {\bm u} \times {\bm B}$.
To exactly preserve the DF constraint \eqref{DF}, Li and Xu \cite{li2011central,li2012arbitrary} proposed a globally DF discretization of \eqref{eqn:B1} within the CDG framework. The key idea is to first approximate the normal component of the magnetic field by solving \eqref{eq:Q} on cell interfaces, and then reconstruct globally DF magnetic fields ${\bf Q}_h^C$ and ${\bf Q}_h^D$ by matching those normal components across interfaces.

Given the CDG solutions $\bd U_{h}^{C,n}$ and $\bd U_{h}^{D,n}$ at time $t = t^n$, the globally DF field ${\bf Q}_h^{C,n+1}$ at $t^{n+1} = t^n + \Delta t$ is constructed as follows. First, the normal component $b_{\imh,j}^{x}(y)$ on the vertical edge $x = x_\imh$ of cell $C_{i,j}$ is approximated by
\begin{align}\label{eqn:bx-ave}
	 \int_{y_\jmh}^{y_\jph} b_{\imh,j}^x(y)\, \varphi(y)\, {\rm d}y
	= \int_{y_\jmh}^{y_\jph} \left( \theta (B_1)_h^{D,n} + (1 - \theta)(B_1)_h^{C,n} \right)(x_\imh, y) \varphi(y)\, {\rm d}y & \notag \\
	\qquad + \DT \left[ -G(\bd U_h^{D,n}(x_\imh, y_\jph)) \varphi(y_\jph) + G(\bd U_h^{D,n}(x_\imh, y_\jmh)) \varphi(y_\jmh) \right] & \notag \\
	\qquad + \DT \int_{y_\jmh}^{y_\jph} G(\bd U_h^{D,n}(x_\imh, y))\, \frac{\partial \varphi(y)}{\partial y}\, {\rm d}y,  \qquad \varphi \in \mathbb{P}^k((y_\jmh, y_\jph)).&
\end{align}
Similarly, the normal component $b_{i,\jmh}^y(x)$ on the horizontal edge $y = y_\jmh$ of $C_{i,j}$ is approximated by
\begin{align}\label{eqn:by-ave}
	\int_{x_\imh}^{x_\iph} b_{i,\jmh}^y(x)\, \varphi(x)\, {\rm d}x
	= \int_{x_\imh}^{x_\iph} \left( \theta (B_2)_h^{D,n} + (1 - \theta)(B_2)_h^{C,n} \right)(x, y_\jmh) \varphi(x)\, {\rm d}x &  \notag \\
	\qquad  + \DT \left[ G(\bd U_h^{D,n}(x_\iph, y_\jmh)) \varphi(x_\iph) - G(\bd U_h^{D,n}(x_\imh, y_\jmh)) \varphi(x_\imh) \right] & \notag \\
	\qquad - \DT \int_{x_\imh}^{x_\iph} G(\bd U_h^{D,n}(x, y_\jmh))\, \frac{\partial \varphi(x)}{\partial x}\, {\rm d}x, \qquad \varphi \in \mathbb{P}^k((x_\imh, x_\iph)).&
\end{align}
As shown in \cite{li2011central}, the following compatibility condition is automatically satisfied:
\begin{align}\label{eqn:B-integra}
	\int_{x_\imh}^{x_\iph} \Big(b_{i,\jph}^y(x) - b_{i,\jmh}^y(x)\Big) {\rm d}x
	+ \int_{y_\jmh}^{y_\jph} \Big(b_{\iph,j}^x(y) - b_{\imh,j}^x(y)\Big) {\rm d}y = 0.
\end{align}
Thanks to this condition,
the globally DF field ${\bf Q}_{i,j}^{C,n+1}=\big((B_{1})_{i,j}^{C,n+1},(B_{2})_{i,j}^{C,n+1}\big)^\top$ can then be reconstructed by matching the normal components   $b_{\ipmh,j}^x(y)$ and $b_{i,\jpmh}^y(x)$ along the cell interfaces and enforcing the locally DF property within $C_{i,j}$.
The reconstruction techniques have been well studied in the literature; see, e.g., \cite{balsara2001divergence, balsara2004second, balsara2009divergence,li2011central, li2012arbitrary}. The reconstruction of ${\bf Q}_{i+\frac{1}{2},j+\frac{1}{2}}^{D,n+1}$ proceeds analogously. For further details, refer to \cite{li2011central, li2012arbitrary}.

\section{PosDiv-CDG: Provably PP and globally DF CDG method}\label{sec:PPGDF-CDG}

This section presents the PosDiv-CDG method for 2D ideal MHD equations. The extension to three dimensions is straightforward.
Without loss of generality, we begin with the forward Euler method for time discretization, noting that all discussions naturally extend to high-order SSP Runge--Kutta (RK) or multi-step methods, as discussed in \Cref{subsec:RK}.

\subsection{Challenges of achieving provable positivity in the globally DF CDG framework}\label{subsec:challenges}

A scheme is said to be PP if it maintains the numerical solution within the admissible state set
\begin{align}\label{ASS}
	\mathcal{G} = \left\{\bd U=(\rho,{\bm m},{\bm B},E)^\top:~~\rho>0,\quad \mathcal{E}(\bd U):=E-\frac12\bigg(\frac{\abs{{\bm m}}^2}{\rho}+\abs{{\bm B}}^2\bigg)>0\right\},
\end{align}
where $\mathcal{E}(\bd U) = \rho e$ denotes the internal energy.

Following the Zhang--Shu framework \cite{zhang2010positivity,zhang2010maximum,zhang2017positivity}, ensuring the provable PP property in a DG-type scheme (either central or non-central) involves two fundamental tasks:
\begin{itemize}
	\item[(i)] ensuring that point values at certain quadrature nodes lie in $\mathcal{G}$;
	\item[(ii)] ensuring that the updated cell averages remain in $\mathcal{G}$ at the next time level.
\end{itemize}
Achieving these within the globally DF framework presents several notable challenges:

\vspace{0.3em}
\noindent
\textbf{Challenge 1: Incompatibility between PP limiter and globally DF condition.}
In the Zhang--Shu framework \cite{zhang2010positivity,zhang2010maximum,zhang2017positivity}, task (i) is typically handled using a local scaling PP limiter, assuming that the cell averages are admissible at the current time step. However, due to its local nature, such a limiter may disrupt the continuity of the normal component of the magnetic field across cell interfaces. As a result, the globally DF condition may be violated, even if it initially holds. This reveals a fundamental incompatibility between standard PP limiters and the globally DF requirement.

\vspace{0.3em}
\noindent
\textbf{Challenge 2: Discrete DF condition required for positive cell averages.}
	Task (ii)---preserving the admissibility of updated cell averages---is more subtle for compressible MHD. As shown in \cite{WuSINUM2018,wu2023provably}, for conservative schemes, the positivity of pointwise values at time $t^n$ alone is insufficient to ensure that the updated cell averages at $t^{n+1}$ remain in $\mathcal{G}$. In particular, a discrete DF  condition is also necessary \cite{wu2023provably}.
This condition is globally coupled and automatically satisfied if the magnetic field is globally DF. However, as noted above, the global or discrete DF condition may be destroyed by the local PP limiter, making it difficult to simultaneously enforce pointwise positivity and the DF constraint.

\vspace{0.3em}
\noindent
\textbf{Challenge 3: Post-processing to restore globally DF may destroy positivity.}
	A natural question arises:
	\begin{quote}
		\emph{Can one first apply a PP limiter and then adjust the magnetic field to restore the globally DF condition?}
	\end{quote}
While this approach may seem promising, it introduces a subtle yet critical issue. Adjusting the magnetic field to restore the DF property may increase the magnetic energy, which---along with kinetic energy---is subtracted from the total energy to compute internal energy. This can result in negative internal energy (pressure), thereby violating the pointwise positivity that was just enforced. This exposes a structural incompatibility between standard PP limiting and the globally DF condition.

\vspace{0.3em}
\noindent
\textbf{Challenge 4: Difficulty in ensuring positivity of reconstructed cell averages.}
Unlike the standard CDG method, the globally DF CDG schemes described in \Cref{subsec:GDF-Q} reconstruct the magnetic field cell averages from the normal magnetic components on cell edges. These reconstructed values are then combined with the cell averages of the other conservative variables $\bd R$ to form the full cell average of $\bd U$. However, due to the distinct evolution mechanisms of $\bd Q$ and $\bd R$, it is highly nontrivial to verify whether the resulting cell average of $\bd U$ lies within $\mathcal{G}$. This adds another layer of difficulty in establishing the PP property within the globally DF CDG framework.

\subsection{Outline of the PosDiv-CDG Method}\label{subsec:outline}

This subsection outlines the proposed PosDiv-CDG method, which addresses the challenges discussed above.
To ensure compatibility with the globally DF constraint, a modification to the PP limiting procedure is introduced. However, this modification disrupts the direct correspondence between PP-limited point values and a single polynomial representation, thereby invalidating the cell average decomposition (see \Cref{subsec:CAD})---an essential component in establishing the PP property of the updated cell averages.
To resolve this issue, we carefully redesign the dissipation term in the CDG scheme, guided by the theoretical PP analysis presented in \Cref{subsec:PP-property}.
To address Challenge 4, we further introduce an auxiliary cell average for the magnetic field on each cell, denoted by $\overline{\bf Q}_{i,j}^{\star,C,n}$ and $\overline{\bf Q}_{i+\frac12,j+\frac12}^{\star,D,n}$, which are evolved using the standard CDG formulation for cell averages.

Let the indices $(\hat{\imath}, \hat{\jmath})$ represent $(i,j)$ for the primal cell $C_{i,j}$ or $(i+\frac{1}{2}, j+\frac{1}{2})$ for the dual cell $D_{i+\frac{1}{2},j+\frac{1}{2}}$.
We initialize the normal component of the magnetic field on cell edges and reconstruct the globally DF magnetic field $\bd Q_{\hat{\imath}, \hat{\jmath}}^{\sharp,0}(x,y)$ at $t^0 = 0$.
The CDG solution $\bd R_{\hat{\imath}, \hat{\jmath}}^{\sharp,0}(x,y)$ and the auxiliary cell average $\overline{\bf Q}_{\hat{\imath}, \hat{\jmath}}^{\star,\sharp,0}$ are initialized via standard $L^2$ projection.
Let $\overline{\bd R}_{\hat{\imath}, \hat{\jmath}}^{\star,\sharp,0}$ denote the cell average of $\bd R_{\hat{\imath}, \hat{\jmath}}^{\sharp,0}(x,y)$.
By the convexity of $\mathcal{G}$, the full cell average $\overline{\bd U}_{\hat{\imath}, \hat{\jmath}}^{\star,\sharp,0}$, formed by $\overline{\bd R}_{\hat{\imath}, \hat{\jmath}}^{\star,\sharp,0}$ and $\overline{\bf Q}_{\hat{\imath}, \hat{\jmath}}^{\star,\sharp,0}$, stays within $\mathcal{G}$.

The steps of the PosDiv-CDG method from $t^n$ to $t^{n+1}$ are as follows:

\begin{description}[leftmargin=*]
	
	\item[\textbf{Step 1.}]
	Given the globally DF CDG solution $\bd U_{h}^{\sharp,n} = \{ \bd R_{\hat{\imath}, \hat{\jmath}}^{\sharp,n}(x,y), \bd Q_{\hat{\imath}, \hat{\jmath}}^{\sharp,n}(x,y) \}$
	and the auxiliary magnetic field average $\overline{\bf Q}_{\hat{\imath}, \hat{\jmath}}^{\star,\sharp,n}$ at $t^n$,
	assume the full cell average $\overline{\bf U}_{\hat{\imath}, \hat{\jmath}}^{\star,\sharp,n} \in \mathcal{G}$.
	Evaluate the entropy Hessian matrix $\nabla^2 g( \overline{\bf U}_{\hat{\imath}, \hat{\jmath}}^{\star,\sharp,n} )$,
	where $g({\bf U}) = -\rho \log(p \rho^{-\gamma})$ is the entropy of the MHD system (see \Cref{subsec:CEO} for details).
	Note that $g(\bf U)$ is convex for ${\bf U} \in {\mathcal G}$ and its Hessian is positive-definite.

	\item[\textbf{Step 2.}]
	Apply the entropy-induced convex-oscillation-suppressing (COS) procedure to suppress spurious oscillations and correct high-order moments of $\bd R$.
	See \Cref{subsec:CEO} for details of the COS procedure.
	
	\item[\textbf{Step 3.}]
	Evaluate $\bd U_{\hat{\imath}, \hat{\jmath}}^{\sharp,n}(x,y)$ at all quadrature points used in the flux integrals of \eqref{eqn:cdg-c}--\eqref{eqn:by-ave} and in the cell average decomposition \eqref{ocad} (see \Cref{Fig:CAD} for the illustration of these points in $C_{i,j}$).
	Denote the set of point values as $\{ \bd U_{\hat{\imath}, \hat{\jmath}}^{(q),\sharp} \}_{q=1}^Q$.
	
	\item[\textbf{Step 4.}]  Given $\overline{\bf U}_{\hat{\imath}, \hat{\jmath}}^{\star,\sharp,n} \in \mathcal{G}$,
	apply a PP limiter using the admissible cell average to modify the point values $\{ \bd U^{(q),\sharp}_{\hat{\imath}, \hat{\jmath}} \}$ such that the resulting limited values $\bdhat U^{(q),\sharp}_{\hat{\imath}, \hat{\jmath}} \in \mathcal{G}$ for all $q$.
	Let $\{ \bdhat V^{(q),\sharp}_{\hat{\imath}, \hat{\jmath}} \}$ denote the corresponding primitive variables.
	Replace the magnetic component $\bdhat Q^{(q),\sharp}_{\hat{\imath}, \hat{\jmath}}$ in $\bdhat V^{(q),\sharp}_{\hat{\imath}, \hat{\jmath}}$ with the original globally DF  field $\bd Q^{(q),\sharp}_{\hat{\imath}, \hat{\jmath}}$,
	resulting in new primitive variables $\bdtilde V^{(q),\sharp}_{\hat{\imath}, \hat{\jmath}}$ that are both PP and globally DF.
	Convert these to conservative variables $\bdtilde U^{(q),\sharp}_{\hat{\imath}, \hat{\jmath}}$; see \Cref{subsec:PP} for details of the full PP limiting procedure.
	
	\item[\textbf{Step 5.}]
	Use the modified point values $ \{ \bdtilde U^{(q),\sharp}_{\hat{\imath}, \hat{\jmath}}  \}_{q=1}^Q$ to evaluate the integrals of flux functions ${\bf F}^{\bf R} ({\bf U})$
	in \eqref{eqn:cdg-c}--\eqref{eqn:cdg-d}, the electric field $G(\bd U)$ in \eqref{eqn:bx-ave}--\eqref{eqn:by-ave} and \eqref{eqn:cdg-c-bx-ave}--\eqref{eqn:cdg-c-by-ave},
	and to modify the dissipation terms in the CDG scheme for $\bf{R}$ as detailed in \Cref{subsec:CAD}.
	
	\item[\textbf{Step 6.}]
	Update the CDG solution $\bd R_{\hat{\imath}, \hat{\jmath}}^{\sharp,n+1}$ using \eqref{eqn:cdg-c}--\eqref{eqn:cdg-d} with the fluxes and modified dissipation terms from Step 5.
	Update the globally DF magnetic field $\bd Q_{\hat{\imath}, \hat{\jmath}}^{\sharp,n+1}$ using the CDG schemes in \Cref{subsec:GDF-Q}, with $G(\bd U)$ computed from the modified values.
	Update the auxiliary average $\overline{\bf Q}_{\hat{\imath}, \hat{\jmath}}^{\star,\sharp,n+1}$ via \eqref{eqn:cdg-c-bx-ave}--\eqref{eqn:cdg-c-by-ave}.
	The resulting high-order scheme ensures that the full cell average $\overline{\bd U}_{\hat{\imath}, \hat{\jmath}}^{\star,\sharp,n+1}$, formed by $\overline{\bd R}_{\hat{\imath}, \hat{\jmath}}^{\star,\sharp,n+1}$ and $\overline{\bf Q}_{\hat{\imath}, \hat{\jmath}}^{\star,\sharp,n+1}$, remains in $\mathcal{G}$, as shown in \Cref{Thm:PP}.
	This positivity of the cell averages allows the PP limiter to be applied again at $t^{n+1}$ to maintain pointwise positivity.
	
\end{description}

\subsection{Entropy-induced convex oscillation suppression} \label{subsec:CEO}

Conventional CDG schemes can produce spurious oscillations near discontinuities. To effectively suppress these oscillations, we propose a novel entropy-induced convex oscillation suppression (COS) procedure. The design of this procedure is inspired by our recent work \cite{cao2025COSDG}, which was proven to preserve entropy stability.
Unlike traditional WENO-type limiters, the proposed COS method preserves the compactness of the CDG framework, relying only on neighboring cell information.
Compared to many other limiters, the key advantages of COS include: avoiding costly characteristic decomposition (which is particularly expensive for MHD equations); acting only on the variable $\mathbf{R}$, thereby preserving the globally DF property; maintaining high-order accuracy without relying on any shock detector...

\subsection{Pointwise positivity with globally DF property}\label{subsec:PP}

 To establish a provably PP scheme, we aim to maintain the positivity of density and pressure not only in the cell averages but also at selected key points within each cell. The PP property of updated cell averages will be proven in \Cref{subsec:PP-property}. Assuming admissible cell averages $\overline{\bf U}_{\hat{\imath}, \hat{\jmath}}^{\star,\sharp,n} \in \mathcal{G}$ at time level $n$, this subsection introduces a new limiting strategy that enforces pointwise positivity while preserving the globally DF property of the magnetic field.

We employ standard Gauss quadrature with $N = k + 1$ nodes for the integrals in \eqref{eqn:cdg-c} and \eqref{eqn:cdg-d} for the $\mathbb{P}^k$-based CDG method. Let $\{ x_{\imf}^{(\mu)} \}_{\mu=1}^{N}$ and $\{ x_{\ipf}^{(\mu)} \}_{\mu=1}^{N}$ be the $N$-point Gauss nodes mapped to intervals $[x_\imh, x_i]$ and $[x_i, x_\iph]$, respectively, and define $\mathbb{Q}_i^x := \bigcup_{\sigma = \pm1} \{ x_{i+\frac{\sigma}{4}}^{(\mu)} \}_{\mu=1}^{N}$. The associated weights $\{\omega_\mu\}$ are normalized over the reference interval $[-\tfrac12, \tfrac12]$. Define $\mathbb{Q}_j^y:= \bigcup_{\sigma = \pm1} \{ y_{j+\frac{\sigma}{4}}^{(\mu)} \}_{\mu=1}^{N}$ analogously in the $y$-direction.
To guarantee positivity at these quadrature nodes, we require
\begin{equation*}\label{eq:S1}
	\bd U^{D,n}_h (x,y) \in \mathcal{G}, \qquad \forall (x,y) \in  (\mathbb{Q}_i^x \otimes \mathbb{Q}_j^y) \cup \mathbb{Q}^{v}_\ij ,
\end{equation*}
where $\mathbb{Q}^{v}_\ij = \bigcup_{\sigma = \pm 1} \{(x_{i\pm\frac14+\frac{\sigma}{4}}, y_{j\pm\frac14+\frac{\sigma}{4}})\}$ is also included to ensure that $G( {\bf U}^{D,n}_h(x,y) )$ is well-defined at the vertices of $C_{i,j}$ in \eqref{eqn:bx-ave} and \eqref{eqn:by-ave}.
In addition, to guarantee the PP property of the updated cell averages at the next time level $t^{n+1}$ (as shown in \Cref{Thm:PP} in \Cref{subsec:PP-property}), we impose
\begin{equation*}
	\bd U^{D,n}_h(x,y) \in \mathcal{G}, \qquad \forall (x,y) \in \mathbb{Q}_\ij^{\rm CAD},
\end{equation*}
where the set $\mathbb{Q}_\ij^{\rm CAD}$ will be defined in \Cref{subsec:CAD}.
Combining the above, we require
\begin{equation}\label{eq:SD}
\bd U^{D,n}_h(x,y) \in {\mathcal G}, \qquad \forall (x,y) \in 	\mathbb{Q}_\ij, \quad \forall i,j,
\end{equation}
with $\mathbb{Q}_\ij = \mathbb{Q}_\ij^{\rm CAD} \bigcup (\mathbb{Q}_i^x \otimes \mathbb{Q}_j^y) \bigcup \mathbb{Q}^{v}_\ij$.  An illustration of the point set $\mathbb{Q}_{i,j}$ is provided in \Cref{Fig:CAD}.
Define
$$
\mathbb{Q}_\ijph := \bigcup\limits_{\sigma=\pm1}\bigcup\limits_{\sigma'=\pm1} \Big\{ \mathbb{Q}_{i+\frac12+\frac{\sigma}2, j+\frac12+\frac{\sigma'}2} \bigcap D_\ijph \Big\}.
$$
Since $\bigcup_{i,j} \mathbb{Q}_\ijph = \bigcup_{i,j} \mathbb{Q}_\ij$, the requirement \eqref{eq:SD} is equivalently expressed as
\begin{equation}\label{eq:SD2}
	\bd U^{D,n}_h(x,y) \in \mathcal{G}, \qquad \forall (x,y) \in \mathbb{Q}_\ijph, \quad \forall i,j.
\end{equation}

Similarly, we require the same pointwise positivity for $\bd U^{C,n}_h$:
$$
\bd U^{C,n}_h(x,y) \in {\mathcal G}, \qquad \forall (x,y) \in 	\mathbb{Q}_\ij, \quad \forall i,j.
$$

We now introduce a PP limiting procedure that ensures pointwise positivity while preserving the globally DF property.
Without loss of generality, we only describe the procedure for $\bd U^{D,n}_h$.
For convenience, denote the original (unlimited) point values as
\begin{equation}\label{eq:333}
	 \{ \bd U_\ijph^{(q),D} \}_{q=1}^Q := \left\{ \bd U^{D,n}_h(x,y):~  (x,y) \in 	\mathbb{Q}_\ijph \right\},
\end{equation}
where $Q$ is the total number of points in $\mathbb{Q}_\ijph$. For simplicity, we omit the superscript ``$n$" for time level when no confusion arises.
Given cell averages $\overline{\bf U}_\ijph^{\star,D} \in \mathcal{G}$, the limiter modifies $\{ \bd U_\ijph^{(q),D} \}_{q=1}^Q$ so that the resulting point values lie in $\mathcal{G}$, using the following steps:
\begin{itemize}[leftmargin=*]
	\item \textbf{Enforce positive density:}
	Let $\epsilon_\rho := \min\{ 10^{-13},\, \Bar \rho_\ijph^{\star,D} \}$. Define the limited density
	\begin{align*}
		\check \rho_\ijph^{(q),D} = (1 - \theta_\rho)\Bar \rho_\ijph^{\star,D} + \theta_\rho \rho_\ijph^{(q),D}, \qquad 1 \le q \le Q,
	\end{align*}
	where $\theta_\rho := \min\left\{ \frac{\Bar\rho_\ijph^{\star,D} - \epsilon_\rho}{\Bar\rho_\ijph^{\star,D} - \min\limits_{1 \le q \le Q} \rho_\ijph^{(q),D} },\, 1 \right\}$.
	Set the modified conservative variables as
$$\check{\bd U}_\ijph^{(q),D} = \Big(\big( \check \rho,~{\bm m},~B_1,~B_2,~B_3,~E \big)^\top\Big)_\ijph^{(q),D}.$$
	
	\item \textbf{Enforce positive internal energy:}
	Let $\epsilon_{\rho e} := \min\{ 10^{-13},\, \mathcal{E}(\bdbar U_\ijph^{\star,D}) \}$. Then apply
	\begin{align*}
		\bdhat U_\ijph^{(q),D} = (1 - \theta_{\rho e})\, \bdbar U_\ijph^{\star,D} + \theta_{\rho e}\, \check{\bd U}_\ijph^{(q),D}, \qquad 1 \le q \le Q,
	\end{align*}
	with
	\[
	\theta_{\rho e} := \min\left\{ \frac{\mathcal{E}(\bdbar U_\ijph^{\star,D}) - \epsilon_{\rho e}}{\mathcal{E}(\bdbar U_\ijph^{\star,D}) - \min\limits_{1 \le q \le Q} \mathcal{E}(\check{\bd U}_\ijph^{(q),D})},\, 1 \right\}.
	\]
	
	\item \textbf{Convert to primitive variables:}
	Transform $\bdhat U_\ijph^{(q),D} = (\Hat\rho,\, \Hat{\bm m},\, \Hat B_1,\, \Hat B_2,\, \Hat B_3,\, \Hat E)^\top$ into the primitive variables
	$
	\bdhat V_\ijph^{(q),D} = \left( \Hat\rho,\, \Hat{\bm u},\, \Hat B_1,\, \Hat B_2,\, \Hat B_3,\, \Hat p \right)^\top.
	$
	
	\item \textbf{Restore globally DF condition:}
	To preserve the globally DF property, replace $(\Hat B_1, \Hat B_2)$ in the primitive variables with the original magnetic components $(B_1, B_2)$:
	\[
	\bdtilde V_\ijph^{(q),D} = \left( \Hat\rho,\, \Hat{\bm u},\, B_1,\, B_2,\, \Hat B_3,\, \Hat p \right)^\top.
	\]
	The corresponding conservative variables are denoted as $
	\{ \bdtilde U_\ijph^{(q),D} \}_{q=1}^Q,
	$
	which lie in $\mathcal{G}$ and satisfy the globally DF condition.
\end{itemize}

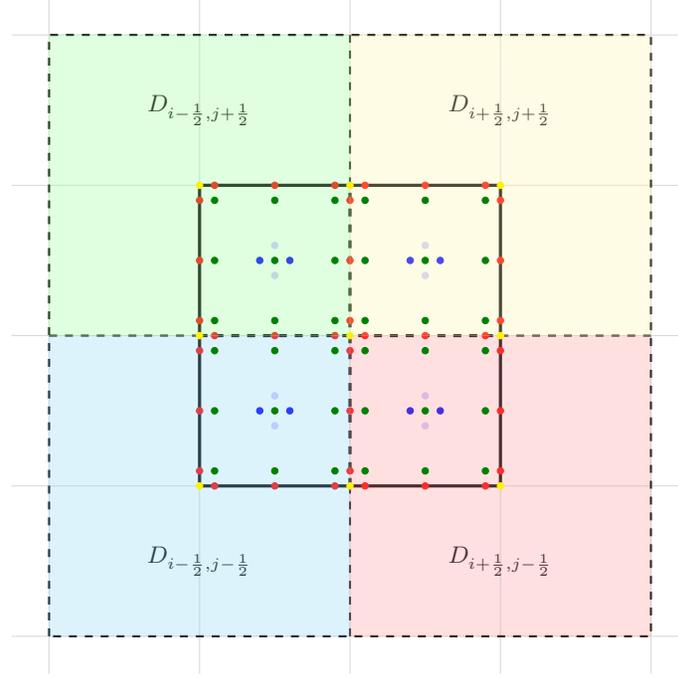
\begin{figure}[!thb]
	\vspace*{-0.2cm}
	\centering
	\begin{tikzpicture}[xscale=1, yscale=1]
		\def\one{-1}
		\def\XonePH{1}
		\def\Xtwo{3}
		\def\XtwoPH{5}
		\def\Xthr{7}
		\def\XthrPH{9}
		\def\Yone{-1}
		\def\YonePH{1}
		\def\Ytwo{3}
		\def\YtwoPH{5}
		\def\Ythr{7}
		\def\YthrPH{9}
		
		\foreach \x in {1,3,5,7,9} {
			\draw[gray!30] (\x,\YonePH-.5) -- (\x,\YthrPH+.5);
		}
		\foreach \y in {1,3,5,7,9} {
			\draw[gray!30] (\XonePH-.5,\y) -- (\XthrPH+.5,\y);
		}
		
		\draw [thick, dashed] (\XonePH, \YtwoPH) -- (\Xtwo, \YtwoPH);
		\draw [thick, dashed] (\Xthr, \YtwoPH) -- (\XthrPH, \YtwoPH);
		\draw [thick, dashed] (\XtwoPH, \YonePH) -- (\XtwoPH, \Ytwo);
		\draw [thick, dashed] (\XtwoPH, \Ythr) -- (\XtwoPH, \YthrPH);
		\draw [thick, dashed] (\XonePH, \YonePH) rectangle (\XthrPH, \YthrPH);
		
		\node at (3,8) {\small $D_{i-\frac12,j+\frac12}$};
		\node at (7,8) {\small $D_{i+\frac12,j+\frac12}$};
		\node at (3,2) {\small $D_{i-\frac12,j-\frac12}$};
		\node at (7,2) {\small $D_{i+\frac12,j-\frac12}$};

		\draw [very thick, dashed] (\Xtwo, \YtwoPH) -- (\Xthr, \YtwoPH);
		\draw [very thick, dashed] (\XtwoPH, \Ytwo) -- (\XtwoPH, \Ythr);

		\draw [very thick] (\Xtwo, \Ytwo) rectangle (\Xthr, \Ythr);

		\foreach \dx in {0.2, 1, 1.8} {
			\filldraw [red] (\XtwoPH+\dx, \Ythr) circle (1.2pt);
			\filldraw [red] (\XtwoPH+\dx, \YtwoPH) circle (1.2pt);
			\filldraw [red] (\XtwoPH+\dx, \Ytwo) circle (1.2pt);
			\filldraw [red] (\XtwoPH-\dx, \Ythr) circle (1.2pt);
			\filldraw [red] (\XtwoPH-\dx, \YtwoPH) circle (1.2pt);
			\filldraw [red] (\XtwoPH-\dx, \Ytwo) circle (1.2pt);
		}
		\foreach \dy in {0.2, 1, 1.8} {
			\filldraw [red] (\Xthr, \YtwoPH+\dy) circle (1.2pt);
			\filldraw [red] (\XtwoPH, \YtwoPH+\dy) circle (1.2pt);
			\filldraw [red] (\Xtwo, \YtwoPH+\dy) circle (1.2pt);
			\filldraw [red] (\Xthr, \YtwoPH-\dy) circle (1.2pt);
			\filldraw [red] (\XtwoPH, \YtwoPH-\dy) circle (1.2pt);
			\filldraw [red] (\Xtwo, \YtwoPH-\dy) circle (1.2pt);
		}

		\foreach \dx in {0.8, 1.2} {
			\filldraw [blue] (\XtwoPH+\dx, \YtwoPH+1) circle (1.2pt);
			\filldraw [blue] (\XtwoPH-\dx, \YtwoPH+1) circle (1.2pt);
			\filldraw [blue] (\XtwoPH+\dx, \YtwoPH-1) circle (1.2pt);
			\filldraw [blue] (\XtwoPH-\dx, \YtwoPH-1) circle (1.2pt);
		}
		\foreach \dy in {0.8, 1.2} {
			\filldraw [opacity=0.2, blue] (\XtwoPH+1, \YtwoPH+\dy) circle (1.2pt);
			\filldraw [opacity=0.2, blue] (\XtwoPH-1, \YtwoPH+\dy) circle (1.2pt);
			\filldraw [opacity=0.2, blue] (\XtwoPH+1, \YtwoPH-\dy) circle (1.2pt);
			\filldraw [opacity=0.2, blue] (\XtwoPH-1, \YtwoPH-\dy) circle (1.2pt);
		}
		
		\filldraw [cyan!40,opacity=0.3] (\XonePH,\YonePH) rectangle (\XtwoPH,\YtwoPH);
		\filldraw [green!40,opacity=0.3] (\XonePH,\YtwoPH) rectangle (\XtwoPH,\YthrPH);
		\filldraw [red!40,opacity=0.3] (\XtwoPH,\YtwoPH) rectangle (\XthrPH,\YonePH);
		\filldraw [yellow!40,opacity=0.3] (\XtwoPH,\YtwoPH) rectangle (\XthrPH,\YthrPH);

		\foreach \dx in {0.2, 1, 1.8} {
			\filldraw [darkgreen] (\XtwoPH-\dx, \YtwoPH-1.8) circle (1.2pt);
			\filldraw [darkgreen] (\XtwoPH-\dx, \YtwoPH-1) circle (1.2pt);
			\filldraw [darkgreen] (\XtwoPH-\dx, \YtwoPH-0.2) circle (1.2pt);
			\filldraw [darkgreen] (\XtwoPH-\dx, \YtwoPH+1) circle (1.2pt);
			\filldraw [darkgreen] (\XtwoPH-\dx, \YtwoPH+0.2) circle (1.2pt);
			\filldraw [darkgreen] (\XtwoPH-\dx, \YtwoPH+1.8) circle (1.2pt);
			
			\filldraw [darkgreen] (\XtwoPH+\dx, \YtwoPH-1.8) circle (1.2pt);
			\filldraw [darkgreen] (\XtwoPH+\dx, \YtwoPH-1) circle (1.2pt);
			\filldraw [darkgreen] (\XtwoPH+\dx, \YtwoPH-0.2) circle (1.2pt);
			\filldraw [darkgreen] (\XtwoPH+\dx, \YtwoPH+1) circle (1.2pt);
			\filldraw [darkgreen] (\XtwoPH+\dx, \YtwoPH+0.2) circle (1.2pt);
			\filldraw [darkgreen] (\XtwoPH+\dx, \YtwoPH+1.8) circle (1.2pt);
		}

		\foreach \dx in {-2, 0, 2} {
			\filldraw [yellow] (\XtwoPH+\dx, \Ytwo) circle (1.2pt);
			\filldraw [yellow] (\XtwoPH+\dx, \YtwoPH) circle (1.2pt);
			\filldraw [yellow] (\XtwoPH+\dx, \Ythr) circle (1.2pt);
		}
	\end{tikzpicture}
	\caption{Illustration for the $\mathbb{P}^2$ case. Solid lines: primal cell $C_\ij$; dashed lines: dual cells $D_\ijmh$, $D_{\imh,\jph}$, $D_\ijph$, and $D_{\iph,\jmh}$. Dark green dots ``\protect\tikz\protect\fill[darkgreen] (0,0) circle (1.5pt);'': tensor-product nodes $(\mathbb{Q}_i^x \otimes \mathbb{Q}_j^y)$; yellow dots ``\protect\tikz\protect\fill[yellow] (0,0) circle (1.5pt);'': vertex nodes $\mathbb{Q}^{v}_\ij$; red dots ``\protect\tikz\protect\fill[red] (0,0) circle (1.5pt);'': boundary nodes in $\mathbb{Q}_\ij^{\rm CAD}$; dark blue (or light blue) dots ``\protect\tikz\protect\fill[blue] (0,0) circle (1.5pt);'' (or ``\protect\tikz\protect\fill[opacity=0.2, blue] (0,0) circle (1.5pt);''): internal nodes in $\mathbb{Q}_\ij^{\rm CAD}$.}
	\label{Fig:CAD}
\end{figure}

\subsection{Cell average decomposition and modified dissipation terms}\label{subsec:CAD}

Decomposing the cell average of DG solution polynomials into a convex combination of certain point values plays a crucial role in establishing the PP property of updated cell averages, as demonstrated in \cite{zhang2010maximum, cui2023111882, cui2024optimal}. Such a cell average decomposition (CAD) is typically constructed using an appropriate 2D quadrature rule for DG polynomials. However, once the PP limiting procedure described above is applied, the modified point values $\{ \bdtilde U_\ijph^{(q),D} \}_{q=1}^Q$ are no longer representable by a single polynomial vector, thereby invalidating the standard CAD technique.

To overcome this challenge, we propose a key modification to the CDG dissipation terms ${\bm d}_{i,j}(\phi)$ and ${\bm d}_{i+\frac12,j+\frac12}(\phi)$ with $\phi(x,y)\equiv 1$ in \eqref{eq:dissipationC} and \eqref{eq:dissipationD}. This modification does not compromise the high-order accuracy of the CDG scheme. Its design is motivated by the rigorous PP analysis in \Cref{subsec:PP-property} and is essential for ensuring the updated cell averages remain in $\mathcal G$. We describe the modification for ${\bm d}_{i,j}(1)$; the treatment for ${\bm d}_{i+\frac12,j+\frac12}(1)$ is analogous.

According to the definition in \eqref{eqn:cdg-c}, the original dissipation with $\phi \equiv 1$ is given by
\begin{align*}
{\bm d}_{i,j}( 1 ) = \frac{1}{\DX\DY} \int_{C_\ij} \bd R_{h}^{D}(x,y) - \bd R_{h}^{C}(x,y) {\rm d}x{\rm d}y =: \overline{\bf R}_{i,j}^D - \overline{\bf R}_{i,j}^{\star,C}.
\end{align*}
Due to the overlapping mesh structure, $\bd R_{h}^{D}(x,y)$ is discontinuous on ${C_\ij}$. To evaluate $\overline{\bf R}_{i,j}^D$, we decompose the integral into four parts:
\begin{align}\label{eq:cell}
\overline{\bf R}_{i,j}^D = \frac14\sum_{\sigma',\sigma=\pm1}
	\frac{4}{\DX\DY} \int_{C_\ij\bigcap D_{i\pm\frac{\sigma'}2, j\pm\frac{\sigma}2}}
	\bd R_h^{D} (x,y) {\rm d}x{\rm d}y.
\end{align}
A suitable 2D quadrature is applied to each of the four subdomains to construct a CAD.

However, after the PP limiting, the modified point values $\{ \bdtilde U_\ijph^{(q),D} \}$ are no longer associated with a polynomial vector. To address this, we first apply a standard CAD to \eqref{eq:cell} and then replace the original values of $\bd R_h^D(x,y)$ at the CAD nodes with the corresponding modified values in $\{ \bdtilde U_\ijph^{(q),D} \}$. This yields the modification to $\overline{\bf R}_{i,j}^D$ as
\begin{align}\label{ocad}
\widetilde{\bf R}_{i,j}^D = \overline{\bf R}_\ij^{D,1} + \overline{\bf R}_\ij^{D,2} + \overline{\bf R}_\ij^{D,3}
\end{align}
with $\overline{\bf R}_\ij^{D,1} = \frac14\sum_{\sigma',\sigma=\pm1} \sum_{s=1}^{S} \Tilde\omega_s \bdtilde R_{i+\frac{\sigma'}4,j+\frac{\sigma}4}^{(s),D}$ and
\begin{align*}
	& \overline{\bf R}_\ij^{D,2} = \frac14\sum_{\sigma=\pm1}\sum_{\mu=1}^N \omega_\mu \bigg( \Tilde\omega_x\Big(
	\bdtilde R_{i,j+\frac{\sigma}4}^{-,(\mu),D}
	+ \bdtilde R_{i,j+\frac{\sigma}4}^{+,(\mu),D} \Big)
	+ \Tilde\omega_y \Big( \bdtilde R_{i+\frac{\sigma}4,j}^{(\mu),-,D}
	+ \bdtilde R_{i+\frac{\sigma}4,j}^{(\mu),+,D} \Big) \bigg), \\
	& \overline{\bf R}_\ij^{D,3} = \frac14\sum_{\sigma=\pm1}\sum_{\mu=1}^N \omega_\mu \bigg( \Tilde\omega_x\Big(
	\bdtilde R_{i+\frac12,j+\frac{\sigma}4}^{(\mu),D}
	+ \bdtilde R_{i-\frac12,j+\frac{\sigma}4}^{(\mu),D} \Big)
	+ \Tilde\omega_y\Big( \bdtilde R_{i+\frac{\sigma}4,j+\frac12}^{(\mu),D}
	+ \bdtilde R_{i+\frac{\sigma}4,j-\frac12}^{(\mu),D}  \Big) \bigg).
\end{align*}
Here, the modified point values ${\bdtilde R}_{i+\frac{\sigma'}{4},j+\frac{\sigma}{4}}^{(s),D}$, ${\bdtilde R}_{i,j+\frac{\sigma}{4}}^{\pm,(\mu),D}$, ${\bdtilde R}_{i+\frac{\sigma}{4},j}^{(\mu),\pm,D}$, ${\bdtilde R}_{i\pm\frac{1}{2},j+\frac{\sigma}{4}}^{(\mu),D}$, and ${\bdtilde R}_{i+\frac{\sigma}{4},j\pm\frac{1}{2}}^{(\mu),D}$ are evaluated at the CAD nodes: internal nodes $(\widetilde x_{i+\frac{\sigma'}{4}}^{(s)}, \widetilde y_{j+\frac{\sigma}{4}}^{(s)})$ and boundary nodes, such as $(x_i^\pm, y_{j+\frac{\sigma}{4}}^{(\mu)})$, etc. The CAD weights $\{\Tilde\omega_x, \Tilde\omega_y, \Tilde\omega_s\}$ satisfy $2\Tilde\omega_x + 2\Tilde\omega_y + \sum_{s=1}^S \Tilde\omega_s = 1$. The CAD nodes (see \Cref{Fig:CAD} for the $\mathbb{P}^2$ case) are defined as follows:
\begin{itemize}[leftmargin=2.em]
	\item Internal nodes ``\tikz\fill[opacity=0.2, blue] (0,0) circle (1.5pt);'' or ``\tikz\fill[blue] (0,0) circle (1.5pt);'':
	$\mathbb{Q}_{i+\frac{\sigma'}4,j+\frac{\sigma}4}^{\text{int}} := \{(\Tilde x_{i+\frac{\sigma'}4}^{(s)}, \Tilde y_{j+\frac{\sigma}4}^{(s)})\}_{s=1}^S$.
	Light blue nodes are used when $a_1/\DX \ge a_2/\DY$; dark blue when $a_1/\DX < a_2/\DY$, where $a_1$ and $a_2$ are the maximum wave speeds in the $x$- and $y$-directions.
	
	\item Boundary nodes ``\tikz\fill[red] (0,0) circle (1.5pt);'':
	$\mathbb{Q}_{i+\frac{\sigma'}4,j+\frac{\sigma}4}^{\text{bd}} := \big\{(x_{i+\frac{\sigma'}4}\pm\frac{\DX}4, y_{j+\frac{\sigma}4}^{(\mu)}), (x_{i+\frac{\sigma'}4}^{(\mu)}, y_{j+\frac{\sigma}4}\pm\frac{\DY}4) \big\}_{\mu=1}^N$, which include all Gauss points for approximating the 1D integrals in \eqref{eqn:cdg-c}.
\end{itemize}
The full set of CAD nodes used to modify the dissipation term $\widetilde{\bm d}_{i,j}^D$ for cell $C_{ij}$ is given by
\begin{align}\label{Q_CAD}
	\mathbb{Q}_\ij^{\rm CAD} = \bigcup_{\sigma'=\pm1} \bigcup_{\sigma=\pm1}
	\left( \mathbb{Q}_{i+\frac{\sigma'}4,j+\frac{\sigma}4}^{\text{bd}} \cup \mathbb{Q}_{i+\frac{\sigma'}4,j+\frac{\sigma}4}^{\text{int}} \right).
\end{align}
Let $\{ \phi_{i,j}^{(\ell)} \}_{\ell \ge 0}$ with $\phi_{i,j}^{(0)} \equiv 1$ denote an orthogonal basis of $\mathbb{P}^k(C_{i,j})$. The modified CDG dissipation term $\widetilde{\bm d}_{i,j}(\phi)$ is then defined as
\begin{align}\label{modified}
\widetilde {\bm d}_{i,j}(\phi_{i,j}^{(\ell)}) =
\begin{cases}
	\widetilde {\bf R}_{i,j}^D - \overline{\bf R}_{i,j}^{\star,C}, \quad & \mbox{if } \ell =0,
	\\
	{\bm d}_{i,j}(\phi_{i,j}^{(\ell)}), \quad & \mbox{if } \ell \ge 1.
\end{cases}
\end{align}
That is, the modification is applied only to the dissipation term associated with the zeroth moment, which governs the update of cell averages.
An analogous modification is applied to the dissipation term ${\bm d}_{i+\frac12,j+\frac12}(1)$ in \eqref{eq:dissipationD}.
It is worth noting that this modification may compromise the conservation of total energy.
Nevertheless, it ensures the provable preservation of the positivity property for the updated cell averages while maintaining the high-order accuracy of the CDG scheme.

We next introduce two representative CADs, either of which can be employed in the current framework.

\begin{rem}[Zhang--Shu CAD \cite{zhang2010maximum}]\label{rem:cad}\rm
	A classical CAD was proposed in \cite{zhang2010maximum}, constructed using tensor products of the 1D Gauss and Gauss--Lobatto quadrature rules.
	The CAD nodes are taken as the tensor product of these quadrature points in the $x$- and $y$-directions; see \cite{zhang2010maximum,cui2023111882} for details.
	The corresponding CAD boundary weights, which directly influence the BP CFL number, are given by
	\begin{align*}
		\Tilde\omega_x = \frac{\Hat\omega_1\, \Hat\alpha_1/\Delta x}{\Hat\alpha_1/\Delta x + \Hat\alpha_2/\Delta y}, \quad
		\Tilde\omega_y = \frac{\Hat\omega_1\, \Hat\alpha_2/\Delta y}{\Hat\alpha_1/\Delta x + \Hat\alpha_2/\Delta y},
	\end{align*}
	where $\Hat\omega_1 = \frac{1}{L(L-1)}$ with $L = \big\lceil \frac{k+3}{2} \big\rceil$, and $\Hat\alpha_\ell$ is the maximum wave speed.
\end{rem}

\begin{rem}[Cui--Ding--Wu CAD \cite{cui2023111882,cui2024optimal}]\label{rem:ocad}\rm
	An improved CAD, recently proposed in \cite{cui2023111882,cui2024optimal}, offers higher efficiency by requiring fewer nodes and providing larger boundary weights. This results in lower computational cost and a more relaxed BP CFL condition. The internal CAD nodes can be found in \cite{cui2023111882,cui2024optimal}. The boundary weights are given by
	\begin{align*}
		\Tilde\omega_x = \frac{\Bar\omega_{\star}\, a_1/\Delta x}{a_1/\Delta x + a_2/\Delta y}, \quad
		\Tilde\omega_y = \frac{\Bar\omega_{\star}\, a_2/\Delta y}{a_1/\Delta x + a_2/\Delta y},
	\end{align*}
	where $\Bar\omega_{\star}$ depends on $\zeta := \frac{a_1/\Delta x - a_2/\Delta y}{a_1/\Delta x + a_2/\Delta y}$ and the polynomial degree $k$.
 For $k = 1$, $\Bar\omega_{\star} = \frac{1}{2}$.
 For $k = 2$ or $3$, $\Bar\omega_{\star} = \frac{1}{4 + 2|\zeta|}$.
 For $k = 4$ or $5$,
		$$
			\Bar\omega_{\star} = \left[\frac{14}{3} + \frac{2}{3}\sqrt{78\zeta^2 + 46} \cos
			\left(\frac{1}{3} \arccos\left( \frac{1476\zeta^2 - 244}{(78\zeta^2 + 46)^{3/2}} \right)\right)\right]^{-1}.
		$$
	For $k \geq 6$, see \cite{cui2024optimal}.
\end{rem}

\subsection{Evolution for the auxiliary magnetic field averages}\label{subsec:Aux-Q}

Due to the distinct evolution mechanisms of $\bd R$ and $\bd Q$ described in \Cref{sec:GDF-CDG}, it is challenging to rigorously ensure that the cell average of $\bd U$, composed of $\bd R$ and $\bd Q$, remains within the admissible set $\mathcal{G}$.
To address this, we introduce auxiliary cell averages for the magnetic field, denoted by $\overline{\bf Q}_\ij^{\star,C}$ and $\overline{\bf Q}_\ijph^{\star,D}$ on the primal cell $C_\ij$ and dual cell $D_\ijph$, respectively. These are evolved using the standard CDG formulation for the zeroth moment (equivalent to a finite volume update), with modified dissipation terms. For instance, on the primal cell $C_\ij$, the evolution equations are given by
\begin{align}\label{eqn:cdg-c-bx-ave}
	(\Bar B_1)_\ij^{\star,C,n+1} &=
	(\Bar B_1)_\ij^{\star,C,n} + \theta\, \widetilde d_\ij^{B_1}
	- \frac{\Delta t}{\Delta y} \sum_{\sigma=\pm1} \sum_{\mu = 1}^N \frac{\omega_{\mu}}{2}
	\left(  G(\bdtilde U_{i+\frac{\sigma}4,\jph}^{(\mu),D})
	- G(\bdtilde U_{i+\frac{\sigma}4,\jmh}^{(\mu),D}) \right), \\
	\label{eqn:cdg-c-by-ave}
	(\Bar B_2)_\ij^{\star,C,n+1} &=
	(\Bar B_2)_\ij^{\star,C,n} + \theta\, \widetilde d_\ij^{B_2}
	+ \frac{\Delta t}{\Delta x} \sum_{\sigma=\pm1} \sum_{\mu = 1}^N \frac{\omega_{\mu}}{2}
	\left(  G(\bdtilde U_{\iph,j+\frac{\sigma}4}^{(\mu),D})
	- G(\bdtilde U_{\imh,j+\frac{\sigma}4}^{(\mu),D}) \right),
\end{align}
where $\widetilde d_\ij^{B_\ell} := (\widetilde B_\ell)_\ij^{D,n} - (\Bar B_\ell)_\ij^{\star,C,n}$ for $\ell = 1,2$ are the modified dissipation terms. Here, $(\widetilde B_\ell)_\ij^{D,n}$ represents the modified version of $(\Bar B_\ell)_\ij^{D,n}$ obtained using the same CAD and procedure in \eqref{ocad}.
The evolution of the auxiliary magnetic field averages $\overline{\bf Q}_{i+\frac{1}{2},j+\frac{1}{2}}^{\star,D}$ on the dual cells $D_{i+\frac{1}{2},j+\frac{1}{2}}$ follows a similar structure and is therefore omitted.

\subsection{Provable PP property for the updated cell averages}\label{subsec:PP-property}

This subsection presents a theoretical proof that the updated cell average $\overline{\bd U}_{\hat{\imath}, \hat{\jmath}}^{\star,\sharp,n+1}$, composed of $\overline{\bd R}_{\hat{\imath}, \hat{\jmath}}^{\star,\sharp,n+1}$ and $\overline{\bf Q}_{\hat{\imath}, \hat{\jmath}}^{\star,\sharp,n+1}$, remains within the admissible state set $\mathcal{G}$.
The cell average $\overline{\bd R}_{\hat{\imath}, \hat{\jmath}}^{\star,\sharp,n+1}$ is computed using \eqref{eqn:cdg-c}--\eqref{eqn:cdg-d} with $\phi = 1$, where the integrals are approximated via Gauss quadrature with nodal values replaced by their modified counterparts $\{ \bdtilde U_{\hat{\imath}, \hat{\jmath}}^{(q),\sharp} \}$, and the dissipation terms are substituted with the modified versions $\widetilde{\bm d}_{\hat{\imath}, \hat{\jmath}}$. The auxiliary magnetic field average $\overline{\bf Q}_{\hat{\imath}, \hat{\jmath}}^{\star,\sharp,n+1}$ is updated using \eqref{eqn:cdg-c-bx-ave}--\eqref{eqn:cdg-c-by-ave}.

The proof is nontrivial due to the strong nonlinearity of the MHD system and the nonlinear constraint $\mathcal{E}(\bd U) > 0$ in \eqref{ASS}. To address this challenge, we adopt the geometric quasi-linearization (GQL) approach \cite{WuSINUM2018,wu2023geometric}, which skillfully transforms the nonlinear constraint into equivalent linear forms.

\begin{lem}[GQL representation \cite{WuSINUM2018,wu2023geometric}]\label{Lem:GQL}
	 The admissible state set $\mathcal{G}$ defined in \eqref{ASS} is equivalent to
	\begin{align}\label{ASS*}
		\mathcal{G}_{*}=\left\{\bd U=(\rho,{\bm m},{\bm B},E)^\top:\bd U\cdot {\bm n}_1>0,\ \bd U\cdot {\bm n}^{*}+\frac{\abs{{\bm B}^{*}}^2}2>0\quad \forall {\bm u}^{*},{\bm B}^{*}\in\mathbb{R}^3\right\},
	\end{align}
	where ${\bm n}_1 = (1,0,0,0,0,0,0,0)^\top$, ${\bm n}^{*} = \big(\frac{\abs{{\bm u}^{*}}^2}2,-{\bm u}^{*},-{\bm B}^{*},1\big)^\top$, and $\{{\bm u}^{*},{\bm B}^{*}\}$ are free auxiliary variables independent of $\bd U$.
\end{lem}

\Cref{Lem:GQL}, first established in \cite{WuSINUM2018} and later interpreted geometrically in \cite{wu2023geometric}, provides the foundation for our analysis.

We also recall a key technical lemma from \cite{WuSINUM2018}, which will be invoked in the proof.

\begin{lem}[Flux inequalities for GQL]\label{Lem:GLF}
	For any $\bd U,\bdtilde U \in \mathcal{G}$, ${\bm u}^{*}, {\bm B}^{*} \in \mathbb{R}^3$, and $\ell \in \{1,2,3\}$, the following hold:
	\begin{align*}
		-(\bd F_{\ell}(\bd U) - \bd F_{\ell}(\bdtilde U))\cdot{\bm n}_{1}
		&> -\alpha_{\ell}(\bd U,\bdtilde U)(\bd U + \bdtilde U)\cdot {\bm n}_1,
	\\
		-(\bd F_{\ell}(\bd U) - \bd F_{\ell}(\bdtilde U))\cdot{\bm n}^{*}
		&\geq -\alpha_{\ell}(\bd U,\bdtilde U)\big((\bd U + \bdtilde U)\cdot {\bm n}^{*} +\abs{{\bm B}^{*}}^2 \big)
		-(B_{\ell}-\Tilde B_{\ell})({\bm u}^{*}\cdot{\bm B}^{*}),
	\end{align*}
	where
	\begin{align*}
		\alpha_{\ell}(\bd U,\bdtilde U)
		&= \max\left\{ \abs{u_\ell}+\mathcal{C}_\ell,\abs{\Tilde u_\ell}+\mathcal{\Tilde C}_\ell,
		\frac{\sqrt{\rho}u_\ell+\sqrt{\Tilde \rho}\Tilde u_\ell}{\sqrt{\rho}+\sqrt{\Tilde \rho}}+\max\{\mathcal{C}_\ell,\mathcal{\Tilde C}_\ell\}\right\} + \frac{|{\bm B} - \Tilde {\bm B}|}{\sqrt{\rho}+\sqrt{\Tilde \rho}},
\\
		\mathcal{C}_\ell &= \frac1{\sqrt{2}}\left[\mathcal{C}_s^2+\frac{\abs{{\bm B}}^2}{\rho}+\sqrt{\left(\frac{\mathcal{C}_s^2 + \abs{{\bm B}}^2}{\rho}\right)^2-\frac{4\mathcal{C}_s^2B_\ell^2}{\rho}}\right]^{\frac12},\quad
		\mathcal{C}_s = \frac{p}{\rho\sqrt{2e}}.
	\end{align*}
\end{lem}

With the GQL representation in \Cref{Lem:GQL} and the flux estimates in \cref{Lem:GLF}, we are now in a position to rigorously establish the PP property of the updated cell averages $\overline{\bd U}_{\hat{\imath}, \hat{\jmath}}^{\star,\sharp,n+1}$ for the proposed PosDiv-CDG schemes.

\begin{thm}\label{Thm:PP}
Assume that the modified dissipation terms are constructed using the improved CAD described in \Cref{rem:ocad}.
Then, for the $\mathbb{P}^k$-based PosDiv-CDG method, the updated cell averages $\overline{\bd U}_{i,j}^{\star,C,n+1} \in \mathcal{G}_{*}$ and $\overline{\bd U}_{i+\frac12,j+\frac12}^{\star,D,n+1} \in \mathcal{G}_{*}$
are provably admissible under the following CFL-type condition:
\begin{align}\label{CFL}
	\Delta t \left( \frac{\Hat\alpha_1}{\Delta x} + \frac{\Hat\alpha_2}{\Delta y} \right)
	\leq \frac{\theta\, \Bar\omega_{\star}}{2}, \qquad
	\theta := \frac{\Delta t}{\tau_{\max}} \in (0,1],
\end{align}
where $\Bar\omega_{\star}$ is given in \Cref{rem:ocad}, and the maximal wave speed estimates are
	\begin{align*}
		& \Hat\alpha_1 = \max_{\sigma=\pm 1}\bigg\{ \max_{i,j,\mu} \alpha_1
		\Big(\bdtilde U_{\iph,{j+\frac{\sigma}4}}^{(\mu),D},
		\bdtilde U_{\imh,{j+\frac{\sigma}4}}^{(\mu),D}\Big),~
		\max_{i,j,\mu} \alpha_1 \Big(\bdtilde U_{i+1,{j+\frac12+\frac{\sigma}4}}^{(\mu),C},
		\bdtilde U_{i,{j+\frac12+\frac{\sigma}4}}^{(\mu),C}\Big) \bigg\}, \\
		& \Hat\alpha_2 = \max_{\sigma=\pm 1}\bigg\{ \max_{i,j,\mu} \alpha_2
		\Big(\bdtilde U_{{i+\frac{\sigma}4},\jph}^{(\mu),D},
		\bdtilde U_{{i+\frac{\sigma}4},\jmh}^{(\mu),D}\Big),~
		\max_{i,j,\mu} \alpha_2 \Big(\bdtilde U_{{i+\frac12+\frac{\sigma}4},j+1}^{(\mu),C},
		\bdtilde U_{{i+\frac12+\frac{\sigma}4},j}^{(\mu),C}\Big) \bigg\}.
	\end{align*}
\end{thm}

\begin{proof}
We present the proof for $\overline{\bd U}_{i,j}^{\star,C,n+1} \in \mathcal{G}_{*}$; the argument for $\overline{\bd U}_{i+\frac12,j+\frac12}^{\star,D,n+1} \in \mathcal{G}_{*}$ follows analogously.

The update of the full cell average $\overline{\bd U}_{i,j}^{\star,C,n+1}$, composed of $\overline{\bd R}_{i,j}^{\star,C,n+1}$ and $\overline{\bf Q}_{i,j}^{\star,C,n+1}$, can be written as
	\begin{align}\label{eqn:evolve-U}
		\bdbar U_\ij^{\star,C,n+1}
		= & (1-\theta) \bdbar U_\ij^{\star,C} + \theta {\bf \Pi}_\ij^{D} + \DT {\bf \Pi}_\ij^{F},
	\end{align}
where ${\bf \Pi}_{i,j}^{D}$ consists of $\widetilde{\bd R}_{i,j}^D$ and $(\widetilde{B}_\ell)_{i,j}^{D,n}$ in the modified dissipation terms, and can be decomposed as
	\begin{align}\label{eq:PiD}
		{\bf \Pi}_\ij^{D} = {\bf \Pi}_\ij^{D,1} + {\bf \Pi}_\ij^{D,2} + {\bf \Pi}_\ij^{D,3}
	\end{align}
	with ${\bf \Pi}_\ij^{D,1} = \frac14\sum_{\sigma',\sigma=\pm1} \sum_{s=1}^{S} \Tilde\omega_s \bdtilde U_{i+\frac{\sigma'}4,j+\frac{\sigma}4}^{(s),D}$ and
	\begin{align*}
		& {\bf \Pi}_\ij^{D,2} = \frac14\sum_{\sigma=\pm1}\sum_{\mu=1}^N \omega_\mu \bigg( \Tilde\omega_x\Big(
		\bdtilde U_{i,j+\frac{\sigma}4}^{-,(\mu),D}
		+ \bdtilde U_{i,j+\frac{\sigma}4}^{+,(\mu),D} \Big)
		+ \Tilde\omega_y\Big( \bdtilde U_{i+\frac{\sigma}4,j}^{(\mu),-,D}
		+ \bdtilde U_{i+\frac{\sigma}4,j}^{(\mu),+,D} \Big) \bigg), \\
		& {\bf \Pi}_\ij^{D,3} = \frac14\sum_{\sigma=\pm1}\sum_{\mu=1}^N \omega_\mu \bigg( \Tilde\omega_x\Big(
		\bdtilde U_{i+\frac12,j+\frac{\sigma}4}^{(\mu),D}
		+ \bdtilde U_{i-\frac12,j+\frac{\sigma}4}^{(\mu),D} \Big)
		+ \Tilde\omega_y\Big( \bdtilde U_{i+\frac{\sigma}4,j+\frac12}^{(\mu),D}
		+ \bdtilde U_{i+\frac{\sigma}4,j-\frac12}^{(\mu),D}  \Big) \bigg).
	\end{align*}
The term ${\bf \Pi}_{i,j}^{F}$ in \eqref{eqn:evolve-U}  arises from the discrete flux integrals in \eqref{eqn:cdg-c-bx-ave}-\eqref{eqn:cdg-c-by-ave} and \eqref{eqn:cdg-c}, evaluated via Gauss quadrature with the modified point values:
	\begin{align}
		{\bf \Pi}_\ij^{F} = & -\frac{1}{\DX}\sum_{\sigma=\pm1}\sum_{\mu=1}^{N}\frac{\omega_\mu}2
		\left( \bd F_1\big(\bdtilde U_{\iph,j+\frac{\sigma}4}^{(\mu),D}\big)
		- \bd F_1\big(\bdtilde U_{\imh,j+\frac{\sigma}4}^{(\mu),D}\big) \right) \notag \\
		& - \frac{1}{\DY}\sum_{\sigma=\pm1}\sum_{\mu=1}^{N}\frac{\omega_\mu}2
		\left( \bd F_2\big(\bdtilde U_{i+\frac{\sigma}4,\jph}^{(\mu),D}\big)
		- \bd F_2\big(\bdtilde U_{i+\frac{\sigma}4,\jmh}^{(\mu),D}\big) \right).
	\end{align}
    Applying \cref{Lem:GLF}, we derive the following lower bounds:
	\begin{align}
		{\bf \Pi}_\ij^{F} \cdot {\bm n}_1
		\geq & -\frac{1}{\DX}\sum_{\sigma=\pm1}\sum_{\mu=1}^{N}\frac{\omega_\mu}2
		\Hat\alpha_1 \Big(\bdtilde U_{\iph,{j+\frac{\sigma}4}}^{(\mu),D} +
		\bdtilde U_{\imh,{j+\frac{\sigma}4}}^{(\mu),D}\Big) \cdot {\bm n}_1 \notag \\
		& - \frac{1}{\DY}\sum_{\sigma=\pm1}\sum_{\mu=1}^{N}\frac{\omega_\mu}2
		\Hat\alpha_2 \Big(\bdtilde U_{{i+\frac{\sigma}4},\jph}^{(\mu),D} +
		\bdtilde U_{{i+\frac{\sigma}4},\jmh}^{(\mu),D}\Big) \cdot {\bm n}_1, \label{eq:PiFn1}
	\end{align}
	\begin{align}\label{eq:flux_nstar}
		{\bf \Pi}_\ij^{F} \cdot {\bm n}^*
		\geq & - \frac{1}{\DX}\sum_{\sigma=\pm1}\sum_{\mu=1}^{N}\frac{\omega_\mu}2
		\Hat\alpha_1 \left( \Big(\bdtilde U_{\iph,j+\frac{\sigma}4}^{(\mu),D} +
		\bdtilde U_{\imh,j+\frac{\sigma}4}^{(\mu),D}\Big) \cdot {\bm n}^* + \abs{{\bm B}^{*}}^2 \right) \notag \\
		& - \frac{1}{\DY}\sum_{\sigma=\pm1}\sum_{\mu=1}^{N}\frac{\omega_\mu}2
		\Hat\alpha_2 \left( \Big(\bdtilde U_{i+\frac{\sigma}4,\jph}^{(\mu),D} +
		\bdtilde U_{i+\frac{\sigma}4,\jmh}^{(\mu),D}\Big) \cdot {\bm n}^* + \abs{{\bm B}^{*}}^2 \right) \notag \\
		& - \frac{1}{\DX}\sum_{\sigma=\pm1}\sum_{\mu=1}^{N}\frac{\omega_\mu}2
		\Big((B_1)_{\iph,j+\frac{\sigma}4}^{(\mu),D} -
		(B_1)_{\imh,j+\frac{\sigma}4}^{(\mu),D}\Big)({\bm u}^{*}\cdot{\bm B}^{*}) \notag \\
		& - \frac{1}{\DY}\sum_{\sigma=\pm1}\sum_{\mu=1}^{N}\frac{\omega_\mu}2
		\Big( (B_2)_{i+\frac{\sigma}4,\jph}^{(\mu),D} -
		(B_2)_{i+\frac{\sigma}4,\jmh}^{(\mu),D} \Big)({\bm u}^{*}\cdot{\bm B}^{*}).
	\end{align}
	Thanks to the globally DF property, we have $0=\int_{ C_{i,j} } \nabla \cdot {\bf Q}_{h}^{D,n} {\rm d}x {\rm d}y = \int_{\partial C_{i,j}} {\bf Q}_{h}^{D,n} \cdot {\bf n} ds$,
	where the second equality follows from the divergence theorem. Owing to the exactness of the Gauss quadrature for integrals of ${\bf Q}_{h}^{D,n} \cdot {\bf n}$ along $\partial C_{i,j}$, the discrete magnetic field satisfies
	\begin{align}\notag
		\sum_{\sigma=\pm1}\sum_{\mu=1}^{N}\frac{\omega_\mu}2
		\left( \frac{(B_1)_{\iph,j+\frac{\sigma}4}^{(\mu),D} -
			(B_1)_{\imh,j+\frac{\sigma}4}^{(\mu),D}}{\DX}
		+ \frac{(B_2)_{i+\frac{\sigma}4,\jph}^{(\mu),D} -
			(B_2)_{i+\frac{\sigma}4,\jmh}^{(\mu),D}}{\DY} \right) =0,
	\end{align}
which ensures that all terms involving $({\bm u}^* \cdot {\bm B}^*)$ in \eqref{eq:flux_nstar} cancel exactly. Consequently, inequality \eqref{eq:flux_nstar} simplifies to
	\begin{align}\label{eq:PiF}
		{\bf \Pi}_\ij^{F} \cdot {\bm n}^*
		\geq & - \frac{1}{\DX}\sum_{\sigma=\pm1}\sum_{\mu=1}^{N}\frac{\omega_\mu}2
		\Hat\alpha_1 \left( \Big(\bdtilde U_{\iph,j+\frac{\sigma}4}^{(\mu),D} +
		\bdtilde U_{\imh,j+\frac{\sigma}4}^{(\mu),D}\Big) \cdot {\bm n}^* + \abs{{\bm B}^{*}}^2 \right) \notag \\
		& - \frac{1}{\DY}\sum_{\sigma=\pm1}\sum_{\mu=1}^{N}\frac{\omega_\mu}2
		\Hat\alpha_2 \left( \Big(\bdtilde U_{i+\frac{\sigma}4,\jph}^{(\mu),D} +
		\bdtilde U_{i+\frac{\sigma}4,\jmh}^{(\mu),D}\Big) \cdot {\bm n}^* + \abs{{\bm B}^{*}}^2 \right).
	\end{align}
Substituting \eqref{eq:PiFn1} and \eqref{eq:PiF} into \eqref{eqn:evolve-U} and combining with the convex decomposition of the dissipation terms \eqref{eq:PiD}, we obtain the following estimates:
	\begin{align*}
		& \bdbar U_\ij^{\star,C,n+1} \cdot {\bm n}_1  = \Big( (1-\theta) \bdbar U_\ij^{\star,C} + \theta \big( {\bf \Pi}_\ij^{D,1} + {\bf \Pi}_\ij^{D,2} + {\bf \Pi}_\ij^{D,3} \big) + \DT {\bf \Pi}_\ij^{F} \Big) \cdot {\bm n}_1 \notag \\
		& \quad \geq (1-\theta) \bdbar U_\ij^{\star,C} \cdot {\bm n}_1 + \theta \Big( {\bf \Pi}_\ij^{D,1} + {\bf \Pi}_\ij^{D,2} \Big) \cdot {\bm n}_1 \notag \\
		& \quad\quad + \frac{\theta}2\sum_{\sigma=\pm1}\sum_{\mu=1}^N \frac{\omega_\mu}2
		\bigg( \Tilde\omega_x \Big(\bdtilde U_{\iph,{j+\frac{\sigma}4}}^{(\mu),D} +
		\bdtilde U_{\imh,{j+\frac{\sigma}4}}^{(\mu),D}\Big)
		+ \Tilde\omega_y \Big(\bdtilde U_{{i+\frac{\sigma}4},\jph}^{(\mu),D} +
		\bdtilde U_{{i+\frac{\sigma}4},\jmh}^{(\mu),D}\Big) \bigg) \cdot {\bm n}_1 \notag \\
		& \quad\quad -\frac{\DT}{\DX}\sum_{\sigma=\pm1}\sum_{\mu=1}^{N}\frac{\omega_\mu}2
		\Hat\alpha_1  \Big(\bdtilde U_{\iph,{j+\frac{\sigma}4}}^{(\mu),D} +
		\bdtilde U_{\imh,{j+\frac{\sigma}4}}^{(\mu),D}\Big) \cdot {\bm n}_1 \notag \\
		& \quad\quad - \frac{\DT}{\DY}\sum_{\sigma=\pm1}\sum_{\mu=1}^{N}\frac{\omega_\mu}2
		\Hat\alpha_2  \Big(\bdtilde U_{{i+\frac{\sigma}4},\jph}^{(\mu),D} +
		\bdtilde U_{{i+\frac{\sigma}4},\jmh}^{(\mu),D}\Big) \cdot {\bm n}_1 \notag \\
		& \quad = (1-\theta) \bdbar U_\ij^{\star,C} \cdot {\bm n}_1 + \theta \Big( {\bf \Pi}_\ij^{D,1} + {\bf \Pi}_\ij^{D,2} \Big) \cdot {\bm n}_1 \notag \\
		& \quad\quad + \bigg(\frac{\theta}2 \Tilde\omega_x -\frac{\DT}{\DX} \Hat \alpha_1 \bigg) \sum_{\sigma=\pm1}\sum_{\mu=1}^{N}\frac{\omega_\mu}2
		\Big(\bdtilde U_{\iph,{j+\frac{\sigma}4}}^{(\mu),D} +
		\bdtilde U_{\imh,{j+\frac{\sigma}4}}^{(\mu),D}\Big) \cdot {\bm n}_1 \notag \\
		& \quad\quad + \bigg(\frac{\theta}2 \Tilde\omega_y -\frac{\DT}{\DY} \Hat \alpha_2 \bigg)
		\sum_{\sigma=\pm1}\sum_{\mu=1}^{N}\frac{\omega_\mu}2
		\Big(\bdtilde U_{{i+\frac{\sigma}4},\jph}^{(\mu),D} +
		\bdtilde U_{{i+\frac{\sigma}4},\jmh}^{(\mu),D}\Big) \cdot {\bm n}_1,
\\
		& \bdbar U_\ij^{\star,C,n+1} \cdot {\bm n}^* + \abs{{\bm B}^{*}}^2  = \Big( (1-\theta) \bdbar U_\ij^{\star,C} + \theta \big( {\bf \Pi}_\ij^{D,1} + {\bf \Pi}_\ij^{D,2} + {\bf \Pi}_\ij^{D,3} \big) + \DT {\bf \Pi}_\ij^{F} \Big) \cdot {\bm n}^* + \abs{{\bm B}^{*}}^2 \notag \\
		& \qquad \geq (1-\theta) \Big( \bdbar U_\ij^{\star,C} + \abs{{\bm B}^{*}}^2 \Big) + \theta \Big(\big( {\bf \Pi}_\ij^{D,1} + {\bf \Pi}_\ij^{D,2} \big)  + \abs{{\bm B}^{*}}^2 \Big) \notag \\
		& \qquad \quad + \frac{\theta}2\sum_{\sigma=\pm1}\sum_{\mu=1}^N \frac{\omega_\mu}2
		\Bigg( \Tilde\omega_x \bigg( \Big( \bdtilde U_{\iph,{j+\frac{\sigma}4}}^{(\mu),D} +
		\bdtilde U_{\imh,{j+\frac{\sigma}4}}^{(\mu),D}\Big) \cdot {\bm n}^* + \abs{{\bm B}^{*}}^2\bigg) \Bigg) \notag \\
		& \qquad \quad + \frac{\theta}2\sum_{\sigma=\pm1}\sum_{\mu=1}^N \frac{\omega_\mu}2
		\Bigg( \Tilde\omega_y \bigg( \Big(\bdtilde U_{{i+\frac{\sigma}4},\jph}^{(\mu),D} +
		\bdtilde U_{{i+\frac{\sigma}4},\jmh}^{(\mu),D}\Big) \cdot {\bm n}^* + \abs{{\bm B}^{*}}^2\bigg) \Bigg) \notag \\
		& \qquad \quad - \frac{\DT}{\DX}\sum_{\sigma=\pm1}\sum_{\mu=1}^{N}\frac{\omega_\mu}2
		\Hat\alpha_1  \left( \Big(\bdtilde U_{\iph,{j+\frac{\sigma}4}}^{(\mu),D} +
		\bdtilde U_{\imh,{j+\frac{\sigma}4}}^{(\mu),D}\Big) \cdot {\bm n}^* + \abs{{\bm B}^{*}}^2 \right) \notag \\
		& \qquad \quad - \frac{\DT}{\DY}\sum_{\sigma=\pm1}\sum_{\mu=1}^{N}\frac{\omega_\mu}2
		\Hat\alpha_2 \left( \Big(\bdtilde U_{{i+\frac{\sigma}4},\jph}^{(\mu),D} +
		\bdtilde U_{{i+\frac{\sigma}4},\jmh}^{(\mu),D}\Big) \cdot {\bm n}^* + \abs{{\bm B}^{*}}^2 \right) \notag \\
		& \qquad = (1-\theta) \Big( \bdbar U_\ij^{\star,C} + \abs{{\bm B}^{*}}^2 \Big) + \theta \Big(\big( {\bf \Pi}_\ij^{D,1} + {\bf \Pi}_\ij^{D,2} \big)  + \abs{{\bm B}^{*}}^2 \Big) \notag \\
		& \qquad \quad + \bigg(\frac{\theta}2 \Tilde\omega_x -\frac{\DT}{\DX} \Hat \alpha_1 \bigg) \sum_{\sigma=\pm1}\sum_{\mu=1}^{N}\frac{\omega_\mu}2
		\left( \Big(\bdtilde U_{\iph,{j+\frac{\sigma}4}}^{(\mu),D} +
		\bdtilde U_{\imh,{j+\frac{\sigma}4}}^{(\mu),D}\Big) \cdot {\bm n}^* + \abs{{\bm B}^{*}}^2 \right) \notag \\
		& \qquad \quad + \bigg(\frac{\theta}2 \Tilde\omega_y -\frac{\DT}{\DY} \Hat \alpha_2 \bigg)
		\sum_{\sigma=\pm1}\sum_{\mu=1}^{N}\frac{\omega_\mu}2
		\left( \Big(\bdtilde U_{{i+\frac{\sigma}4},\jph}^{(\mu),D} +
		\bdtilde U_{{i+\frac{\sigma}4},\jmh}^{(\mu),D}\Big) \cdot {\bm n}^* + \abs{{\bm B}^{*}}^2 \right).
	\end{align*}
Under the CFL-type condition \eqref{CFL}, it follows that $\bdbar U_\ij^{\star,C,n+1} \cdot {\bm n}_1 > 0$ and $\bdbar U_\ij^{\star,C,n+1} \cdot {\bm n}^* + \abs{{\bm B}^{*}}^2 > 0$ for any ${\bm u}^{*},{\bm B}^{*} \in \mathbb{R}^3$.
By the GQL representation in \cref{Lem:GQL}, we conclude that $\bdbar U_\ij^{\star,C,n+1} \in \mathcal{G}_{*} = \mathcal{G}$.
Using analogous arguments, one can similarly show that $\bdbar U_\ijph^{\star,D,n+1} \in \mathcal{G}_{*} = \mathcal{G}$. The proof is completed.
\end{proof}

\subsection{High-order time discretization}\label{subsec:RK}
For clarity, the preceding discussions focused on the first-order forward Euler time discretization.
To attain higher-order temporal accuracy, one may instead employ SSP RK or multi-step methods~\cite{gottlieb2001strong}.
Due to their essence as convex combinations of forward Euler steps, SSP methods preserve the PP and globally DF properties.
\section{Numerical experiments}\label{sec:exp}

This section presents a series of benchmark and challenging tests to demonstrate the accuracy and robustness of the proposed PosDiv-CDG methods.
We consider the third-order PosDiv-CDG scheme with polynomial degree $k=2$, coupled with the explicit third-order SSP RK time discretization~\cite{gottlieb2001strong}.
Unless otherwise specified, the ideal EOS $p=(\gamma - 1)\rho e$ with $\gamma = 5/3$ is used.
The CFL number is set to $0.2$, and the parameter $\theta = \Delta t/\tau_{\max}$ is set to $1$.

\begin{Exp}[Smooth Alfv\'en Wave]\label{Exp:alfven}
\rm	We first consider the smooth Alfv\'en wave problem~\cite{toth2000b,li2008high}, involving a circularly polarized wave propagating at an angle $\alpha = \pi/4$ to the $x$-axis over the domain $[0,1/\cos\alpha] \times [0,1/\sin\alpha]$. The initial conditions are
	\begin{align*}
		(\rho,~u_{\|},~u_{\bot},~u_3,~B_{\|},~B_{\bot},~B_3,~p) = (1,~0,~0.1\sin(2\pi\beta),~0.1\cos(2\pi\beta),~1,~u_{\bot},~u_3,~0.1),
	\end{align*}
	where $\beta = x\cos\alpha + y\sin\alpha$, and the subscripts $\|$ and $\bot$ denote directions parallel and perpendicular to the wave propagation. Periodic boundary conditions are applied. The wave travels toward the origin $(0,0)$ at a constant Alfv\'en speed $B_{\|}/\sqrt{\rho} = 1$, returning to its initial state at $t = 1$. The simulation is run up to this time.
	Table~\ref{table:order-alfven} reports the $l^1$, $l^2$, and $l^\infty$ errors in $B_1$ along with the convergence rates, demonstrating that the third-order PosDiv-CDG scheme achieves the expected accuracy. This implies the COE procedure and modified dissipation terms maintain the high-order accuracy.
\end{Exp}

\begin{table}[!thb]
	\vspace*{0.2cm}
	\small
	\caption{Example~\ref{Exp:alfven}: Numerical errors in $B_1$ at $t=1$ and corresponding convergence orders of the third-order PosDiv-CDG scheme.}
	\vspace*{-0.2cm}
	\label{table:order-alfven}
	\medskip
	\centering
	\begin{tabular}{ccccccc}
		\toprule
		$N_x\times N_y$ \quad
		&  $l^1$ error      \quad &   order \quad
		&  $l^2$ error \quad &   order \quad
		&  $l^\infty$ error      \quad &   order \quad \\
		\midrule
		$20 \times 20$   & 6.331E-04 & {---}   & 4.962E-04 & {---}   & 5.047E-04 & {---}   \\
		$40 \times 40$   & 7.501E-05 & 3.077   & 5.883E-05 & 3.076   & 6.000E-05 & 3.072   \\
		$80 \times 80$   & 9.216E-06 & 3.025   & 7.231E-06 & 3.024   & 7.371E-06 & 3.025   \\
		$160 \times 160$ & 1.146E-06 & 3.007   & 8.994E-07 & 3.007   & 9.163E-07 & 3.008   \\
		$320 \times 320$ & 1.430E-07 & 3.002   & 1.122E-07 & 3.002   & 1.143E-07 & 3.003   \\
		$640 \times 640$ & 1.787E-08 & 3.001   & 1.402E-08 & 3.001   & 1.427E-08 & 3.001   \\
		\bottomrule
	\end{tabular}
\end{table}

\begin{Exp}[Smooth Vortex Problem]\label{Exp:vortex1}
	\rm
This example further assesses the accuracy of the proposed method using a smooth isentropic vortex with low pressure, following the setup in~\cite{christlieb2015positivity,wu2018provably}. The initial conditions are
\begin{align*}
	(\rho,~\bm u,~\bm B,~p) = (1,~1+\delta u_1,~1+\delta u_2,~0,~\delta B_1,~\delta B_2,~0,~1+\delta p),
\end{align*}
with perturbations
$(\delta u_1,\delta u_2) = \frac{\mu}{\sqrt{2\pi}}e^{0.5(1-r^2)}(-y,x)$,
$(\delta B_1,\delta B_2) = \frac{\mu}{2\pi}e^{0.5(1-r^2)}(-y,x)$,
$\delta p = -\frac{\mu^2(1+r^2)}{8\pi^2}e^{1-r^2}$,
where $r = \sqrt{x^2 + y^2}$ and $\mu = 5.389489439$ denotes the vortex strength.
The simulation runs up to $t = 0.05$ in the domain $[-10, 10] \times [-10,10]$ with periodic boundaries. Table~\ref{table:order-vortex} lists the $l^1$ errors and convergence orders for  $m_1$, $B_1$, and $E$ using the third-order PosDiv-CDG scheme.
The results confirm that the scheme achieves the optimal convergence order, demonstrating that our PP technique, modified dissipation terms, and COS procedure preserve the high-order accuracy.

\end{Exp}

\begin{table}[!thb]
    \vspace*{0.2cm}
    \small
    \caption{Example~\ref{Exp:vortex1}: $l^1$ errors in ${m}_1$, ${B}_1$, and ${E}$ at $t = 0.05$, along with the corresponding convergence orders for the third-order PosDiv-CDG scheme.}
    \vspace*{-0.2cm}
    \label{table:order-vortex}
    \medskip
    \centering
    \begin{tabular}{ccccccc}
    \toprule
    \multicolumn{1}{c}{\multirow{2}*{$N_x\times N_y$}} & \multicolumn{2}{c}{$m_1$} & \multicolumn{2}{c}{$B_1$} & \multicolumn{2}{c}{$E$}\\
    \cmidrule(lr){2-3} \cmidrule(lr){4-5} \cmidrule(lr){6-7}
                  &$l^1$ error   &order    &$l^1$ error    &order      &$l^1$ error    &order\\
    \midrule
    $20 \times 20$   & 3.570E-01 & {---}   & 2.829E-01 & {---}   & 7.426E-01 & {---}   \\
    $40 \times 40$   & 4.854E-02 & 2.879   & 3.193E-02 & 3.147   & 1.040E-01 & 2.836   \\
    $80 \times 80$   & 5.790E-03 & 3.068   & 3.768E-03 & 3.083   & 1.163E-02 & 3.161   \\
    $160 \times 160$ & 5.557E-04 & 3.381   & 4.499E-04 & 3.066   & 1.090E-03 & 3.416   \\
    $320 \times 320$ & 5.765E-05 & 3.269   & 5.478E-05 & 3.038   & 1.204E-04 & 3.178   \\
    $640 \times 640$ & 6.643E-06 & 3.118   & 6.766E-06 & 3.017   & 1.432E-05 & 3.072   \\
    \bottomrule
    \end{tabular}
\end{table}

\begin{Exp}[Orszag--Tang Problem]\label{Exp:OrszagTang}
\rm
This benchmark is widely used to assess the ability of numerical schemes to resolve complex MHD dynamics. Although the initial state is smooth, nonlinear interactions among MHD waves quickly generate a range of discontinuities and fine-scale structures.
The initial conditions are
\begin{align*}
	(\rho,~\bm u,~\bm B,~p) = (\gamma^2,\ -\sin y,\ \sin x,\ 0,\ -\sin y,\ \sin 2x,\ 0,\ \gamma)
\end{align*}
on the domain $[0, 2\pi] \times [0, 2\pi]$, discretized with $200 \times 200$ uniform cells and periodic boundary conditions.
\Cref{Fig:Orszag-Tang} shows the density contours at $t = 0.5$, $2$, and $3$ computed with the third-order PosDiv-CDG method. As the solution evolves, complex structures emerge, including interacting shocks and turbulence-like patterns. The results agree well with those from  \cite{li2011central, li2012arbitrary, wu2023provably}, demonstrating the robustness and effectiveness of the proposed PosDiv-CDG method.
\end{Exp}

\begin{figure}[!thb]
    \vspace*{-0.2cm}
	\centering
    {	
        \includegraphics[width=0.31\textwidth]{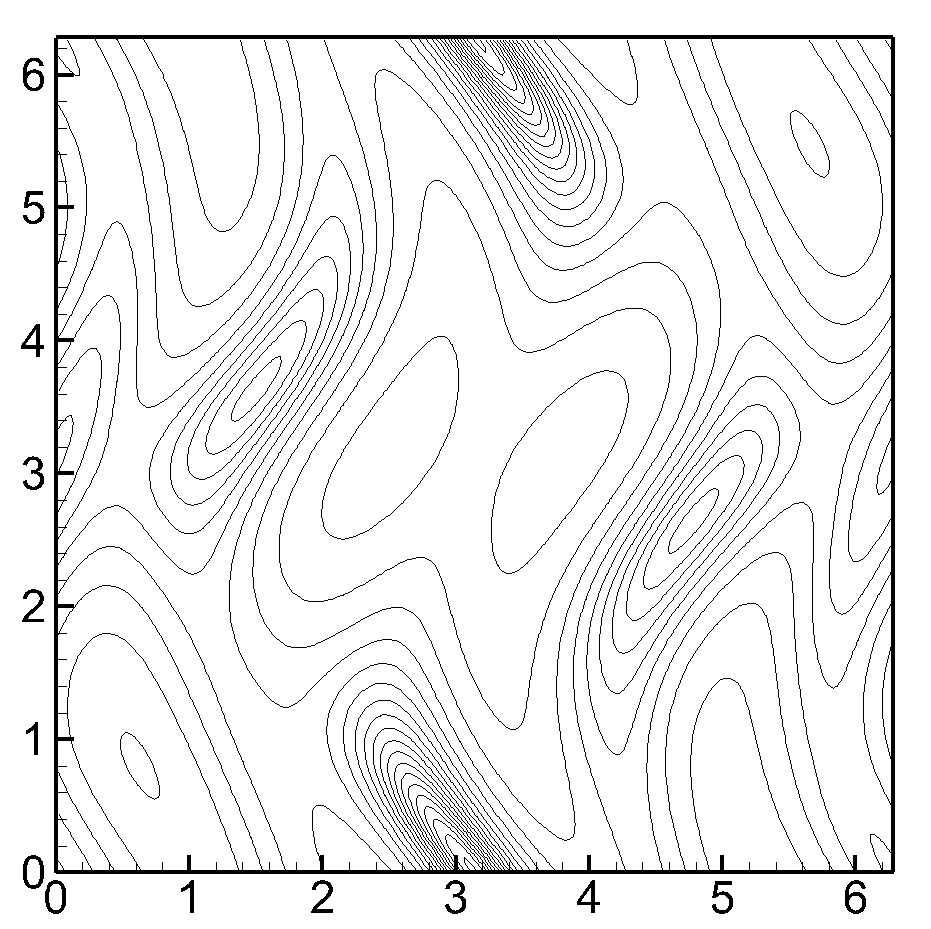}
    }
    {	
        \includegraphics[width=0.31\textwidth]{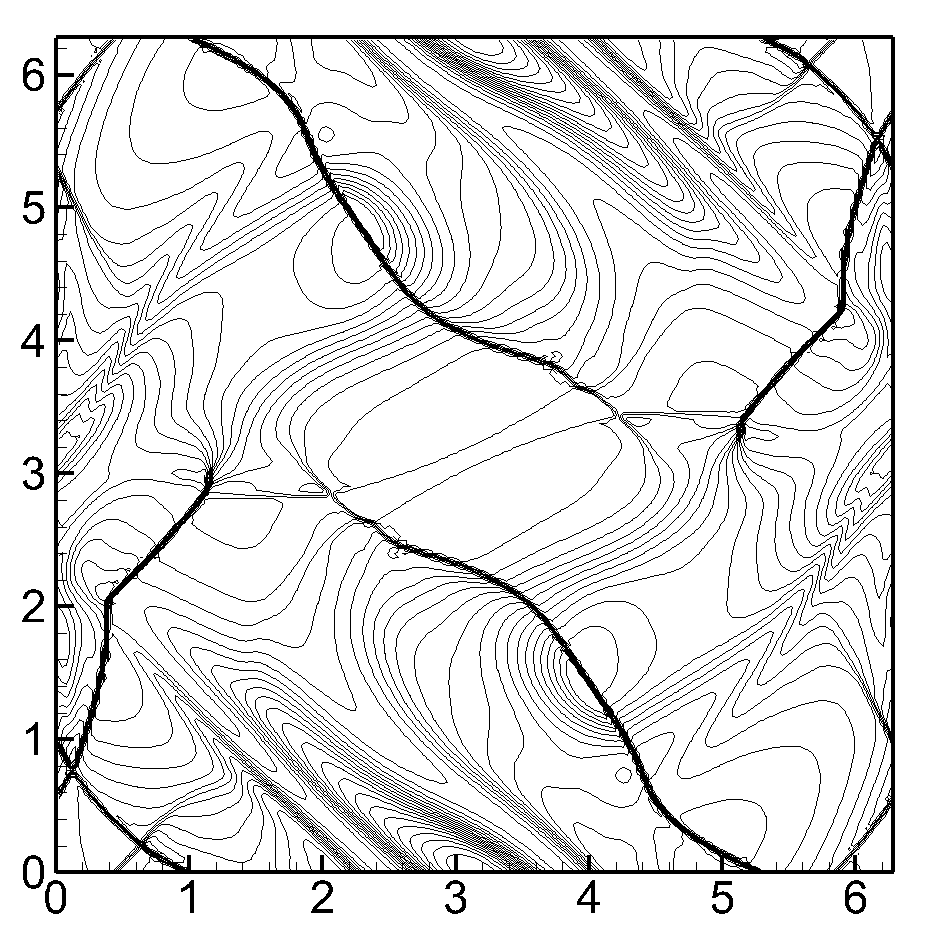}
    }
    {	
        \includegraphics[width=0.31\textwidth]{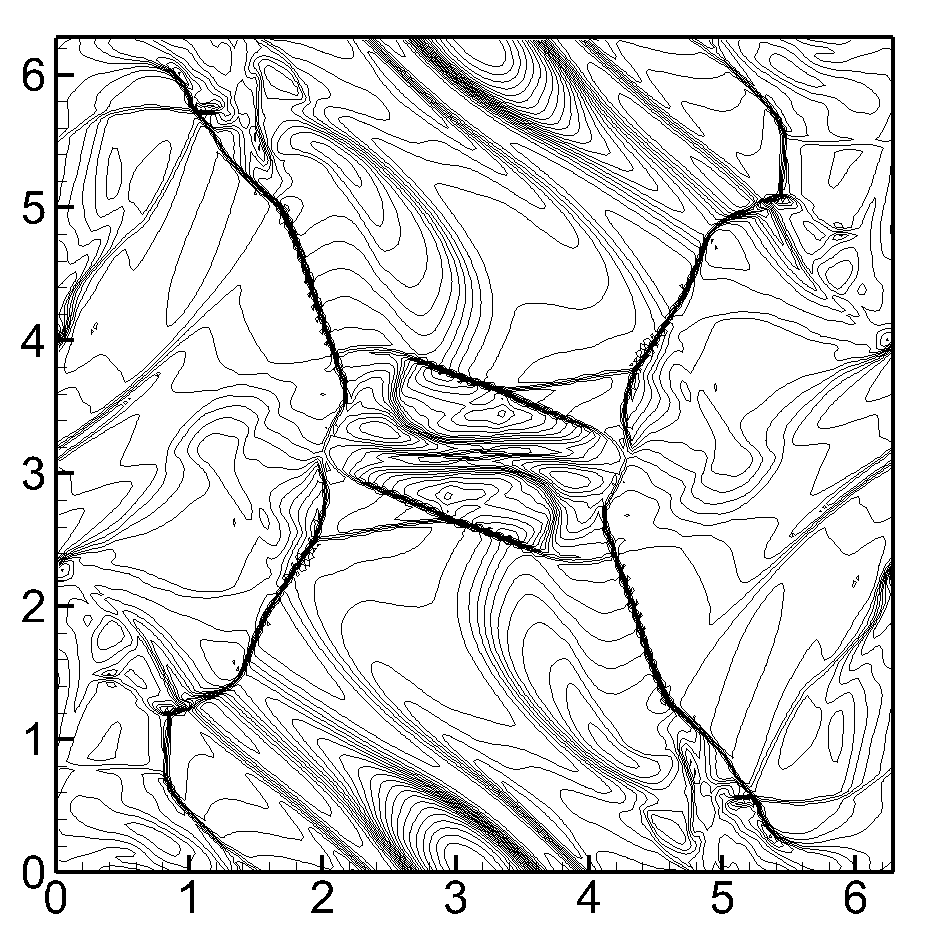}
    }
	\caption{\small Example \ref{Exp:OrszagTang}: Contour plots of  $\rho$ at $t=0.5$, $t=2$, and $t=3$ (from left to right).}
    \label{Fig:Orszag-Tang}
\end{figure}

\begin{Exp}[Rotor Problem]\label{Exp:Rotor}
\rm
This widely used benchmark \cite{balsara1999staggered} involves a dense, rapidly rotating disk. The initial conditions are
\begin{align*}
\big(\rho,~\bm u,~\bm B,~p\big) =
\begin{cases}
\big(10,~-\frac{y-0.5}{r_1},~\frac{x-0.5}{r_1},~0,~\frac{2.5}{\sqrt{4\pi}},~0,~0,~0.5\big),& r\leq r_1,\\
\big(1+9\phi,~-\frac{\phi(y-0.5)}{r},~\frac{\phi(x-0.5)}{r},~0,~\frac{2.5}{\sqrt{4\pi}},~0,~0,~0.5\big),& r_1<r\leq r_2,\\
\big(1,~0,~0,~0,~\frac{2.5}{\sqrt{4\pi}},~0,~0,~0.5\big),& r_2<r,
\end{cases}
\end{align*}
where $r = \sqrt{(x-0.5)^2 + (y-0.5)^2}$, $r_1 = 0.1$, $r_2 = 0.115$, and $\phi = (r_2 - r)/(r_2 - r_1)$.
The simulation is run on the domain $[0,1] \times [0,1]$ using $400 \times 400$ uniform cells and periodic boundary conditions. \Cref{Fig:Rotor problem} shows the computed contours of density $\rho$, pressure $p$, and Mach number $\abs{\bm u} / c$ at $t = 0.295$, where $c = \sqrt{\gamma p / \rho}$ is the local sound speed. The results exhibit sharp, well-resolved features and align well with those in~\cite{balsara1999staggered,li2011central, li2012arbitrary, wu2023provably}, confirming the high resolution of the proposed PosDiv-CDG scheme.
Thanks to the proposed COS procedure, the numerical solution is free of spurious oscillations.
\end{Exp}

\begin{figure}[!thb]
    \vspace*{-0.2cm}
	\centering
	{
		\includegraphics[width=0.31\linewidth]{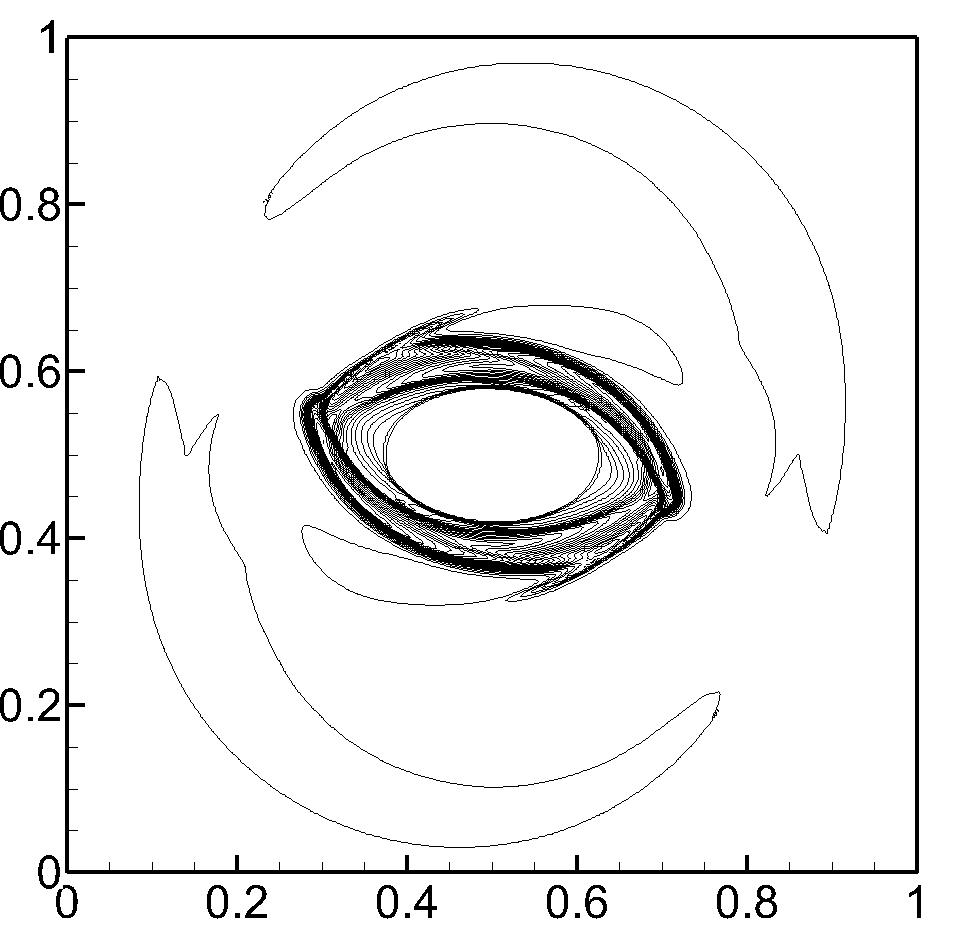}
    }
	{
		\includegraphics[width=0.31\linewidth]{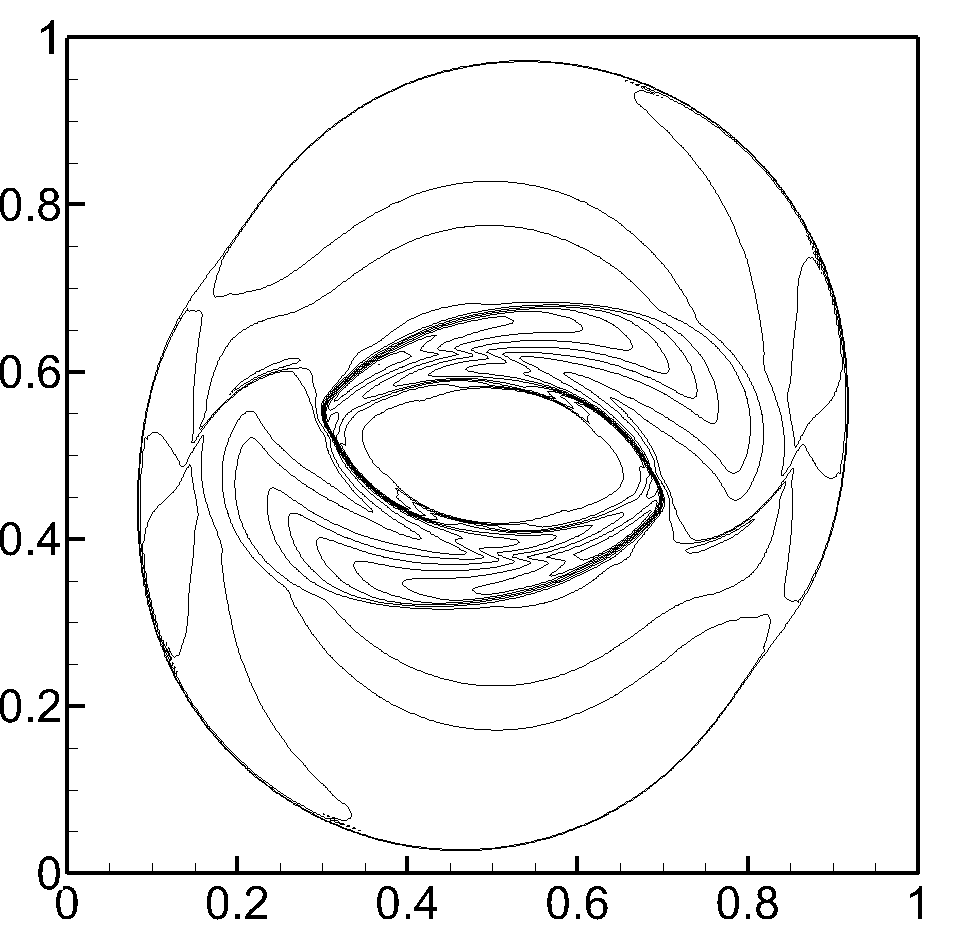}
	}
	{
		\includegraphics[width=0.31\linewidth]{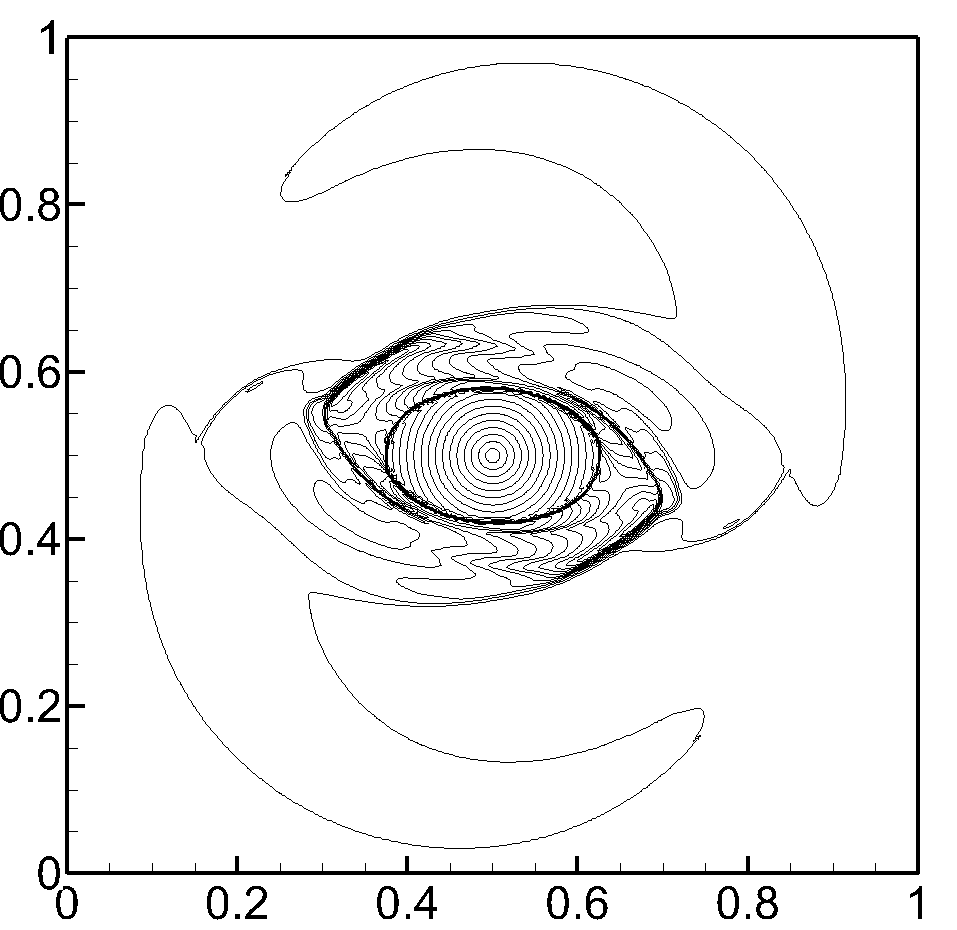}
	}
 \caption{\small Example~\ref{Exp:Rotor}: Contour plots of $\rho$ (left), $p$ (center), and $\abs{\bm u} / c$ (right) at $t = 0.295$.}
	\label{Fig:Rotor problem}
\end{figure}

\begin{Exp}[Blast Problems]\label{Exp:Blast}
\rm
This example verifies the PP property of the PosDiv-CDG scheme and demonstrates the effectiveness of our entropy-induced COS procedure in suppressing nonphysical oscillations. We consider three challenging blast problems~\cite{balsara1999staggered, li2011central, wu2018provably}, characterized by extremely low gas pressure and strong magnetosonic shocks. The initial conditions are
\begin{align*}
\big(\rho,~\bm u,~\bm B,~p\big) =
\begin{cases}
\big( 1,~0,~0,~0,~B_x,~B_y,~0,~p_0 \big), & \mbox{if }\sqrt{x^2+y^2} \leq 0.1,\\
\big( 1,~0,~0,~0,~B_x,~B_y,~0,~0.1 \big), & \text{otherwise},
\end{cases}
\end{align*}
with $\gamma = 1.4$. The computational domain $[-0.5, 0.5] \times [-0.5, 0.5]$ is discretized using $320 \times 320$ uniform cells, and outflow boundary conditions are applied on all sides. We test the following three configurations:
\begin{enumerate}[label=(\roman*)]
  \item $B_x = 100/\sqrt{4\pi}$, $B_y = 0$ (plasma-beta $\beta \approx 2.51\times10^{-4}$), and $p_0 = 10^3$;
  \item $B_x = 1000/\sqrt{4\pi}$, $B_y = 0$ (plasma-beta $\beta \approx 2.51\times10^{-6}$), and $p_0 = 10^4$;
  \item $B_x = B_y = 100/\sqrt{8\pi}$ (plasma-beta $\beta \approx 2.51\times10^{-4}$), and $p_0 = 10^3$.
\end{enumerate}
\Cref{Fig:Blast problem} shows contour plots of the density $\rho$, thermal pressure $p$, and magnetic pressure $p_m$ at $t=0.01$, $t=0.001$, and $t=0.01$ for the three cases, respectively. This test is highly sensitive to the oscillation control strategy. While the oscillation-eliminating technique developed in~\cite{liu2025structure,PSW} performs well in locally DF DG frameworks, it does not fully suppress small nonphysical oscillations when applied directly in globally DF CDG methods.
In contrast, our entropy-induced COS procedure significantly enhances robustness by effectively eliminating nonphysical oscillations and preserving the physical realism of the solution. The numerical results show good agreement with those in~\cite{li2011central, christlieb2015positivity, wu2018provably}, demonstrating the reliability and effectiveness of our PosDiv-CDG method.
\end{Exp}

\begin{figure}[!thb]
    \vspace*{-0.2cm}
	\centering  
    {
		\includegraphics[width=0.31\linewidth]{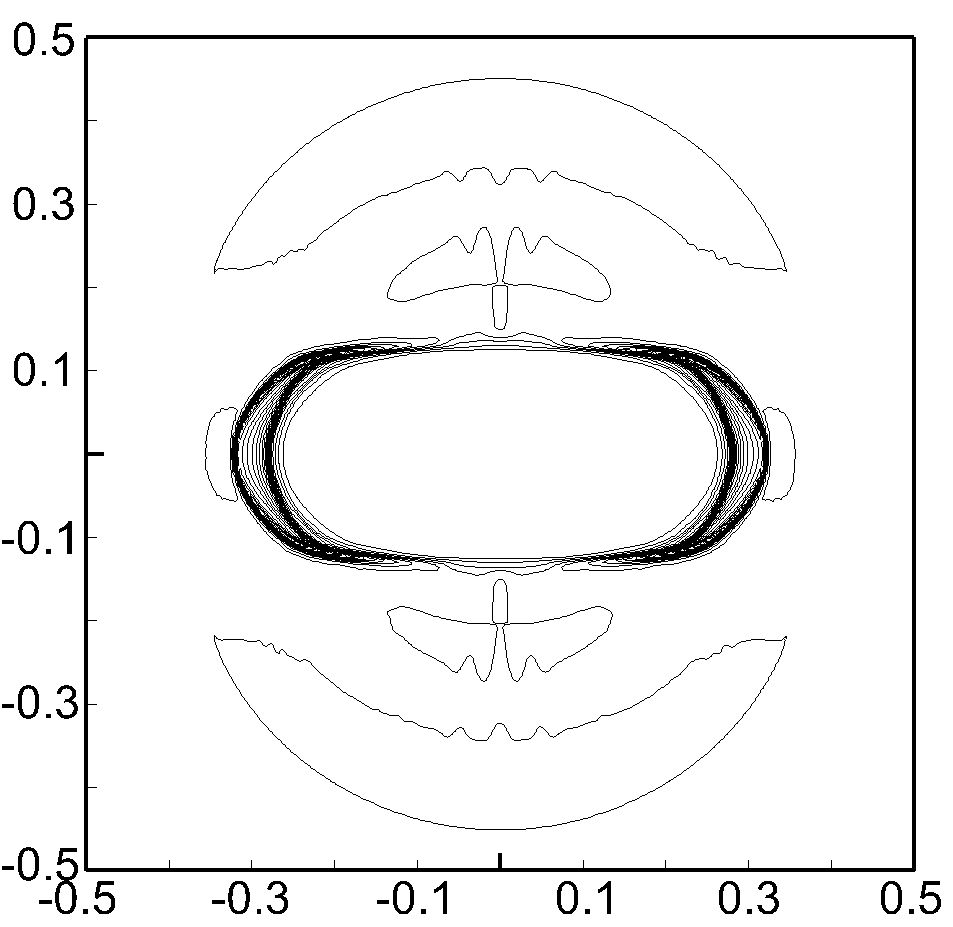}
    }
    {
		\includegraphics[width=0.31\linewidth]{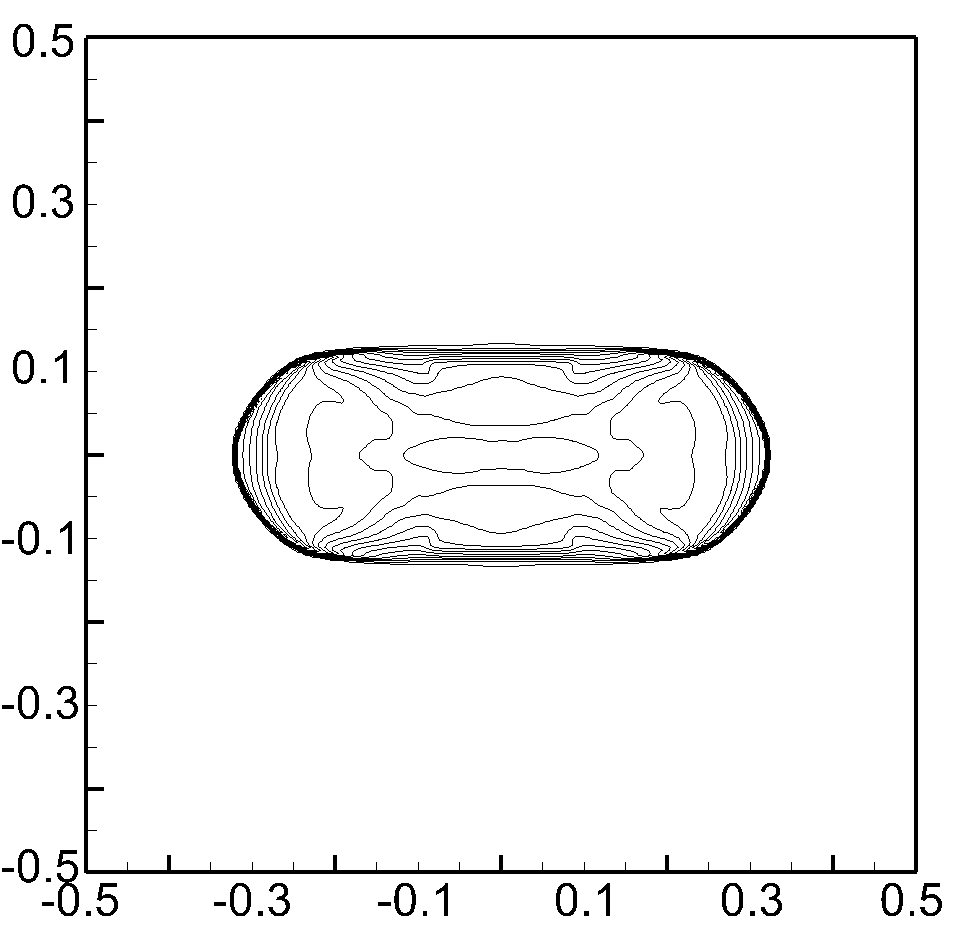}
    }
    {
		\includegraphics[width=0.31\linewidth]{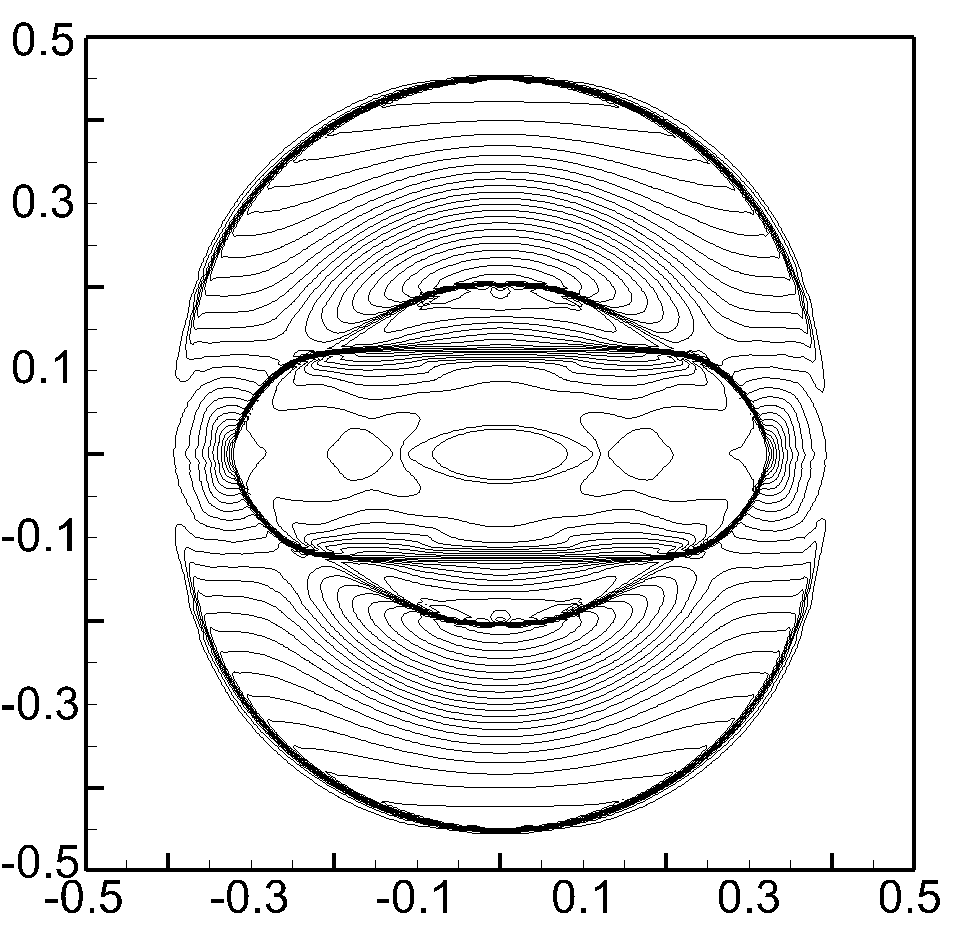}
    }

    {
		\includegraphics[width=0.31\linewidth]{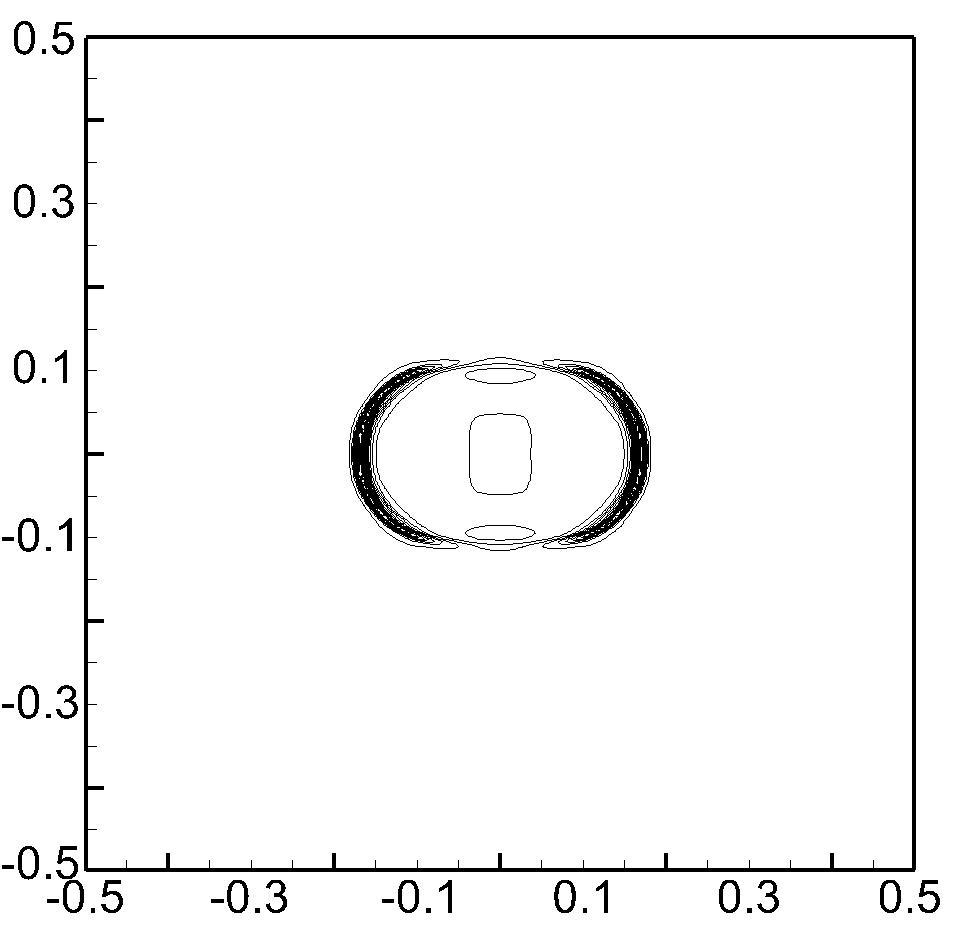}
    }
    {
		\includegraphics[width=0.31\linewidth]{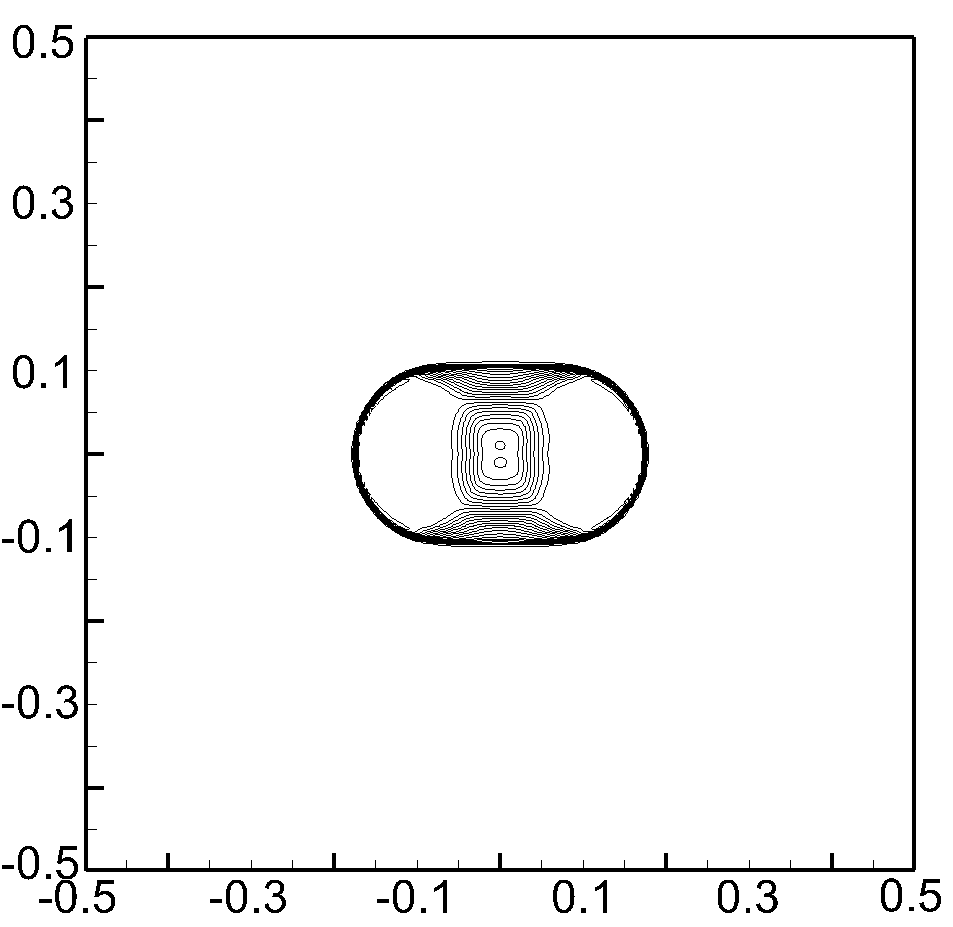}
    }
    {
		\includegraphics[width=0.31\linewidth]{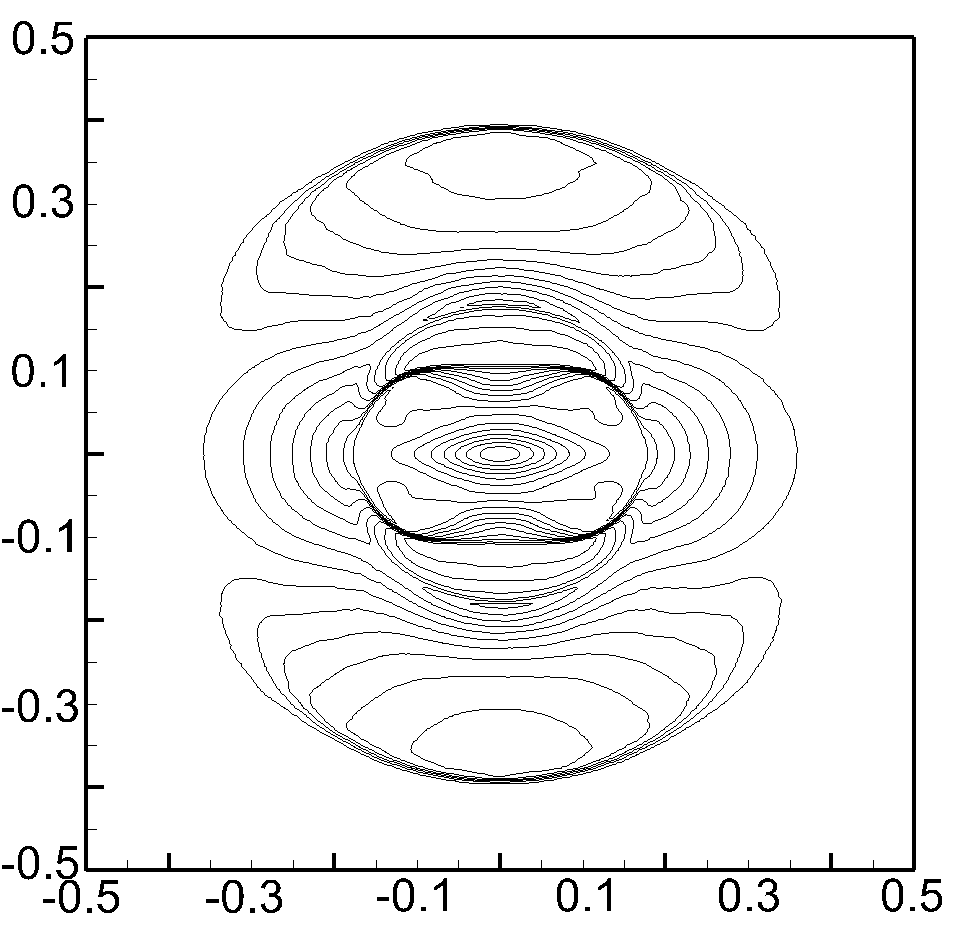}
    }

    {
		\includegraphics[width=0.31\linewidth]{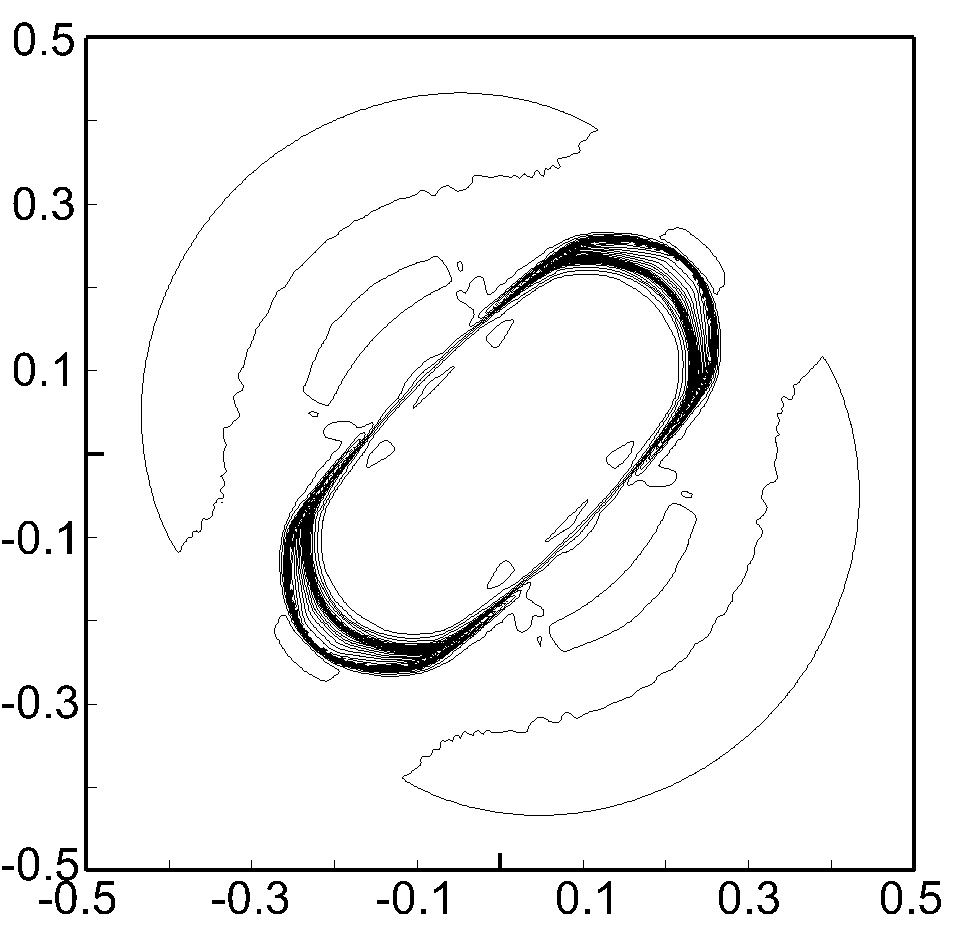}
    }
    {
		\includegraphics[width=0.31\linewidth]{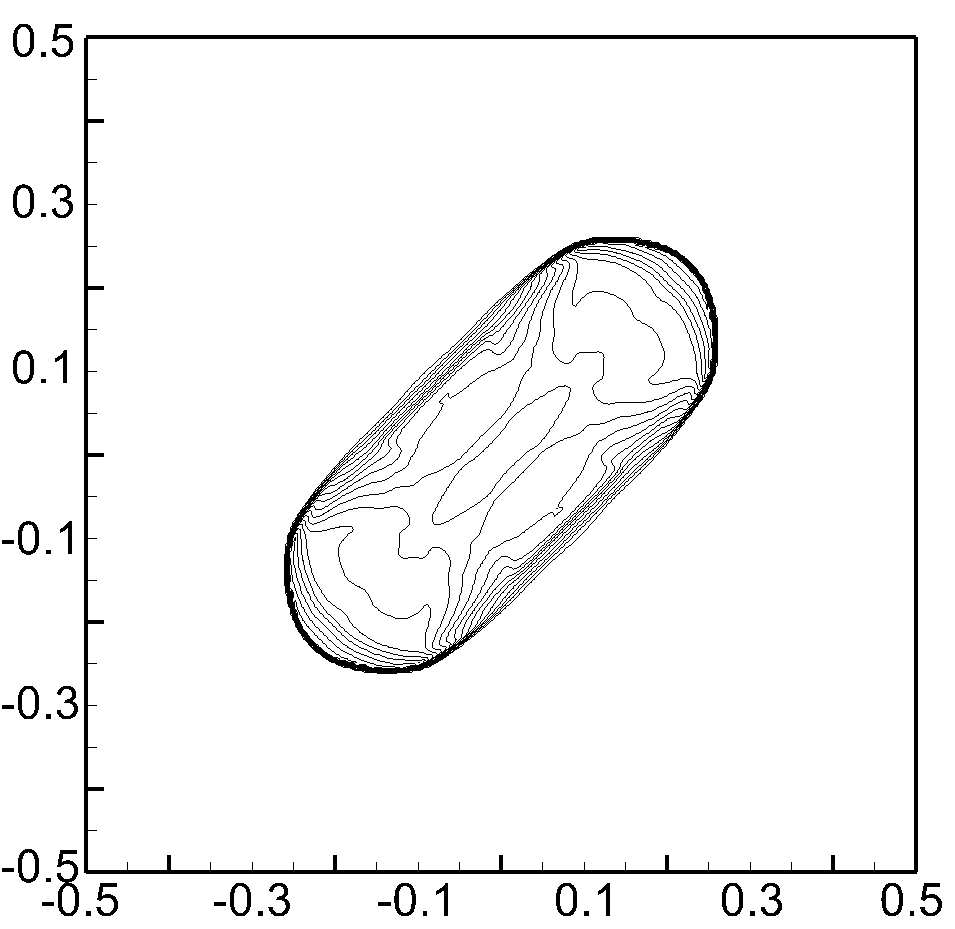}
    }
    {
		\includegraphics[width=0.31\linewidth]{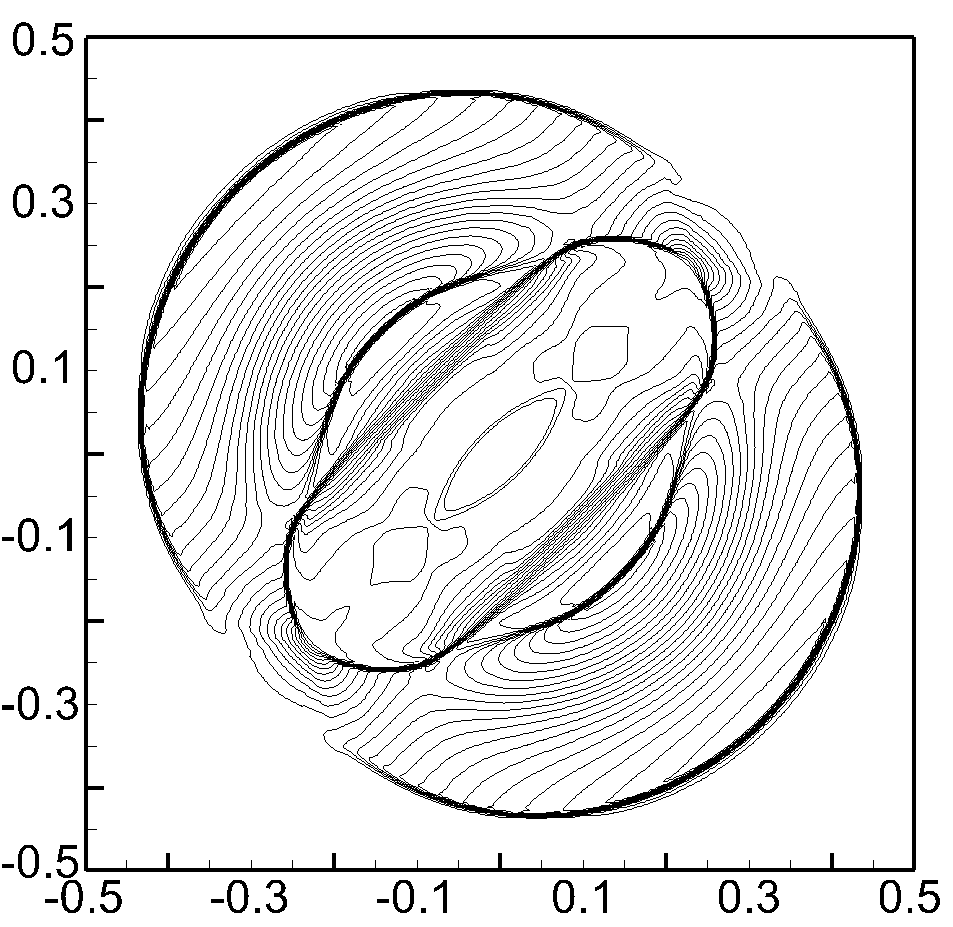}
    }
	\caption{\small Example \ref{Exp:Blast}: Contour plots of $\rho$ (left), $p$ (center),
and $p_m$ (right) for the three configurations (from top to bottom) at $t=0.01$, $t=0.001$, and $t=0.01$, respectively.}
	\label{Fig:Blast problem}
\end{figure}

\begin{Exp}[Shock-Cloud Interaction]\label{Exp:ShockCloud}
\rm
This test examines the robustness of our PosDiv-CDG scheme under the interaction of a strong shock with a dense cloud, following the setup in~\cite{toth2000b,balbas2006nonoscillatory}. The initial conditions are
\begin{align*}
\big(\rho,~\bm u,~\bm B,~p\big) =
\begin{cases}
    \big(3.86859,~0,~0,~0,~0,~2.1826182,~-2.1826182,~167.345\big), & x<0.6, \\
    \big(1,~-11.2536,~0,~0,~0,~0.56418958,~0.56418958,~1\big), & x>0.6.
\end{cases}
\end{align*}
The computational domain $[0,1] \times [0,1]$ is discretized using $300 \times 300$ uniform cells. An inflow condition is imposed at the right boundary, while outflow conditions are applied elsewhere. A high-density circular cloud with $\rho = 10$ is placed in the post-shock region, centered at $(0.8,\, 0.5)$ with radius $0.15$. The simulation runs until $t = 0.06$.
\Cref{Fig:ShockCloud problem} shows the contour plots of the density $\rho$, thermal pressure $p$, and magnetic pressure $p_m$ obtained with the third-order PosDiv-CDG scheme. This test highlights the importance of carefully handling boundary conditions to maintain the global DF property. Improper treatment can introduce divergence errors that propagate and accumulate, potentially leading to instability or failure of the simulation. The proposed PosDiv-CDG method demonstrates strong robustness in this test.
\end{Exp}

\begin{figure}[!thb]
    \vspace*{-0.2cm}
	\centering  
	{
		\includegraphics[width=0.31\linewidth]{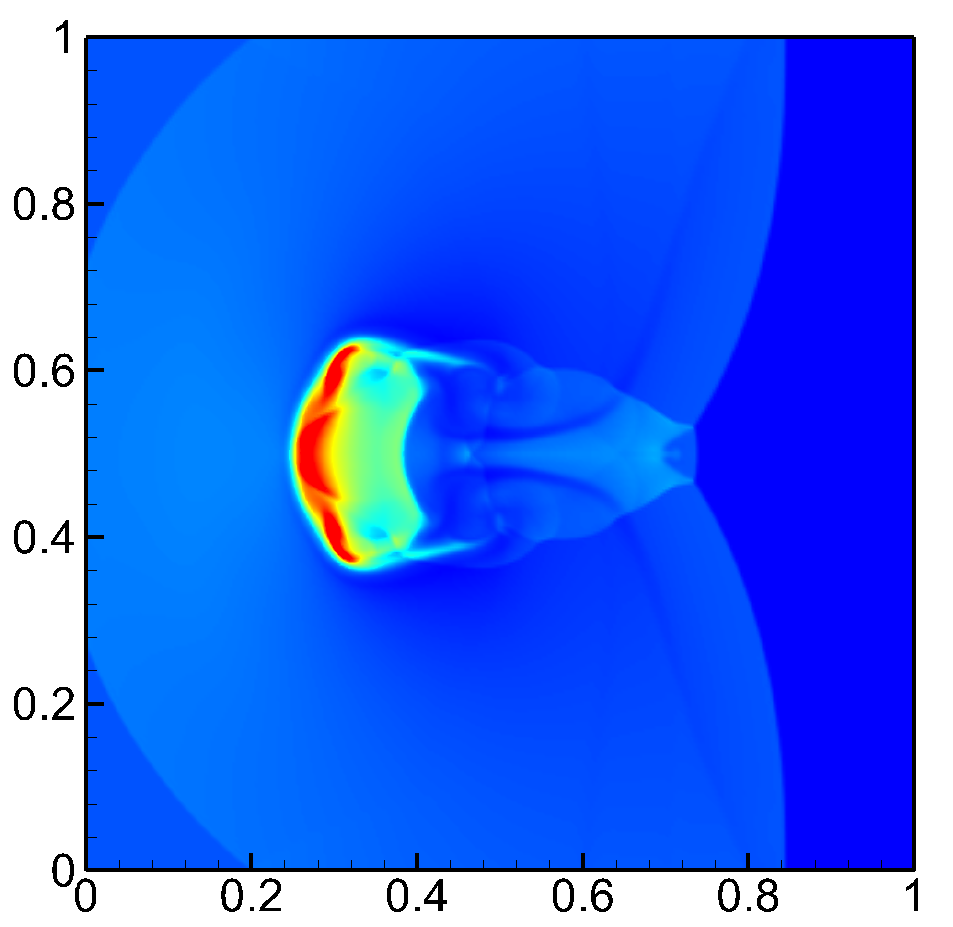}
    }
	{
		\includegraphics[width=0.31\linewidth]{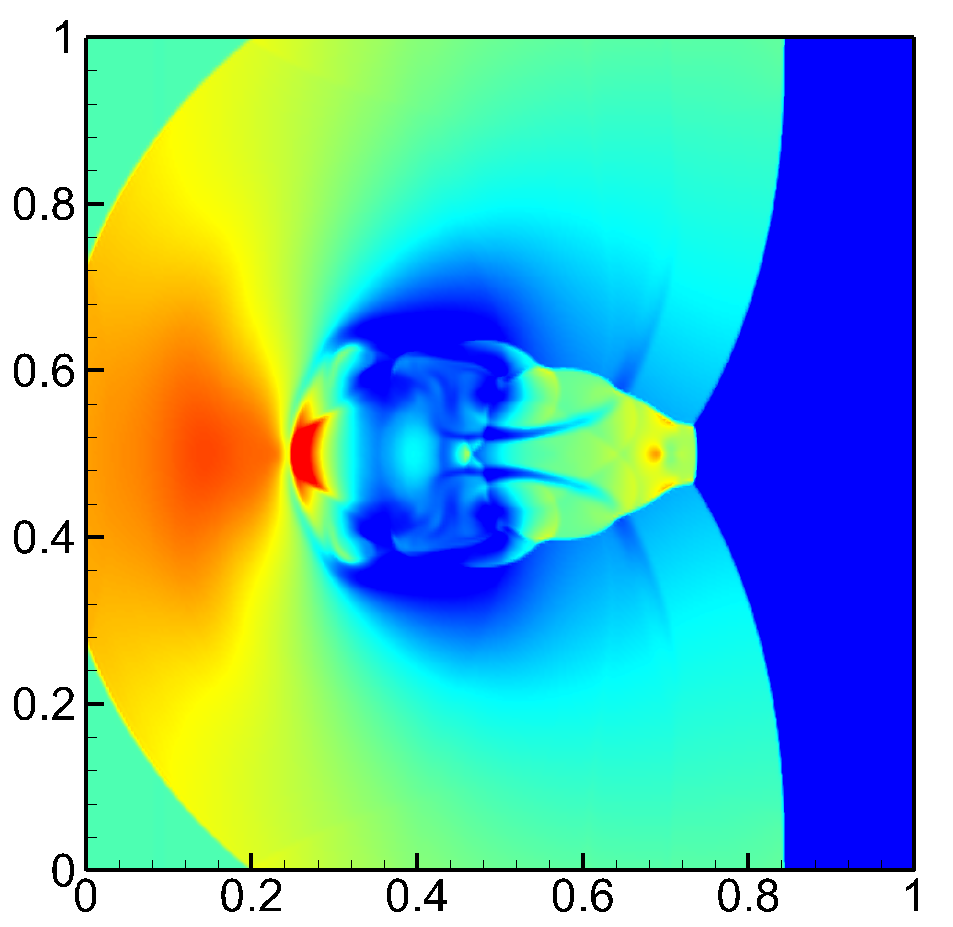}
    }
	{
		\includegraphics[width=0.31\linewidth]{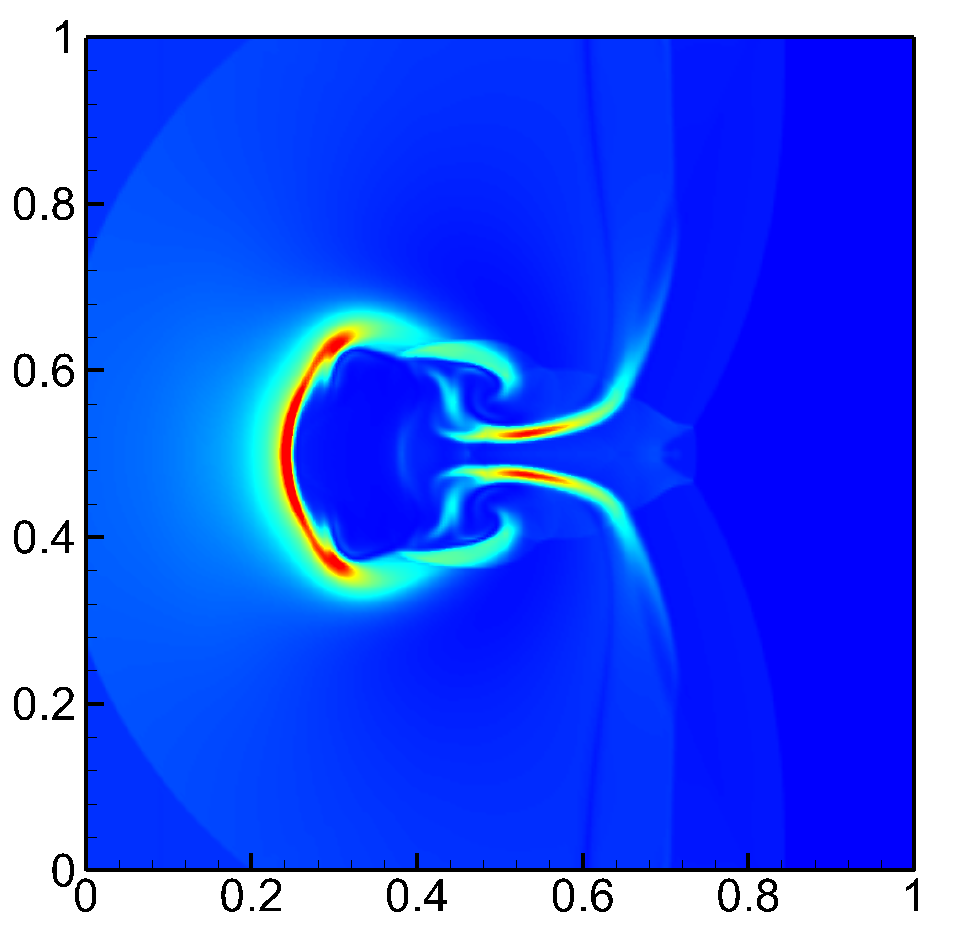}
    }
	\caption{\small Example \ref{Exp:ShockCloud}: Contour plots of density $\rho$ (left), pressure $p$ (center), and magnetic pressure $p_m$ (right) at $t=0.06$.}
	\label{Fig:ShockCloud problem}
\end{figure}

\begin{figure}[!thb] 
	\vspace*{-0.2cm}
	\centering  
	{
		\includegraphics[width=0.31\linewidth]{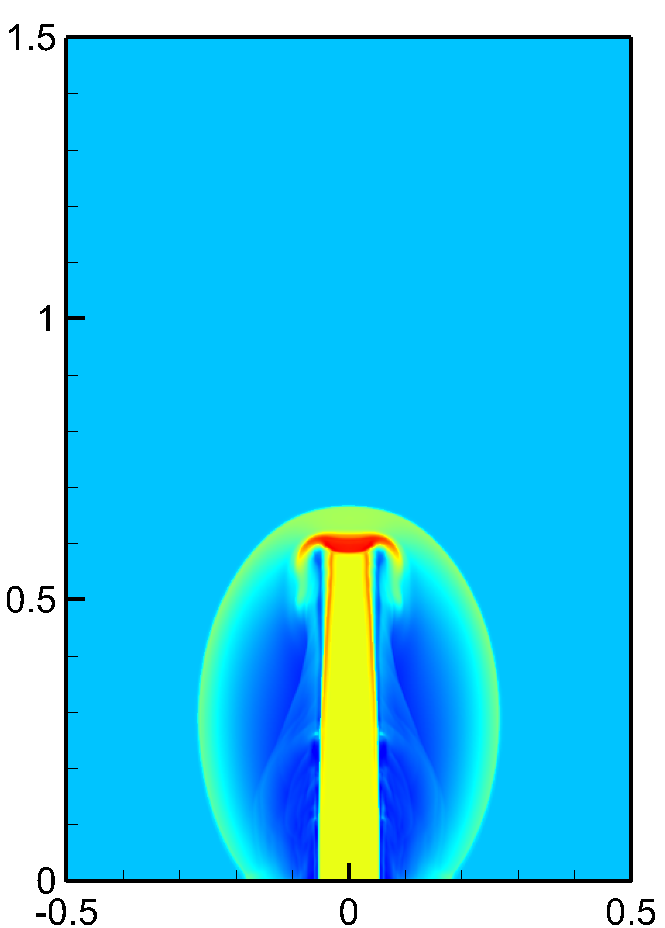}
	}
	{
		\includegraphics[width=0.31\linewidth]{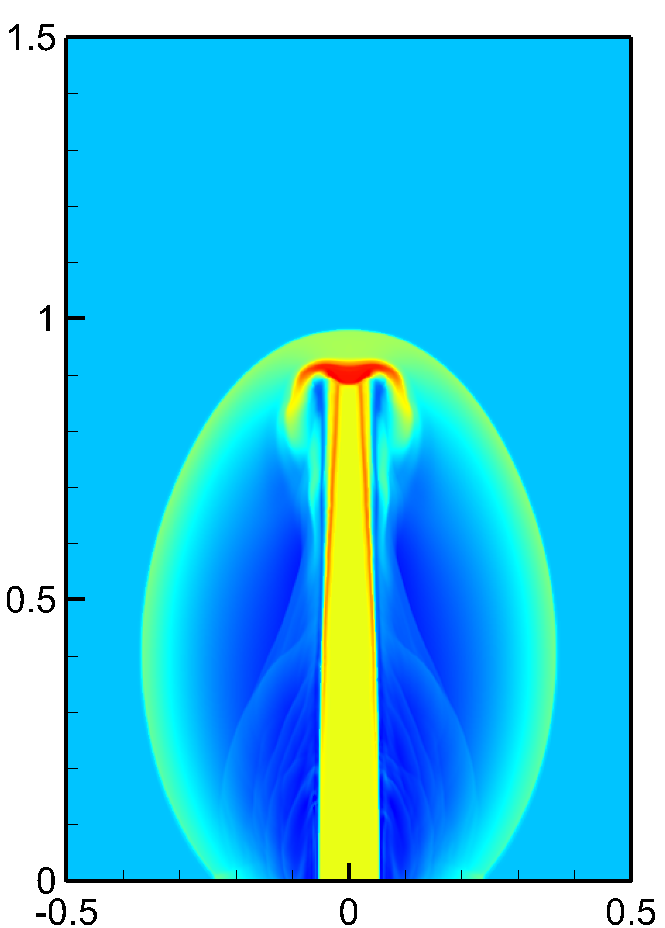}
	}
	{
		\label{fig:Jet1_rho_t3}
		\includegraphics[width=0.31\linewidth]{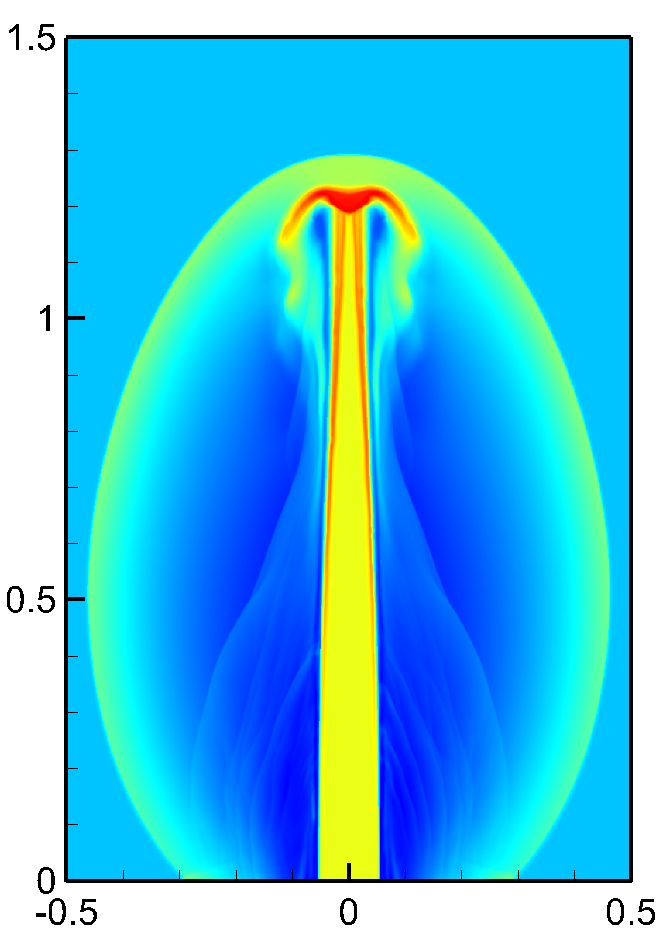}
	}
	\caption{\small Example \ref{Exp:Jets}: Density logarithm for the Mach 800
		jet problem with $B_0 = \sqrt{200}$ at $t = 0.001$, 0.0015 and 0.002 (from left to right).}
	\label{Fig:Jet1}
\end{figure}

\begin{Exp}[Astrophysical Jets]\label{Exp:Jets}
\rm
This example tests several MHD jet problems to further assess the robustness and PP property of the proposed PosDiv-CDG method. These simulations are particularly challenging due to the presence of extremely high Mach numbers and strong magnetic fields, which can easily result in negative pressure if the PP technique is not employed.
We begin with the Mach 800 jet problem \cite{WuSINUM2018,wu2018provably}, using the initial conditions
\begin{align*}
    \big(\rho,~\bm u,~\bm B,~p\big) = \big(0.1\gamma,~0,~0,~0,~0,~B_0,~0,~1\big),  \qquad \gamma = 1.4.
\end{align*}
Outflow boundary conditions are applied on all sides, except for an inflow segment along the bottom boundary where $\abs{x} \leq 0.05$:
\begin{align*}
    \big(\rho,~\bm u,~\bm B,~p\big)_{y=0,~\abs{x}\leq0.05} = \big(\gamma,~0,~800,~0,~0,~B_0,~0,~1\big).
\end{align*}
Improper treatment of boundary conditions can introduce non-zero divergence errors, potentially compromising the globally DF property.
As the magnetic field strength $B_0$ increases, the problem becomes more challenging due to increased magnetic pressure and reduced plasma-beta. We test three magnetized cases with $\beta \in \{ {10^{-2}, 10^{-3}, 10^{-4}} \}$, corresponding to $B_0 = \sqrt{200}, \sqrt{2000}, \sqrt{20,000}$, respectively.
The domain $[-0.5, 0.5]\times [0, 1.5]$ is discretized using $400 \times 600$ uniform cells. \Cref{Fig:Jet1,Fig:Jet2,Fig:Jet3} present the numerical results at $t=0.001$, $t=0.0015$, and $t=0.002$, respectively. The solution structures differ significantly with increasing $B_0$, illustrating our method's robustness and its ability to resolve strongly magnetized MHD flows. In addition, the PosDiv-CDG method maintains the symmetry of the jet profile.

To further highlight the method's capability, we simulate even more extreme jets with $B_0 = \sqrt{20,000}$ and Mach numbers of 2000 and 1,000,000. Schlieren images of the logarithmic density are shown in \Cref{Fig:Jet4,Fig:Jet7} at $t = 0.00075$ and $t = 0.0000015$, respectively. As expected, higher Mach numbers yield narrower jets, consistent with observations in \cite{wu2019provably}. To our knowledge, the {\bf Mach 1,000,000 jet} with $B_0 = \sqrt{20,000}$ represents the first successful simulation of such an extreme MHD case, which was not
addressed by any existing globally DF CDG or DG methods \cite{li2011central, li2012arbitrary, fu2018globally}.
\end{Exp}

\begin{figure}[!thb] 
    \vspace*{-0.2cm}
	\centering  
    {
		\includegraphics[width=0.31\linewidth]{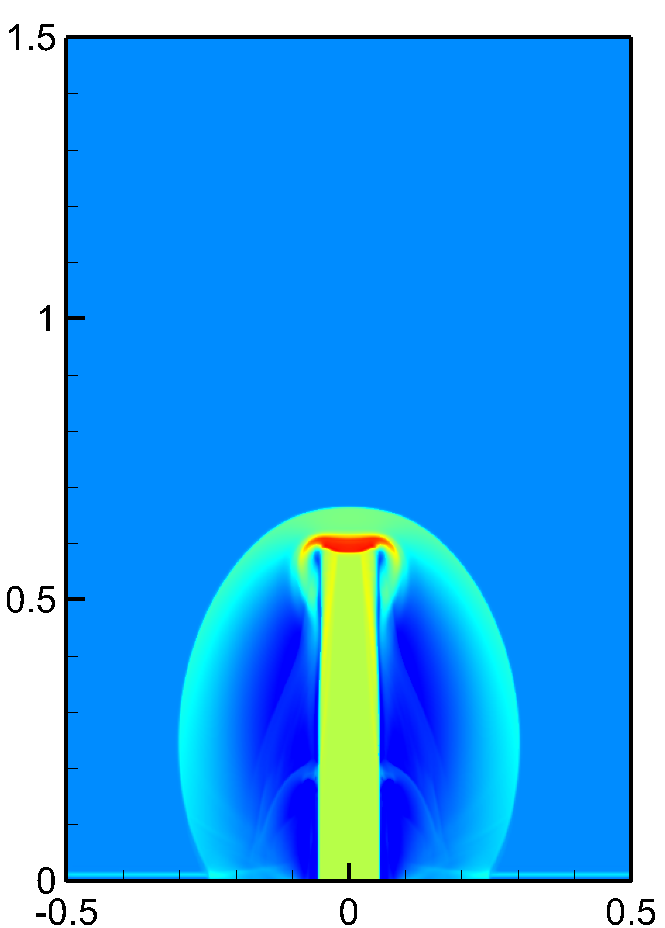}
    }
    {
		\includegraphics[width=0.31\linewidth]{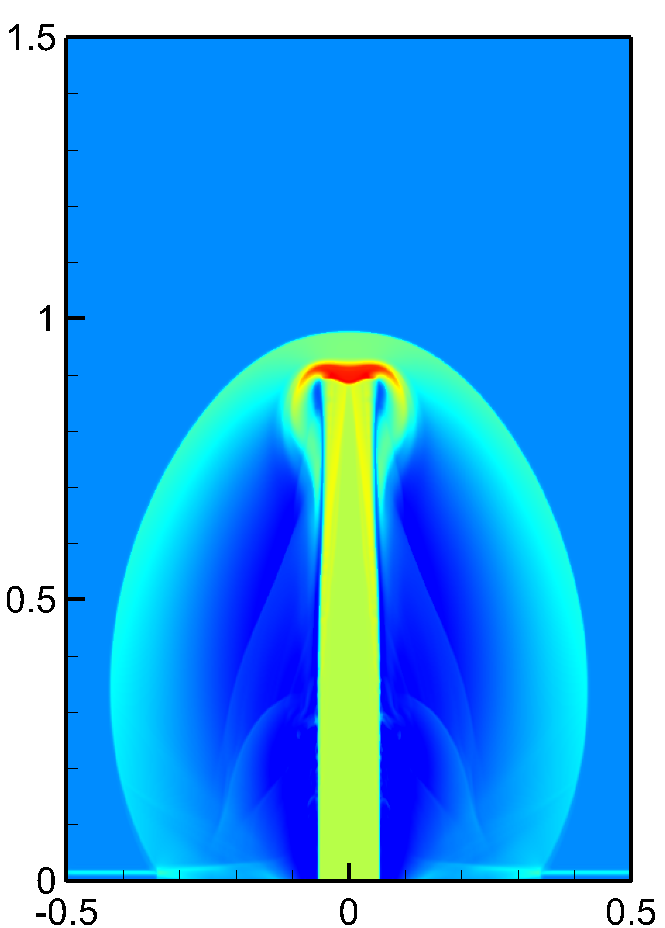}
    }
    {
		\includegraphics[width=0.31\linewidth]{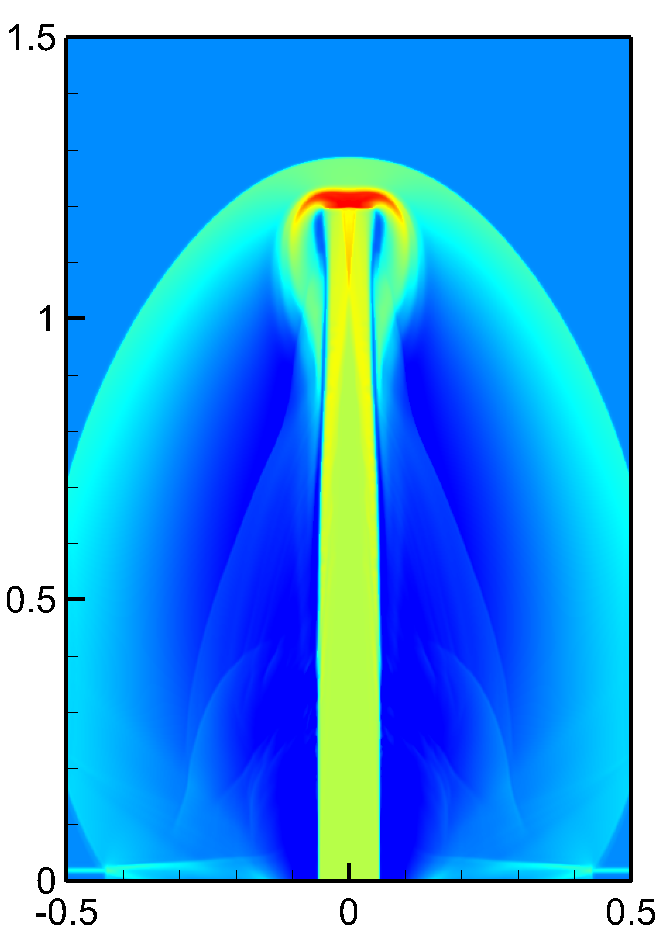}
    }

    {
    	\includegraphics[width=0.31\linewidth]{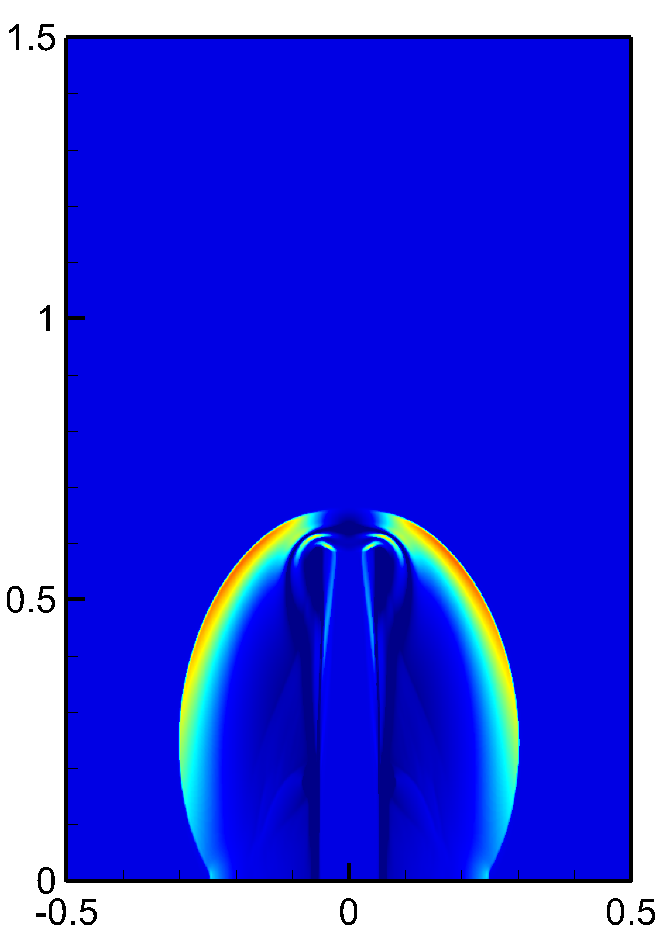}
    }
    {
    	\includegraphics[width=0.31\linewidth]{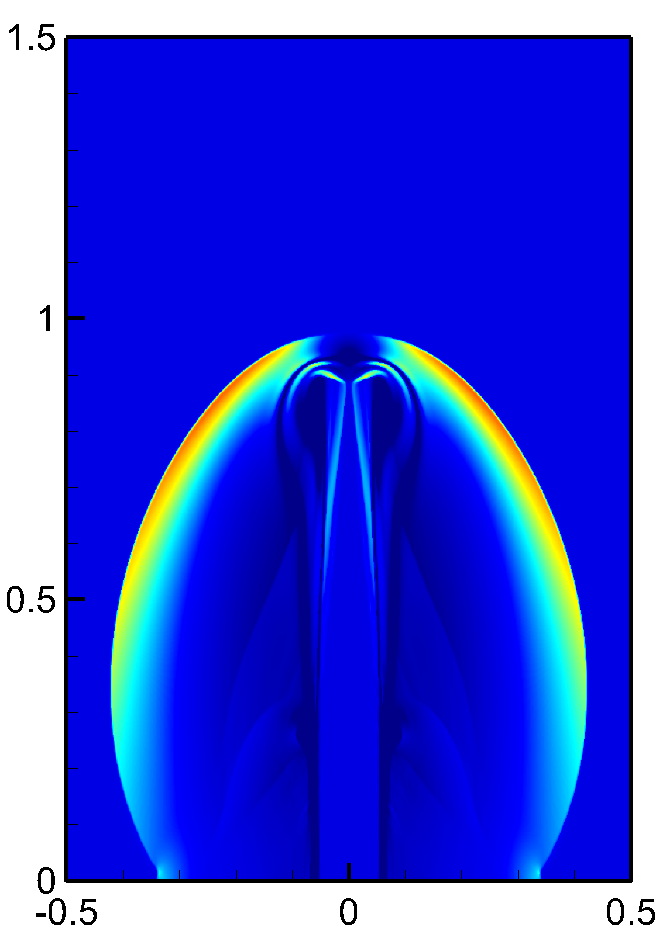}
    }
    {
    	\includegraphics[width=0.31\linewidth]{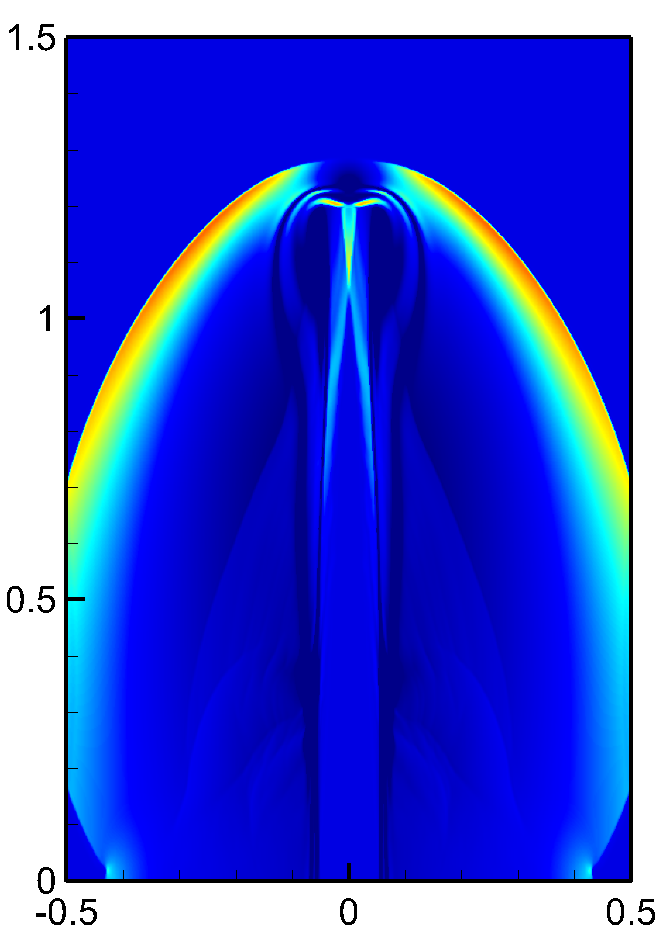}
    }
	\caption{\small Example \ref{Exp:Jets}: Density logarithm (top) and magnetic pressure  (bottom) for the Mach 800 jet with $B_0 = \sqrt{2000}$ at $t = 0.001$, 0.0015, and 0.002 (from left to right).}
	\label{Fig:Jet2}
\end{figure}

\begin{figure}[!thb]
    \vspace*{-0.2cm}
	\centering 
    {
		\includegraphics[width=0.31\linewidth]{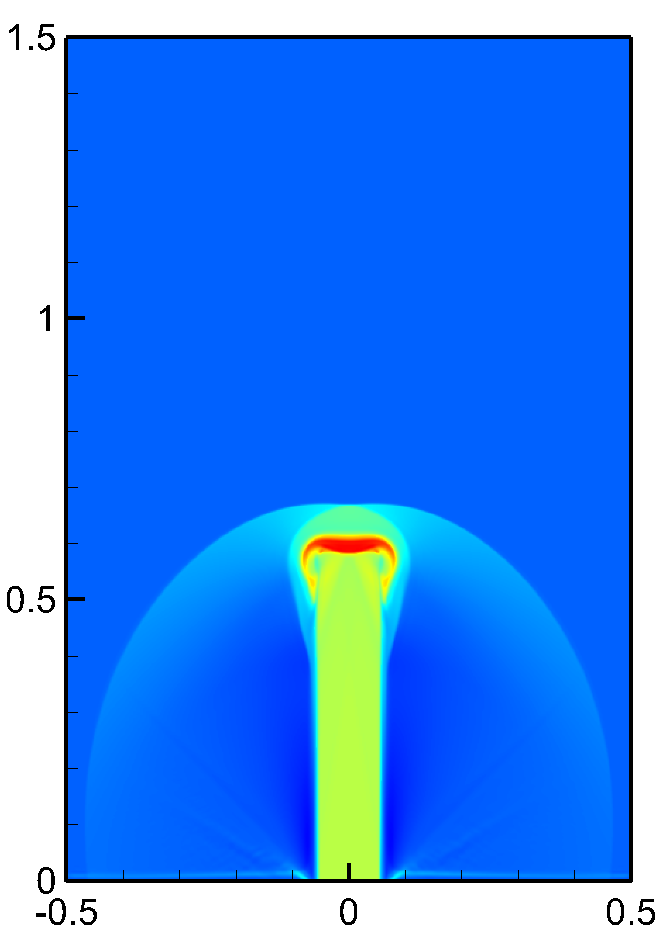}
    }
    {
		\includegraphics[width=0.31\linewidth]{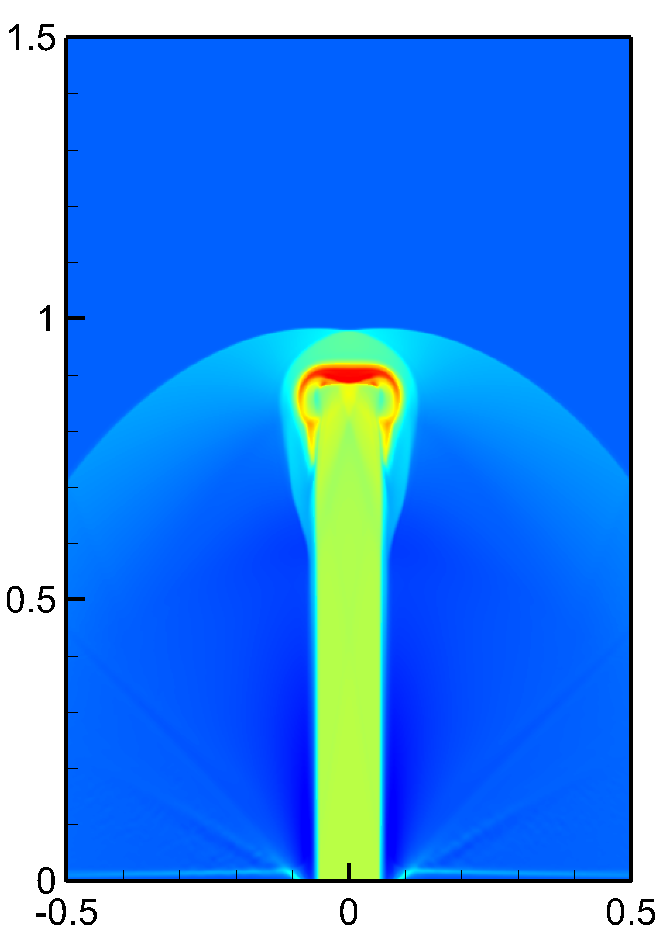}
    }
    {
		\includegraphics[width=0.31\linewidth]{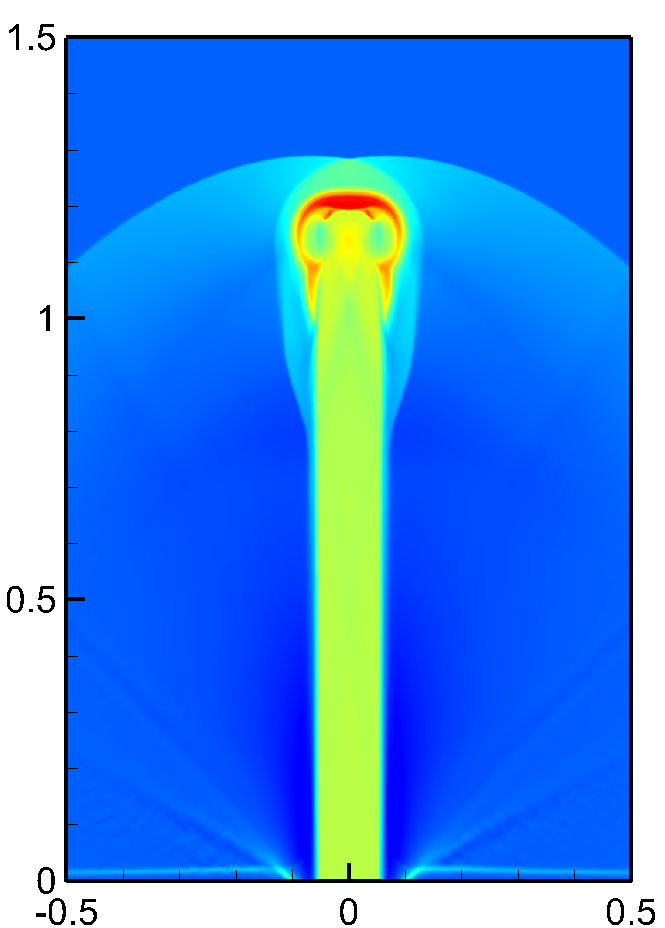}
    }

    {
		\includegraphics[width=0.31\linewidth]{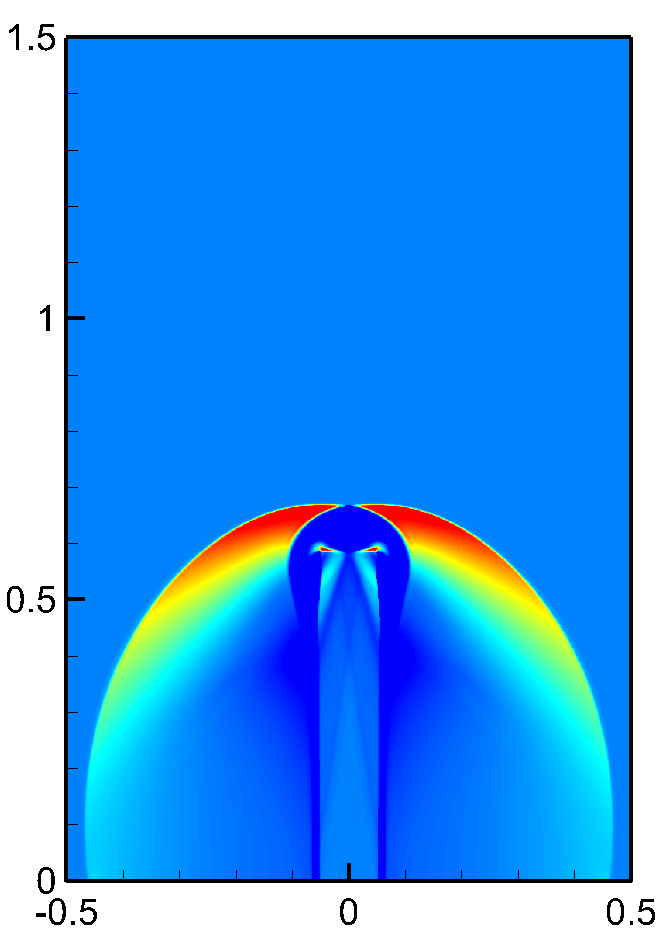}
    }
    {
		\includegraphics[width=0.31\linewidth]{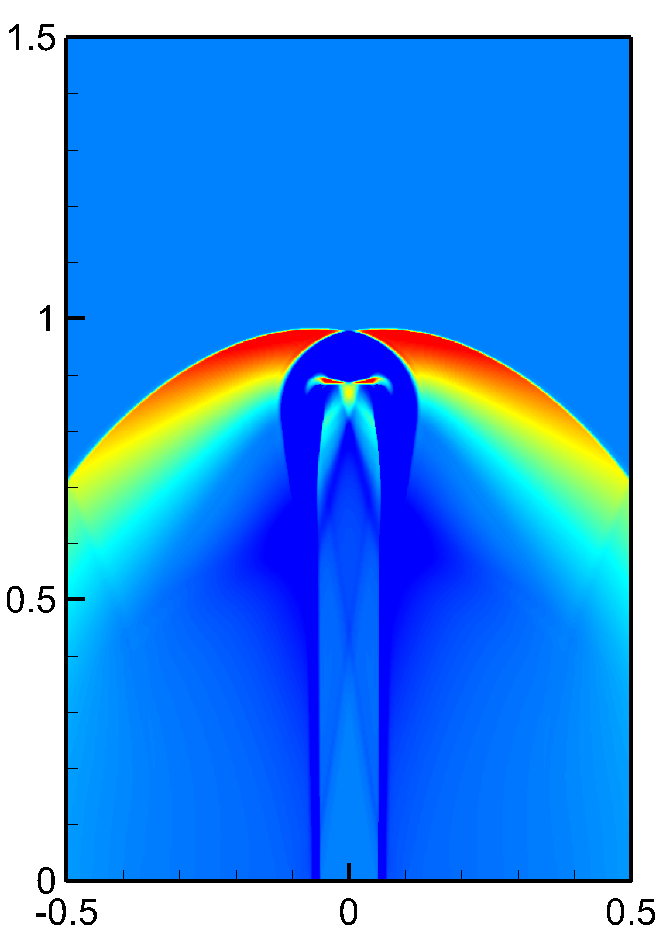}
    }
    {
		\includegraphics[width=0.31\linewidth]{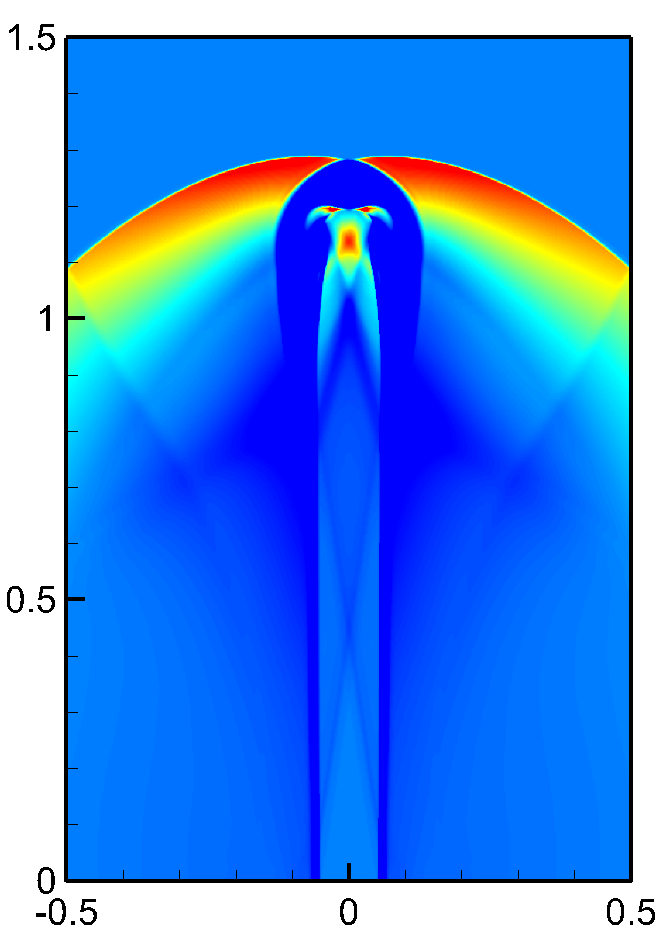}
    }
	\caption{\small Example \ref{Exp:Jets}: Density logarithm (top) and magnetic pressure (bottom) for the Mach 800 jet with $B_0 = \sqrt{20,000}$ at $t = 0.001$, 0.0015, and 0.002 (from left to right).}
	\label{Fig:Jet3}
\end{figure}

\begin{figure}[!thb] 
    \vspace*{-0.2cm}
	\centering 
    {
		\includegraphics[width=0.31\linewidth]{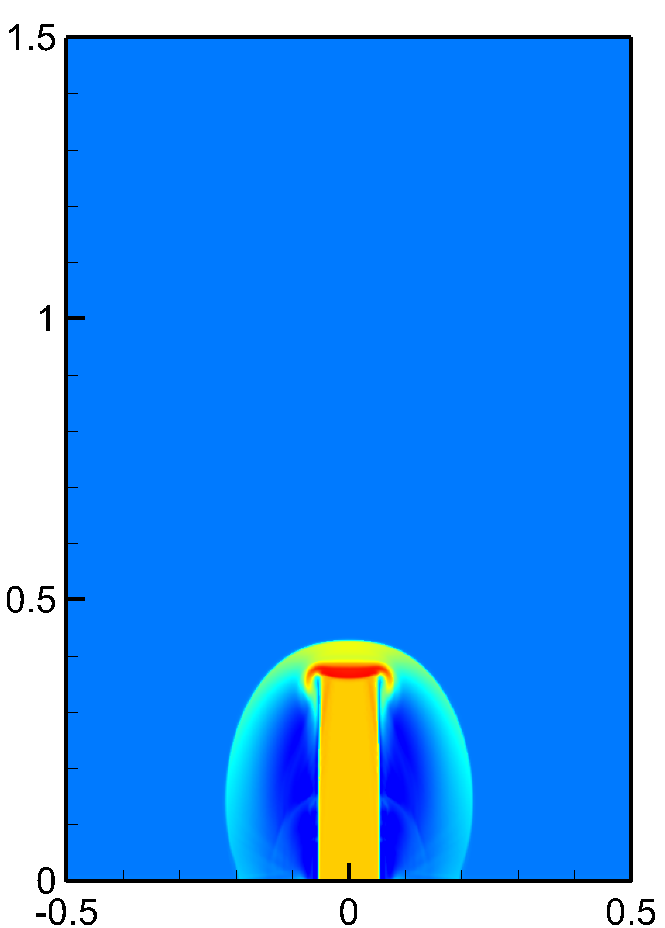}
    }
    {
		\includegraphics[width=0.31\linewidth]{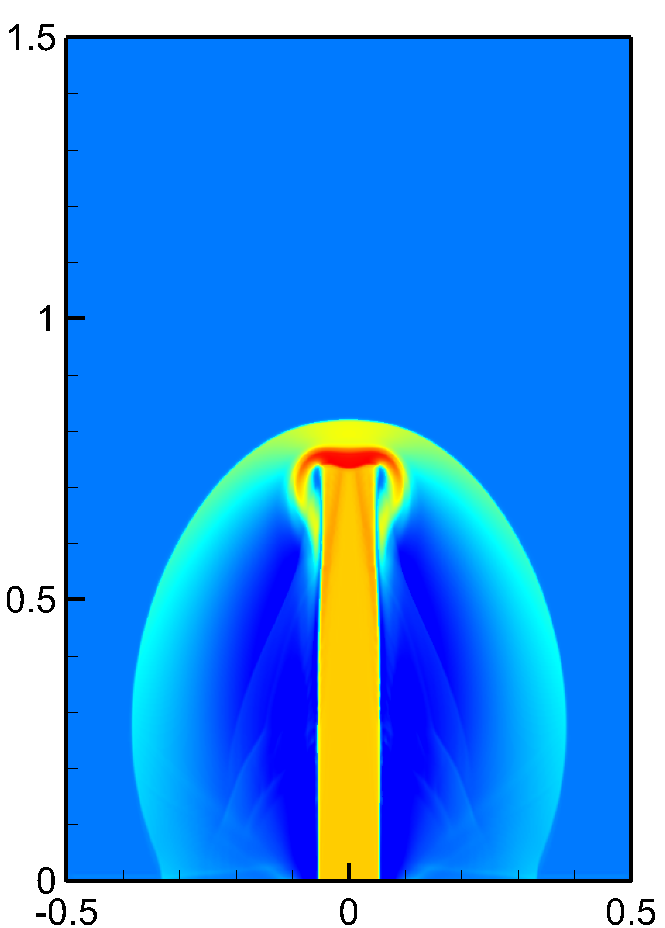}
    }
    {
		\includegraphics[width=0.31\linewidth]{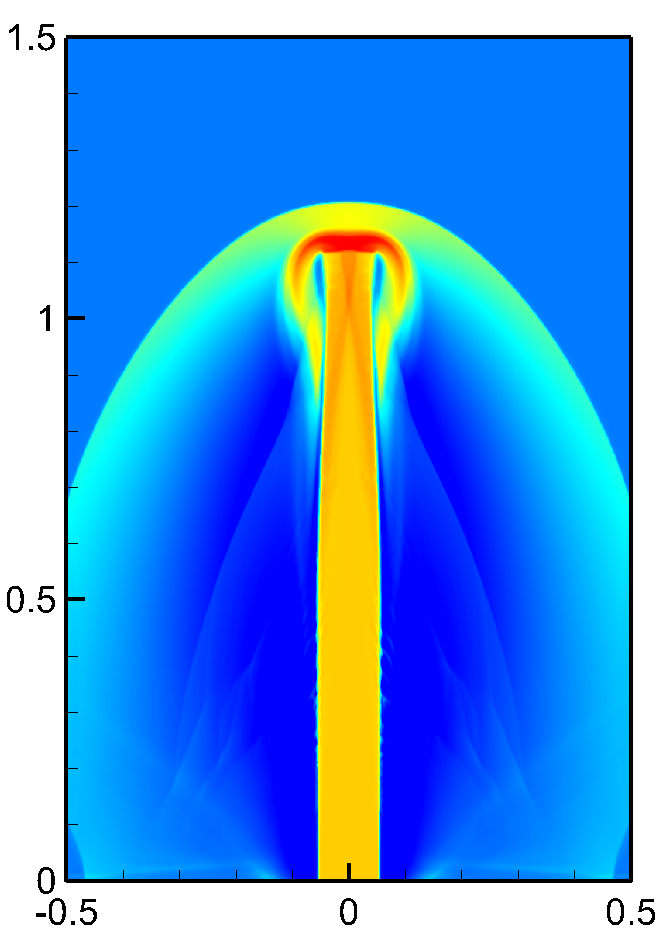}
    }
	\caption{\small Example \ref{Exp:Jets}: Density logarithm for the Mach
2000 jet with $B_0 = \sqrt{20,000}$ at $t = 0.00025$, 0.0005, and 0.00075 (from left to right).}
	\label{Fig:Jet4}
\end{figure}

\begin{figure}[!thb] 
    \vspace*{-0.2cm}
	\centering  
    {
		\includegraphics[width=0.31\linewidth]{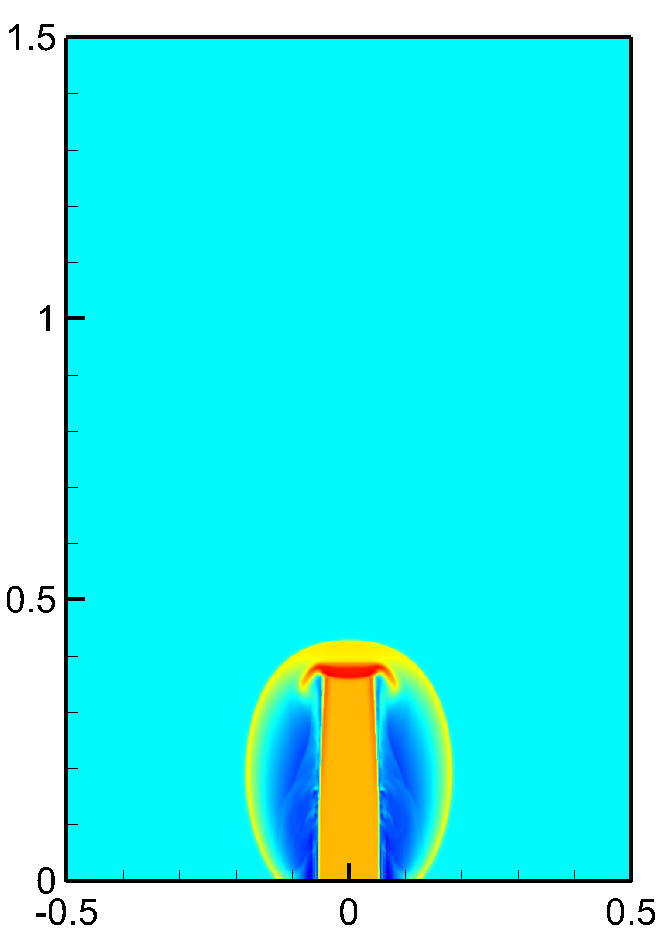}
    }
    {
		\includegraphics[width=0.31\linewidth]{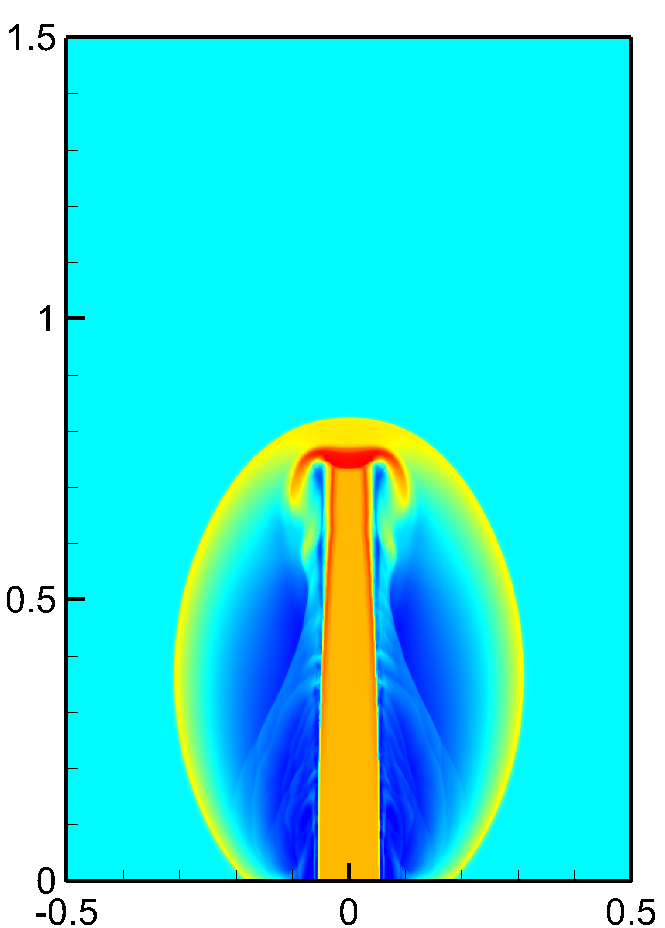}
    }
    {
		\includegraphics[width=0.31\linewidth]{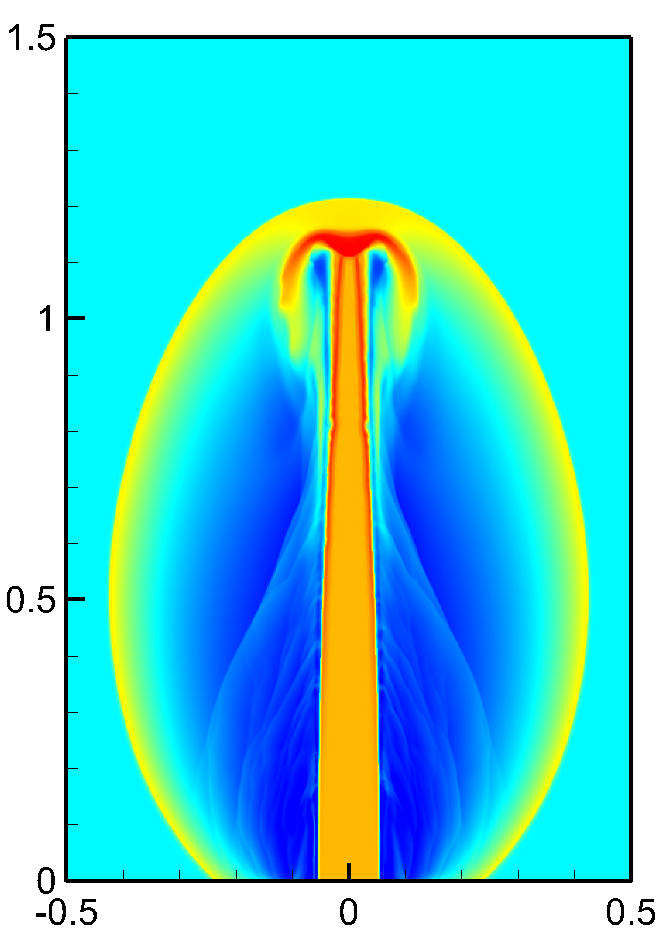}
    }
	\caption{\small Example \ref{Exp:Jets}: Density logarithm for the {\bf Mach
1,000,000 jet} with $B_0 = \sqrt{20,000}$ at $t = 0.0000005$, 0.000001, and 0.0000015 (from left to right).}
	\label{Fig:Jet7}
\end{figure}

\section{Conclusions}\label{sec:con}

This paper has presented a key advancement in the numerical simulation of compressible ideal MHD by developing the PosDiv-CDG method, which provably preserves both the positivity of density and pressure and the global divergence-free (DF) condition at arbitrarily high order in multiple dimensions. This method resolves the intrinsic incompatibility between standard positivity-preserving (PP) limiters and global DF enforcement through a combination of innovations: a new limiting strategy, a dissipation mechanism guided by convex decomposition, and an auxiliary magnetic field evolution. A theoretical  proof has been provided for preserving the positivity of updated cell averages, established within the geometric quasi-linearization (GQL) framework. In addition, a compact, non-intrusive, entropy-induced convex-oscillation-suppressing (COS) method is developed to effectively eliminate spurious oscillations near shocks, while preserving high-order accuracy and the global DF condition. Several numerical experiments, including extremely challenging MHD jet problems with low plasma-beta and Mach numbers up to $10^6$, have confirmed the robustness, accuracy, and structural fidelity of the proposed method. These results establish PosDiv-CDG as a new approach for high-fidelity, reliable simulation of complex, shock-dominated, and highly magnetized MHD flows.

\newpage 

\bibliographystyle{siamplain}
\bibliography{references_article}

\end{document}